\documentclass{amsart}
\usepackage{amssymb,amsmath,amsrefs}
\usepackage[colorlinks=true, linkcolor=red, citecolor=blue, urlcolor=blue]{hyperref}
\usepackage[dvips]{graphicx}
\usepackage{subcaption}
\usepackage{xcolor}
\definecolor{OrangeRed}{rgb}{1,0.27,0}
\usepackage{enumerate}
\usepackage{soul}

\definecolor{verdemusgo}{rgb}{0.33,0.42,0.18}

\usepackage{tikz}
\usetikzlibrary{arrows.meta, decorations.markings, patterns, positioning}

\definecolor{boxblue}{HTML}{D0E0E3}
\definecolor{boxyellow}{HTML}{FFF2CC}
\definecolor{arrowblue}{HTML}{4A86E8}
\definecolor{dashblue}{HTML}{1A6E8C}
\tikzset{
	mid arrow/.style={
		decoration={
			markings,
			mark=at position #1 with {\arrow{Stealth[scale=1.2]}}
		},
		postaction={decorate}
	},
	mid arrow/.default=0.5
}

\usepackage{tabularx}
\usepackage{soul,color}
\usepackage{marginnote}
\usepackage[font=scriptsize]{caption}
\usepackage{mwe}
\theoremstyle{plain}

\newtheorem{mainthm}{Theorem}

\newtheorem{mainclly}[mainthm]{Corollary}
\newtheorem*{conj*}{Conjecture}
\newtheorem*{cor*}{Corollary}
\newtheorem{theorem}{Theorem}[section]

\newtheorem{proposition}[theorem]{Proposition}
\newtheorem{corollary}[theorem]{Corollary}
\newtheorem{lemma}[theorem]{Lemma}

\newtheorem{question}{Question}

\newtheorem{claim}{Claim}

\theoremstyle{definition}
\newtheorem*{def*}{Definition}
\newtheorem{remark}[theorem]{Remark}
\newtheorem{rmk}[theorem]{Remark}
\newtheorem{example}[theorem]{Example}

\newtheorem{definition}[theorem]{Definition}

\newcommand{\SP}{{\mathcal P}}

\renewcommand{\epsilon}{\varepsilon}

\newcommand{\ka} {\kappa}

\DeclareMathOperator{\Sing}{Sing}

\newcommand{\Z}{\mathbb{Z}}

\newcommand{\R}{\mathbb{R}}
\newcommand{\eps}{\varepsilon}

\newcommand{\interior}{\operatorname{int}}

\newcommand{\sing}{\operatorname{Sing}}

\newcommand{\diam}{\operatorname{diam}}		
\makeatletter
\newcommand{\tpitchfork}{%
  \vbox{
    \baselineskip\z@skip
    \lineskip-.52ex
    \lineskiplimit\maxdimen
    \m@th
    \ialign{##\crcr\hidewidth\smash{$-$}\hidewidth\crcr$\pitchfork$\crcr}
  }%
}
\makeatother

\title[Shadowing in the presence of singularities]{ Shadowing in the presence of singularities: oriented versus standard shadowing, entropy and the structure of recurrent sets}

\author{Sakshi Jain}
\address[Jain]{School of Mathematics and Physics\\
   University of Queensland\\
	St Lucia, QLD 4072\\
	Australia}
\email{sakshi.jain@uq.edu.au}

\author{Piotr Oprocha}
\address[Oprocha]{University of Ostrava,
	IRAFM,
	30. dubna 22, 70103 Ostrava,
	Czech Republic}
\email{piotr.oprocha@osu.cz}

\author{Elias Rego}
\address[Rego]{Faculty of Applied Mathematics, AGH University of Krakow, Krakow, Poland.}
\email{rego@agh.edu.pl}



\begin{document}

\begin{abstract}

We study two shadowing properties for flows that differ in the allowed
reparametrizations of time: oriented shadowing permits arbitrary
increasing reparametrizations, whereas standard shadowing requires
their distortion to be uniformly close to one. We prove that these
notions are distinct already for $C^\infty$ flows on every closed
oriented surface. Moreover, such examples are $C^0$-dense among
$C^1$ flows with a singularity and consequently, on closed oriented
surfaces with non-zero Euler characteristic, they are dense among all
$C^1$ flows. We then relate local standard shadowing to recurrence and
entropy. A non-trivial chain-transitive set with local standard
shadowing forces positive topological entropy unless it is an
irreducible almost heteroclinic set. Consequently, for a zero-entropy
flow with standard shadowing, every non-trivial chain-recurrent class
has this form, and every non-singular one is minimal. For surface
flows, we further prove that oriented shadowing together with finitely
many singularities forces every chain-recurrent class to be minimal.
\end{abstract}

\maketitle

\section{Introduction}

The shadowing property asks whether a trajectory assembled with small
errors can be followed, uniformly in time, by an exact orbit. It plays a
central role in the theoretical framework of dynamical systems. As a basic
feature of hyperbolic dynamics, shadowing was used in~\cite{B71} to
study the structure of Axiom~A diffeomorphisms. Moreover, it serves as a
crucial tool for proving the stability of hyperbolic sets and,
conversely, lies at the core of structural stability: as shown by
Pilyugin~\cite{Py}, every structurally stable system must exhibit the
shadowing property. The deep relationship between shadowing and symbolic
dynamics was established by Good and Meddaugh~\cite{GM20}, and its
versatility is further illustrated by appearances in other
contexts~\cite{Bm21,ACT26,CS99}. Beyond its theoretical
importance, shadowing is especially valuable in applied sciences. Since
computer simulations inherently involve rounding errors, it provides a
rigorous framework for interpreting approximate trajectories produced by
numerical procedures; its relevance to numerical analysis is well
documented~\cites{HYG87,TN25,CW21,T09}.

For flows, the notion of shadowing becomes considerably more nuanced than for its discrete-time counterpart. An exact orbit cannot in general be required to follow a
pseudo-orbit with the same time parametrization. Instead, one allows
a reparametrization of time along the tracing orbit.
Let $\mathrm{Rep}$ denote the
set of increasing homeomorphisms $h\colon\mathbb R\to\mathbb R$ with
$h(0)=0$, and let
\[
\mathrm{Rep}(\varepsilon)=
\left\{h\in\mathrm{Rep}:
\left|\frac{h(t)-h(s)}{t-s}-1\right|\leq\varepsilon
\text{ for all }t\neq s\right\}.
\]
In oriented shadowing the time reparametrization may be any element of
$\mathrm{Rep}$, whereas standard shadowing requires it to belong to
$\mathrm{Rep}(\varepsilon)$. Thus standard shadowing controls
both the geometry of the tracing orbit and the distortion of elapsed
time while oriented shadowing retains only the former control (precise
definitions are given in Section~\ref{sec:Prelim}).

The distinction becomes significant in the presence of singularities.
Near a singularity, geometrically close orbit segments may spend very
different amounts of time in the same region. This phenomenon already
appears in Komuro's proof that geometric Lorenz attractors do not have
the shadowing property \cite{K85}. Subsequent results exclude shadowing
for singular-hyperbolic sets containing a singularity and for
codimension-one sectional-Anosov flows \cites{WW,ALS}, as well as for
certain recurrent configurations with an attached hyperbolic
singularity of stable or unstable index one \cite{ALRS}. These
obstructions are not universal: Murakami recently constructed on
$\mathbb{S}^4$ a non-trivial shadowable chain-recurrent set with an attached
hyperbolic singularity \cite{Mu2}. Together, these results show that
shadowing in the singular setting depends on more than the mere
coexistence of singularities and non-trivial recurrence.

Komuro introduced the two forms of shadowing above and proved that
they coincide for flows without singularities \cite{Ko}. He also
asked whether the two notions are distinct in general. Tikhomirov
answered this question by constructing a $C^1$ vector field on a
four-dimensional closed manifold with oriented shadowing but without
standard shadowing \cite{Ty}. More recently, Murakami proved the
equivalence for topological flows with finitely many singularities,
each Lyapunov stable or backward Lyapunov stable \cite{Mu23}. On a
closed surface he obtained the equivalence when the nonwandering set
consists of finitely many critical elements \cite{Mu1}. The known
counterexample was therefore genuinely higher-dimensional, while the
available surface results supported the possibility that the two
notions might coincide for arbitrary surface flows. Our first result
shows that they do not, demonstrating that, even in the topological restriction provided by the low dimensional scenario, the equivalence cannot be achieved.

\begin{mainthm}\label{thm:denseflows}
Let $M$ be a compact, boundaryless and oriented Riemannian surface.
 Then $M$ supports a $C^\infty$ flow with the oriented
shadowing property but without the standard shadowing property.
Moreover, the set of all such $C^\infty$ flows is $C^0$-dense in the
space of $C^1$ flows on $M$ which have a singularity, where the
$C^0$ topology is induced by their continuous generating vector
fields.
\end{mainthm}

This result shows that the oriented and standard shadowing properties are genuinely distinct, even on compact surfaces. In particular, the assumption of finitely many singularities cannot be omitted. Furthermore, the examples are abundant and may be chosen to have the highest regularity. 
The construction underlying Theorem~\ref{thm:denseflows} starts from
a simple Morse--Smale surface flow and changes only its speed. A
non-negative $C^\infty$ multiplier introduces two carefully arranged
sequences of singularities accumulating inside a pair of flow boxes.
The multiplier is chosen so that nearby regular orbits retain the same
geometric itinerary but acquire crossing times separated by
arbitrarily large factors. An arbitrary increasing reparametrization
can absorb these differences and trace the itinerary. A
reparametrization in $\mathrm{Rep}(\varepsilon)$ cannot do so.
The main point of the construction is to obtain this separation while
preserving smoothness at the accumulation points and retaining enough
control of all pseudo-orbits to prove oriented shadowing.

The density assertion shows that this phenomenon is not confined to
a single specially chosen flow. It can be inserted by an arbitrarily
small $C^0$ perturbation near any zero of a $C^1$ vector field. There
is also an immediate geometric consequence. If $\chi(M)\neq0$, then
the Poincar\'e--Hopf theorem forces every $C^1$ vector field on $M$ to
have a zero. Hence the $C^\infty$ examples of
Theorem~\ref{thm:denseflows} are $C^0$-dense among all $C^1$ flows on
every closed oriented surface with non-zero Euler characteristic. In
particular, this holds on $\mathbb{S}^2$ and on every closed oriented surface
of genus at least two. We record this consequence as
Corollary~\ref{cor:dense-all-surface} below.

The second theme of the paper is the relation between shadowing,
recurrence and entropy. In discrete time, shadowing places strong
restrictions on zero-entropy recurrence. For a continuous map with
shadowing, the presence of either a non-minimal recurrent point or a
sensitive minimal subsystem forces positive topological entropy; at
zero entropy, the restriction to the nonwandering set is an
equicontinuous homeomorphism \cite{Moo11}. In particular, a
nonwandering system with shadowing is either equicontinuous or has
positive entropy \cite{LO13}. The mechanism behind these results is
that shadowing converts sufficiently distinct recurrent patterns into
symbolic dynamics; see also \cite{MO13}.

There is no reason to expect the equicontinuous alternative to pass
unchanged to flows. Equicontinuity depends on the parametrization of
orbits, and a change of speed may alter it while leaving the oriented
orbit foliation unchanged. This already occurs within the shadowing
class. Indeed, let $f$ be an aperiodic odometer. It is equicontinuous
and, since its phase space is totally disconnected, it has the
shadowing property \cite[Proposition~2.4]{LO13}. Consequently, its
constant-roof suspension (see Definition~\ref{Def-Susp}) is an equicontinuous regular flow with
oriented shadowing. More generally, suspension and
orientation-preserving changes of speed preserve the latter property
\cite{Thomas82}. Now replace the constant roof by a positive continuous
function which is not cohomologous to a constant. The resulting flow
is no longer equicontinuous \cite[Corollary~4.11]{KS87}, but it still
has oriented shadowing and hence, being regular, also standard
shadowing \cite{Ko}. Its entropy remains zero. Thus even among regular
flows with standard shadowing, zero entropy does not force
equicontinuity. More drastic changes of behavior are also possible:
Liouvillian linear flows on tori admit smooth weak-mixing time changes
\cite{Fayad02}. We therefore seek restrictions on chain-recurrent
classes rather than a literal equicontinuity--entropy alternative.

Both entropy production and the shadowing mechanism used below are
local. For a $\delta$-$T$-pseudo-orbit
$P=(x_i,t_i)_{i\in\mathbb Z}$, put
\[
\mathcal P(P)=\bigcup_{i\in\mathbb Z}\phi_{[0,t_i]}(x_i).
\]

\begin{definition}
A compact set $\Lambda$ has the \emph{local oriented shadowing
property} if, for every $\varepsilon>0$ and $T>0$, there are a
neighborhood $U$ of $\Lambda$ and $\delta>0$ such that every
$\delta$-$T$-pseudo-orbit $P$ satisfying $\mathcal P(P)\subset U$ is
$\varepsilon$-traced by some point of $X$. Requiring instead that
every such pseudo-orbit be strongly $\varepsilon$-traced defines the
\emph{local standard shadowing property}.
\end{definition}

The two forms of local shadowing are linked away from singularities.
We prove in Section~\ref{sec: regularity} that oriented shadowing can
be upgraded to standard shadowing for pseudo-orbits staying a fixed
positive distance from the singular set. Consequently, every compact
invariant set without singularities in a flow with oriented shadowing
has the local standard shadowing property
(Corollary~\ref{cor:locshad}).

To describe the remaining obstruction, call an orbit
$\mathcal O(x)$ a \emph{homoclinic loop} if, for some
$\sigma\in\operatorname{Sing}(\phi)$,
$\overline{\mathcal O(x)}=\mathcal O(x)\cup\{\sigma\}$. It is a
\emph{heteroclinic connection} if it is not a homoclinic loop and
$\alpha(x)\cup\omega(x)\subset\operatorname{Sing}(\phi)$. A compact
invariant set is \emph{non-trivial} if it is not a single orbit and
is not contained in $\operatorname{Sing}(\phi)$.

\begin{definition}
A compact, invariant and chain-transitive set $\Lambda$ is an
\emph{almost heteroclinic set} if every regular point in $\Lambda$
whose orbit is not dense in $\Lambda$ lies on a heteroclinic
connection. It is \emph{irreducible} if it contains no proper
non-trivial almost heteroclinic set.
\end{definition}

For a non-singular almost heteroclinic set, every orbit must be dense,
so the set is minimal. In the presence of singularities, however, the
definition also accommodates homoclinic loops and more complicated
singularly minimal dynamics. Our second main result identifies the
failure of this structure as a source of positive entropy.

\begin{mainthm}\label{thm:positive entropy}
Let $\phi$ be a continuous flow and $\Lambda\subset X$ be a compact, invariant and non-trivial chain-transitive subset. If $\Lambda$ is not an irreducible almost heteroclinic set and has the local standard shadowing property, then $\phi$ has positive topological entropy.
\end{mainthm}

Theorem~\ref{thm:positive entropy} requires standard shadowing only
near $\Lambda$. Its proof turns the failure of irreducibility into
two long almost-returning orbit blocks: one remains close to a proper
chain-transitive subset, while the other makes a uniformly separated
excursion. Arbitrary binary concatenations of these blocks are local
pseudo-orbits. The bounded time distortion in standard shadowing
makes the two types of block distinguishable after tracing and yields
exponentially many separated orbits.

The exceptional alternative should not be read as synonymous with
zero entropy. Singularly minimal sets---those whose proper compact
invariant subsets lie in the singular set---are irreducible almost
heteroclinic sets, and singular suspensions may have zero, positive or
infinite entropy \cites{SYZ09,SZ11}. Thus the theorem gives a
one-sided entropy criterion rather than claiming a complete
dichotomy. Applied to chain-recurrent classes, it yields the following
structural consequence.

\begin{mainthm}\label{thm:standard}
Let $\phi$ be a continuous flow with the standard shadowing property. If $\phi$ has zero topological entropy, then every non-trivial chain-recurrent class of $\phi$ is an irreducible almost heteroclinic set. Moreover, if a non-trivial class is non-singular, then it is minimal. In particular, if $\phi$ is a regular flow, its chain-recurrent set is a disjoint union of minimal sets.
\end{mainthm}

Every continuous flow on a closed surface has zero topological
entropy \cite{LSY}. Theorem~\ref{thm:standard} therefore gives an
unconditional restriction for surface flows with standard
shadowing.

\begin{mainclly}\label{cor:standard surface}
The chain-recurrent set of a surface flow with the standard shadowing property is a disjoint union of almost heteroclinic sets. Moreover, if the flow is regular, its chain-recurrent set is a disjoint union of minimal sets.
\end{mainclly}

Theorem~\ref{thm:denseflows} shows that such an extension cannot be
obtained merely by identifying oriented with standard shadowing.
Nevertheless, the local upgrade away from singularities, combined
with a surface argument at the singular classes, gives a stronger
conclusion when the singular set is finite.

\begin{mainthm}\label{thm:surfacecharac}
Let $\phi$ be a continuous flow on a compact, boundaryless surface $M$ with the oriented shadowing property. If $\operatorname{Sing}(\phi)$ is finite, then every chain-recurrent class of $\phi$ is minimal.
\end{mainthm}

These structural conclusions are necessary restrictions, not
characterizations of shadowing: a minimal irrational linear flow on
$\mathbb T^2$ is the simplest counterexample to the corresponding
converses.

Section~\ref{sec:Prelim} fixes the notation and preliminary facts.
Section~\ref{sec:examplesufaces} constructs the surface examples and
proves Theorem~\ref{thm:denseflows}. Section~\ref{sec: regularity}
compares oriented and standard shadowing away from the singular set.
Section~\ref{sec:classification} develops the recurrence--entropy
argument and proves Theorems~\ref{thm:positive entropy},
\ref{thm:standard} and~\ref{thm:surfacecharac}.

\section{Preliminaries}\label{sec:Prelim}

In this section, we introduce the preliminary concepts and classical results that will be instrumental in proving our main results. We refer the reader to the treatises \cites{KH95,AL85,Bruin,Kurka} for comprehensive expositions of the material presented here.

Throughout this paper, $(X,d)$ denotes a compact metric space, and $\phi$ denotes a continuous flow on $X$. The orbit of a point $x\in X$ is the set
\[
O(x)=\{\phi_t(x):t\in\mathbb{R}\}.
\]
Given $A\subset\mathbb{R}$ and $S\subset X$, we also write
\[
\phi_A(x)=\{\phi_t(x):t\in A\}
\qquad\text{and}\qquad
\phi_A(S)=\{\phi_t(x):x\in S,\ t\in A\}.
\]

A point $x\in X$ is called a \emph{singularity} of $\phi$ if it is fixed by the flow, that is, $\phi_t(x)=x$, for every  $t\in\mathbb{R}$.
We denote by $Sing(\phi)$ the set of singularities of $\phi$. A point that is not a singularity is called a \emph{regular point}. A regular point is said to be \emph{periodic} if there exists $t>0$ such that $\phi_t(x)=x$. In this case, the period of $x$ is defined by
\[
\pi(x)=\inf\{t>0:\phi_t(x)=x\}.
\]
Let $Per(\phi)$ denote the set of periodic points of $\phi$. The \emph{critical set} of $\phi$ is then defined as
\(
Crit(\phi)=Per(\phi)\cup Sing(\phi).
\)

Fix $\delta,T>0$, and let $A\subset\mathbb{Z}$ be an interval of integers. A sequence
\(
(x_i,t_i)_{i\in A}
\)
is called a \emph{$\delta$-$T$-pseudo-orbit} if $t_i\geq T$ and
\[
d(\phi_{t_i}(x_i),x_{i+1})\leq\delta
\]
for every pair of consecutive indices $i,i+1\in A$. When $A=[a,b]\cap\mathbb{Z}$ is finite, we say that
\(
(x_i,t_i)_{i=a}^{b}
\) is a finite $\delta$-$T$-pseudo-orbit from $x_a$ to $x_b$. If $A=\mathbb{Z}$, we simply write
\((x_i,t_i)\bigr)_{i\in\mathbb{Z}}.
\)  We will also refer to finite $\eps$-$T$-pseudo-orbits as $\eps$-$T$-chains. By a straightforward modification of the sequence, we may always assume that every $\delta$-$T$-pseudo-orbit satisfies
\(
t_i\in[T,2T)
\)
for all $i$.

For a bi-infinite pseudo-orbit $(x_i,t_i)_{i\in\mathbb Z}$, set
\begin{equation}\label{def:S_i}
S_i=
\begin{cases}
\displaystyle\sum_{n=0}^{i-1}t_n, & i>0,\\[2mm]
0, & i=0,\\[1mm]
\displaystyle-\sum_{n=1}^{|i|}t_{-n}, & i<0.
\end{cases}
\end{equation}
We consider the following two classes of reparametrizations:
\[
\mathrm{Rep}
=
\{h\colon\mathbb R\to\mathbb R:
h\text{ is an increasing homeomorphism and }h(0)=0\},
\]
and, for $\varepsilon>0$,
\[
\mathrm{Rep}(\varepsilon)
=
\left\{h\in\mathrm{Rep}:
\left|\frac{h(t)-h(s)}{t-s}-1\right|\leq\varepsilon
\text{ for all }t\neq s\right\}.
\]
We say that $(x_i,t_i)_{i\in\mathbb Z}$ is
$\varepsilon$-traced by $x\in X$ if there exists
$h\in\mathrm{Rep}$ such that
\[
d\bigl(\phi_{t-S_i}(x_i),\phi_{h(t)}(x)\bigr)
\leq\varepsilon
\]
for every $i\in\mathbb Z$ and every
$t\in[S_i,S_{i+1}]$. If the same conclusion holds with
$h\in\mathrm{Rep}(\varepsilon)$, then the pseudo-orbit is said to be
\emph{strongly $\varepsilon$-traced} by $x$.

\begin{definition}\label{def:oriensh}
Let $\Lambda\subset X$ be compact and $\phi$-invariant. We say that
$\Lambda$ has the \emph{oriented shadowing property} if, for every
$\varepsilon>0$, there exists $\delta>0$ such that every
$\delta$-$1$-pseudo-orbit contained in $\Lambda$ is
$\varepsilon$-traced by some point of $X$.
\end{definition}

\begin{definition}\label{def:standsh}
Let $\Lambda\subset X$ be compact and $\phi$-invariant. We say that
$\Lambda$ has the \emph{standard shadowing property} if, for every
$\varepsilon>0$, there exists $\delta>0$ such that every
$\delta$-$1$-pseudo-orbit contained in $\Lambda$ is strongly
$\varepsilon$-traced by some point of $X$.
\end{definition}

\begin{remark}
The pseudo-orbit and tracing conventions above agree with those of
\cite{Ko}. They differ in form from the conventions used in
\cite{Ty}, but define the same oriented and standard shadowing
properties. Moreover, replacing $1$ by any fixed $T>0$ in
Definitions~\ref{def:oriensh} and~\ref{def:standsh} does not change
either property; see \cite{Ko}. When $\Lambda=X$, we simply say that
the flow has the corresponding shadowing property.
\end{remark}

Given $x,y\in X$, we say that $y$ is \emph{chain-attainable} from $x$, and write $x\sim y$, if for every $\delta>0$ there exist $N\in\mathbb{N}$ and a $\delta$-$1$-pseudo-orbit
\(
(x_i,t_i)_{i=0}^{N}
\)
from $x_0=x$ to $x_N=y$. We say that $x$ and $y$ are \emph{chain-related} if $x\sim y$ and $y\sim x$. The \emph{chain-recurrent set} of $\phi$ is defined by
\(
CR(\phi)=\{x\in X:x\sim x\}.
\)
It is well known that $\sim$ is an equivalence relation on $CR(\phi)$. Therefore, $CR(\phi)$ decomposes into equivalence classes, called the \emph{chain-recurrent classes}. For each $x\in CR(\phi)$, we denote by $H(x)$ the chain-recurrent class containing $x$, namely,
\[
H(x)=\{y\in CR(\phi):x\sim y\}.
\]
Chain-recurrent classes are compact and invariant. Moreover, every singularity belongs to $CR(\phi)$. A chain-recurrent class $H$ is called \emph{singular} if $H=H(\sigma)$ for some $\sigma\in Sing(\phi)$. A chain-recurrent class is called \emph{trivial} if it consists of a single orbit. A compact invariant set $K\subset X$ is said to be \emph{chain-transitive} if, for every pair of points $x,y\in K$, the point $y$ is chain-attainable from $x$ by means of $\delta$-$1$-pseudo-orbits entirely contained in $K$. It is well know that chain-recurrent classess are chain-transitive.

We now recall the definition of topological entropy. Let $f:X\to X$ be a homeomorphism, and let $K\subset X$ be a compact subset, not necessarily invariant. Given $n\in\mathbb{N}$ and $\varepsilon>0$, a subset $Q\subset K$ is called an \emph{$(n,\varepsilon)$-separated set} if, for every pair of distinct points $x,y\in Q$, there exists $0\leq j\leq n$ such that
\[
d(f^j(x),f^j(y))>\varepsilon.
\]
Let $S(n,\varepsilon,K)$ denote the maximum cardinality of an $(n,\varepsilon)$-separated subset of $K$. Since $K$ is compact, $S(n,\varepsilon,K)$ is finite. The \emph{topological entropy of $f$ on $K$} is defined by
\[
h(f,K)=\lim_{\varepsilon\to0}\limsup_{n\to\infty}
\frac{1}{n}\log S(n,\varepsilon,K),\]
and the \emph{topological entropy of $f$} is given by
\(
h(f)=h(f,X).
\)

\begin{definition}
Let $\phi$ be a continuous flow on $X$, and let $K\subset X$ be compact. The \emph{topological entropy of $\phi$ on $K$} is defined by
\(
h(\phi,K):=h(\phi_1,K),
\)
where $\phi_1$ is the time-one map of the flow. The \emph{topological entropy of $\phi$} is defined by
\(
h(\phi):=h(\phi_1).
\)
\end{definition}

We now recall some basic notions from symbolic dynamics. Let
\(
A_n=\{0,1,\dots,n-1\},
\)
and denote by
\(
\Sigma_n=A_n^{\mathbb{Z}}
\)
the set of all bi-infinite sequences with entries in $A_n$. The space $\Sigma_n$ can be endowed with a metric that makes it a compact metric space. For distinct points $s,s'\in\Sigma_n$, define
\[
d(s,s')=\frac{1}{2^{N}},
\textrm{ where }
N=\min\{|i|:s_i\neq s'_i\}.
\]
Equivalently,
\(
d(s,s')=\frac{1}{2^{N}},
\)
where $N$ is the largest integer such that
\(
s_j=s'_j\), for every \( |j|<N.
\)
The \emph{full shift on $n$ symbols} is the homeomorphism
\[
\sigma:\Sigma_n\longrightarrow\Sigma_n,
\textrm{ defined by }
\sigma\bigl((s_i)_{i\in\mathbb{Z}}\bigr)
=(s_{i+1})_{i\in\mathbb{Z}}.
\]
It is well known that
\(
h(\sigma)=\log n.
\)
We say that a homeomorphism $f:X\to X$ \emph{factors onto} the full shift $\sigma:\Sigma_n\to\Sigma_n$ if there exists a continuous surjective map
\(
\pi:X\to\Sigma_n
\)
such that
\(
\sigma\circ\pi=\pi\circ f.
\)
Since topological entropy does not increase under factor maps, every homeomorphism that factors onto the full shift has positive topological entropy.

Another important concept used throughout this work is that of a suspension flow, which we now recall. Let $f:X\to X$ be a homeomorphism and let $\rho:X\to(0,\infty)$ be a continuous function, called the \emph{roof function}. Consider the quotient space
\[
\widetilde{X}
=
(X\times\mathbb{R})/\sim,
\]
where the equivalence relation is generated by
\(
(x,\rho(x))\sim(f(x),0).\)
\begin{definition}\label{Def-Susp}
The \emph{suspension flow} of $f$ with roof function $\rho$ is the flow
\(
\phi:\mathbb{R}\times\widetilde{X}\longrightarrow\widetilde{X},
\)
defined by
\(
\phi_t(x,s)=(f^n(x),s'),
\)
where $n$ and $s'$ satisfy
\[
t+s
=
\sum_{j=0}^{n-1}\rho(f^j(x))+s'.
\]
\end{definition}

It is straightforward to verify that the flow defined in Definition~\ref{Def-Susp} is regular and continuous. Moreover, the suspension construction preserves positivity of topological entropy: the suspension flow has positive topological entropy if and only if its base map $f$ has positive topological entropy.

We end this section by recalling the concept of cross-section  at regular points of continuous lows, a concept that will be widely used in this work.
\begin{definition}\label{def:cross_section}
A compact subset \(S\subset X\) is called a cross-section of time \(T>0\) through $x$  if:
\begin{enumerate}
    \item $x\in S$.
    \item $S\cap \phi_{[-T,T]}(y)=\{y\}$ for all \(y\in S\).
    \item  $\phi_{[-T,T]}(S)$ is a closed neighbourhood of $x$.
\end{enumerate}
For a cross-section $S$ of time $T$, we denote $
S^*:=S\cap\operatorname{Int}_X\bigl(\phi_{[-T,T]}(S)\bigr).$
\end{definition}
Observe that by the flow-box property,
\[
\phi_{(-T,T)}(S^*)
=
\operatorname{Int}_X\bigl(\phi_{[-T,T]}(S)\bigr)
=
\operatorname{Int}_X\bigl(\phi_{[-T,T]}(S^*)\bigr).
\]
Cross-sections will be crucial to our results and will appear many times during this text. The construction of cross-sections through regular points of a $C^1$-flow $\phi$  is easily obtained by applying the exponential map to the normal bundle of the velocity vector field generated by $\phi$. The case of continuous flows is more delicate and we refer the reader to \cite{Whitney33} for a detailed exposition of this subject. In particular, from \cite[Theorem~17A]{Whitney33} we obtain the following lemma:

\begin{lemma}\label{lemma_cross-sections}
    For every $x\in X\setminus Sing(\phi)$, there is $t_x>0$ and 
    a local cross-section \(S_x\) of time \(t_x\) such that \(x\in S_x^*\).
\end{lemma}

We observe that $t_x$ and the diameter of the cross-section $S_x$ in the previous theorem is highly dependent on the distance of $x$ to $Sing(\phi)$. Nevertheless, we can have a much finer control of these parameters for sets of points which are uniformly away from $Sing(\phi)$.

\section{Standard vs Oriented Shadowing on Surfaces}\label{sec:examplesufaces}

In this section we are going to prove Theorem \ref{thm:denseflows}. The first ingredients of the proof consist of Theorems \ref{thm:disc} and \ref{thm:annulus}. In these results, we shall construct flows on the closed disc and in the annulus that have the oriented shadowing property but fails to satisfy the standard shadowing property. To achieve this, we introduce a countable collection of singularities into a flow that originally possesses the shadowing property. These singularities are carefully arranged so that the shadowing property is preserved, while the reparametrization estimates required for standard shadowing are no longer valid. Let us denote $$D=\{x\in \mathbb{R}^2: \Vert x\Vert\leq 1\} \textrm{ and } A=\{x\in \mathbb{R}^2: 1\leq\Vert x\Vert\leq 2\}.$$

In what follows, we prove Theorem~\ref{thm:disc}, which is one of the main results of the paper. It is one of the main ingredients in the Proof of Theorem \ref{thm:denseflows} and shows that oriented shadowing is not sufficient for standard shadowing, even if we have low dimensional smooth flow. The proof is by a rather technical
construction, requiring several steps with careful control of parameters.  For the reader's convenience, we first explain the structure of the argument and the purpose of its main steps.

\begin{enumerate}
    \item \textbf{Construction of the base flow.} 
    We start with a simple Morse-Smale flow on the disk. We then multiply 
    its vector field by a smooth non-negative function \(\rho\). 
    This operation preserves the orbit curves but changes the speed at 
    which they are traversed. At the zeros of \(\rho\), new singularities 
    appear, and the corresponding orbit curves split into branches 
    asymptotic to these singularities.

    \item \textbf{Geometric arrangement and separation of crossing times.} 
    The zeros of \(\rho\) are arranged so that small pseudo-orbits can 
    jump between nearby branches, but only in an order compatible with the 
    flow's orientation. Inside two distinguished flow boxes, \(\rho\) is designed to produce very large crossing times. The crossing  times of sufficiently close orbit curves can differ by arbitrarily 
    large factors. This time-scale separation is the key mechanism 
    that will distinguish oriented shadowing from standard shadowing.

    \item \textbf{Proving oriented shadowing.} 
    We first prove that every sufficiently accurate pseudo-orbit must 
    follow the correct geometric order. Locally, as long as the velocity 
    is not too small, the forward displacement dominates the errors from 
    the jumps. Globally, this implies that the pseudo-orbit visits the 
    neighbourhoods of the singularities in the same sequence as a genuine 
    orbit, and its regular segments stay within a small distance. Consequently, we can select a genuine orbit that shares the same geometric itinerary. Since oriented shadowing allows arbitrary increasing 
    reparametrizations, we can adjust the timing of this selected orbit to 
    match the pseudo-orbit's geometry without reproducing its exact 
    crossing times.

    \item \textbf{Obstacle to standard shadowing.} 
    To rule out standard shadowing, we construct a specific pseudo-orbit 
    that stays near a given point for a very long time and then crosses 
    the second flow box through small jumps. Any genuine orbit tracing 
    this pseudo-orbit must combine short and long crossing times 
    through the box. Because standard  shadowing requires extra regularity for the reparametrizations, it cannot allow this discrepancy.
\end{enumerate}

Let us now proceed with the announced theorem and its proof.
\medskip

\begin{theorem}\label{thm:disc}
There is a $C^\infty$-flow in  $D$ with the oriented shadowing property, but without the standard shadowing property.
\end{theorem}

\begin{proof}
We begin by considering the closed disc \(D\) and the vertical radius
\[
 I=\{(w,z)\in D:w=0,\ z\ge0\}.
\]
We use the \(C^\infty\) vector field
\begin{equation}\label{eq:explicit-initial-flow}
 X(x,y)=\bigl(-(1-x^2-y^2)x-y,\,
               x-(1-x^2-y^2)y\bigr).
\end{equation}
Its flow will be denoted by \(\phi'\).  In polar coordinates it satisfies
\begin{equation}
 \dot r=-r(1-r^2),\qquad \dot\theta=1.\label{eq:ramdtheta}
\end{equation}
Consequently, \(0\) is a hyperbolic sink, \(\partial D\) is a
hyperbolic repelling periodic orbit, every other orbit converges to
\(\partial D\) in the past and to \(0\) in the future, and \(r\) is
strictly decreasing on \(D\setminus(\{0\}\cup\partial D)\).

For \(0<r<1\), define
\[
 u(r)=\log\frac{r}{\sqrt{1-r^2}}
\]
and the functions
\begin{equation}\label{eq:def:a}
 a\colon\operatorname{Int}D\setminus\{0\}\longrightarrow\mathbb S^1,
 \qquad
 a(r,\theta)=\theta+u(r)\pmod{2\pi},
\end{equation}
and
\begin{equation}\label{eq:def:q}
 q\colon\operatorname{Int}D\setminus\{0\}\longrightarrow\mathbb R,
 \qquad
 q(r,\theta)=-u(r).
\end{equation}
Since
\[
 u'(r)=\frac1{r(1-r^2)},
\]
we have
\begin{equation}\label{eq:orbit-coordinates}
 \dot a=0,\qquad \dot q=1
\end{equation}
for \(\phi'\).  Thus \(a\) is constant along every
orbit and labels the orbits of \(\phi'\), whereas \(q\) is strictly
increasing along them.  On every compact annulus contained
in \(\operatorname{Int}D\setminus\{0\}\), the coordinates \((a,q)\)
and the Euclidean coordinates are uniformly bi-Lipschitz. 

 In what follows, we shall modify $\phi'$, by including some singularities, to obtain a new flow $\phi$ with the shadowing property, but without the standard shadowing property.  We will perform speed change on the orbits. After the change we will have
 \begin{equation}\label{eq:time-changed-coordinates}
 	\dot a=0,\qquad \dot q=\rho(a,q)\ge0
 \end{equation}
 that is, the time change will alter only the speed but not the
 orientation of the orbit curves (possibly, with finitely many singularities on some curves).
In particular,
 \[
  a(\phi_t(x))=a(x)
 \]
 whenever the orbit is contained in
 \(\operatorname{Int}D\setminus\{0\}\).
 For this sake, we are going to construct a compact set $S\subset I$ and a $C^\infty$ map $\rho: D\to [0,1]$ such that
\begin{enumerate}
    \item $0\leq \rho(x)\leq 1$, for every $x\in D$.
     \item  $\rho(x)=0$ if and only if $x\in S$.
     \item $\rho\equiv 1$ on a small collar of $\partial D$.
\end{enumerate}
The desired flow $\phi$ is then obtained as the flow generated by the continuous vector field $\rho X$. Consequently, the velocity of $\phi$ at any point $x\in D$ is given by $\rho(x)X(x)$  and every point in $S$ is a singularity for $\phi$.
Since $X$ and $\rho$ are $C^\infty$, the vector field $\rho X$ is $C^\infty$, and in particular locally Lipschitz; hence it has unique integral curves. As $D$ is compact and $\rho X$ is tangent to $\partial D$, these curves are defined for all time, so $\rho X$ generates a $C^\infty$ flow $\phi$ on $D$, whose singularities are exactly the zeros of $\rho X$, that is $\sing(\phi)=S\cup\{0\}$. Moreover, near $\sing(\phi)$ the norm $\|\rho X(x)\|$ is bounded above by a constant multiple of $d(x,\sing(\phi))$, so an orbit takes infinite time to reach a singularity. Thus every point of $S\cup\{0\}$ is reached only asymptotically. Consequently the orbits of $\phi$ through points of $D\setminus(S\cup\{0\})$ are exactly the orbits of $\phi'$, reparametrized and broken into asymptotic branches at each point where they meet $S$.

\medskip
\noindent\textit{The Construction of $S$.}
Let $P$ denote the first return map of the points in $I\setminus \{0\}$ to $I\setminus \{0\}$ under the flow $\phi'$. Fix any point $x_0$ in the interior of the segment $I$ and let $z_0=P(x_0)$ and $z'_0=P(z_0)$. Let $J, J',J''\subset I$ be small intervals with endpoints  $x_0, z_0,z'_0$, respectively, such that additionally:
\begin{itemize}
\item $P(J)=J'$, $P(J')=J''$.
\item $J\cap J'=\emptyset$ and $J'\cap J''=\emptyset$.
\end{itemize}
Since $0$ is a hyperbolic sink, $\partial D$ is a hyperbolic repelling periodic orbit, and $r$ is strictly decreasing along every orbit in $D\setminus(\{0\}\cup\partial D)$, the map $P$ is a homeomorphism of $I$ with exactly two fixed points, the endpoint $I\cap\partial D$, which is repelling, and $0$, which is attracting, and $P(s)$ lies strictly between $0$ and $s$ otherwise and the return time is bounded above and below. Thus $P$ is a Morse--Smale interval map and $\phi'$ is a Morse--Smale, hence structurally stable, flow on $D$; in particular $\phi'$ has the shadowing property.

Let \(x_1\) be the endpoint of \(J\) different from \(x_0\), and put
\(z_1=P(x_1)\).  Thus
\[
 J=[x_0,x_1],\qquad J'=[z_0,z_1],\qquad
 J''=[z'_0,z'_1],
\]
where \(z'_1=P(z_1)\).  Choose \(y_1\) strictly between \(z_0\) and
\(z_1\).  Using the affine structure of the segment \(J'\), define,
for every \(n\ge2\),
\begin{equation}\label{eq:centered-choice}
 y_n=z_0+\frac12(y_{n-1}-z_0),\qquad
 z_n=z_0+\frac34(y_{n-1}-z_0),\qquad
 x_n=P^{-1}(z_n).
\end{equation}
Observe that
\begin{equation}\label{eq:zn-midpoint}
 z_n=\frac{y_{n-1}+y_n}{2}\qquad(n\ge2).
\end{equation}
and  \(y_n,z_n\to z_0\).  Since \(P^{-1}\) is continuous
and order preserving, \((x_n)\) is strictly monotone and
\(x_n\to P^{-1}(z_0)=x_0\).  Also, \(y_n\) lies strictly between
\(z_n=P(x_n)\) and \(z_{n+1}=P(x_{n+1})\).  Hence
\(P^{-1}(y_n)\) lies strictly between \(x_n\) and \(x_{n+1}\), and the
trajectories through the points \(x_n\) and \(y_m\) are pairwise
distinct.
Denote
 \begin{equation}\label{eq:eta}
     \eta=\frac{1}{100}\min \{\diam(J),\diam(J'),\diam(J''),|x_1-z_0|,|z_1-z'_0|\}.\end{equation}
Finally, we conclude our choice of $S$ by denoting
$$
S=\{x_n;n\geq 0\}\cup \{y_n;n\geq 1\}\cup \{z_0\}.
$$
Observe that $S$ is clearly compact, because the only accumulation points of $S$ are $x_0$ and $z_0$, by the construction.

\begin{figure}[htbp!]
    \centering
    \begin{subfigure}[b]{0.6\textwidth}
        \centering
   	\begin{tikzpicture}[yscale=0.5*1.2, xscale=0.5*1.2]	
	\draw[line width=2pt, cyan!70!blue] (0, 0) -- (0, 5.5);
	
	\draw[red, line width=2.5pt] (0, 4.3) -- (0, 5.2) node[midway, left=0.1cm, black] {$J$};
	\draw[red, line width=2.5pt] (0, 2.9) -- (0, 3.8) node[midway, left=0.1cm, black] {$J'$};
	\draw[red, line width=2.5pt] (0, 1.6) -- (0, 2.3) node[midway, left=0.15cm, black] {$J''$};
	
	\draw[thick, mid arrow=0.5] (0, 5.5) arc[start angle=90, end angle=270, x radius=5.5, y radius=5];
	\draw[thick, mid arrow=0.5] (0, -4.5) arc[start angle=-90, end angle=90, x radius=5.5, y radius=5];
	
	\draw[thick, color=darkgray, mid arrow=0.25, mid arrow=0.75] (0, 4.7) 
arc[start angle=90, end angle=270, x radius=3.2, y radius=2.95] 
arc[start angle=-90, end angle=90, x radius=2.2, y radius=2.35];
	
	\draw[thick, mid arrow=0.25, mid arrow=0.75] (0, 3.2) 
	arc[start angle=90, end angle=270, x radius=1.8, y radius=1.85] 
	arc[start angle=-90, end angle=90, x radius=1.3, y radius=1.25];
	
	\filldraw (0, 0) circle (2pt) node[right=0.1cm] {$0$};
	
	\filldraw (0, 1.6) circle (2pt) node[right=0.1cm] {$z_1'$};
	\filldraw (0, 2.3) circle (2pt) node[right=0.1cm] {$z_0'$};
	\filldraw (0, 2.9) circle (2pt) node[right=0.1cm] {$z_1$};
	\filldraw (0, 3.8) circle (2pt) node[right=0.1cm] {$z_0$};
	\filldraw (0, 4.3) circle (2pt) node[right=0.1cm] {$x_1$};
	\filldraw (0, 5.2) circle (2pt) node[right=0.1cm] {$x_0$};
	
\end{tikzpicture}
    \end{subfigure}
    \begin{subfigure}[b]{0.3\textwidth}
        \centering
	
        	\begin{tikzpicture}[yscale=0.6*1.2, xscale=0.6*1.5]
        		\coordinate (X1) at (0,0);
        		\coordinate (X2) at (0,2);
        		\coordinate (Xn) at (0,3.5);
        		\coordinate (Xn1) at (0,4.5);
        		\coordinate (X0) at (0,5.5);
        		
        		\coordinate (Z1) at (4,0);
        		\coordinate (Y1) at (4,1);
        		\coordinate (Z2) at (4,2);
        		\coordinate (Zn) at (4,3.5);
        		\coordinate (Yn) at (4,4);
        		\coordinate (Zn1) at (4,4.5);
        		\coordinate (Z0) at (4,5.5);
        		
        		\draw[red, line width=2pt] (X1) -- (X0);
        		\draw[red, line width=2pt] (Z1) -- (Z0);
        		\node[above=0.3cm] at (X0) {$J$};
        		\node[above=0.3cm] at (Z0) {$J'$};
        		
        		\draw[-{Stealth[scale=1.2]}] (X1) -- (2.1,0) -- (Z1);
        		\draw[-{Stealth[scale=1.2]}] (X2) -- (2.1,2) -- (Z2);
        		\draw[-{Stealth[scale=1.2]}] (Xn) -- (2.1,3.5) -- (Zn);
        		\draw[-{Stealth[scale=1.2]}] (Xn1) -- (2.1,4.5) -- (Zn1);
        		\draw[-{Stealth[scale=1.2]}] (X0) -- (2.1,5.5) -- (Z0);
        		
        		\filldraw (X1) circle (1.5pt) node[left=0.2cm] {$x_1$};
        		\filldraw (X2) circle (1.5pt) node[left=0.2cm] {$x_2$};
        		\filldraw (Xn) circle (1.5pt) node[left=0.2cm] {$x_n$};
        		\filldraw (Xn1) circle (1.5pt) node[left=0.2cm] {$x_{n+1}$};
        		\filldraw (X0) circle (1.5pt) node[left=0.2cm] {$x_0$};
        		
        		\filldraw (Z1) circle (1.5pt) node[right=0.1cm] {$z_1$};
        		\filldraw (Y1) circle (1.5pt) node[right=0.1cm] {$y_1$};
        		\filldraw (Z2) circle (1.5pt) node[right=0.1cm] {$z_2$};
        		\filldraw (Zn) circle (1.5pt) node[right=0.1cm] {$z_n$};
        		\filldraw (Yn) circle (1.5pt) node[right=0.1cm] {$y_n$};
        		\filldraw (Zn1) circle (1.5pt) node[right=0.1cm] {$z_{n+1}$}; 
        		\filldraw (Z0) circle (1.5pt) node[right=0.1cm] {$z_0$};
        		
        	\end{tikzpicture}
    \end{subfigure}
    
    \caption{The structure of $S$, $J$, $J'$ and $J''$.}
    \label{fig:setS}
\end{figure}
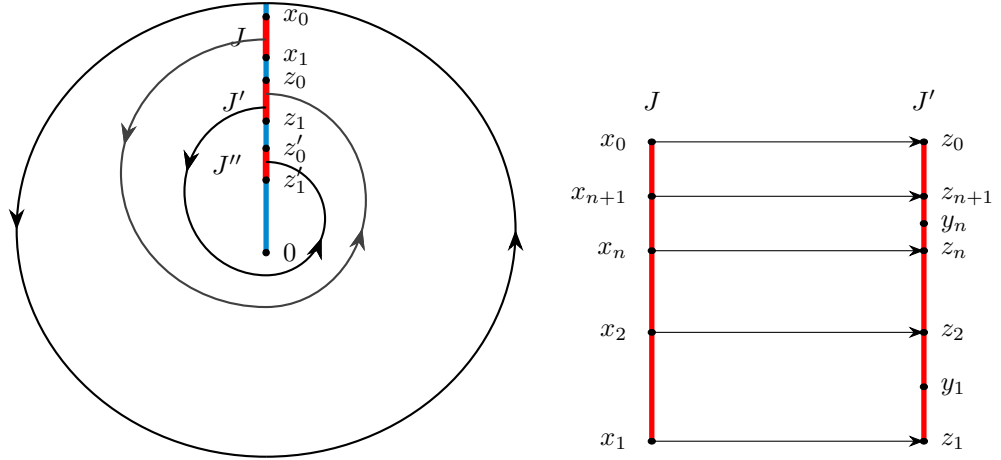

\medskip
\noindent \textit{The construction of $\rho$ and the flow $\phi$.}
After choosing the set $S$, we proceed to construct the function $\rho$. We begin by setting notation. Let $\kappa>0$ be a positive constant. Since all the points in $J$ and $J'$ are regular for $\phi'$, by reducing $\kappa$ when necessary, one can assume that the sets $$Q_1=\phi'_{[-\kappa,\kappa]}(J), Q'_1=\phi'_{[-\kappa,\kappa]}(J') \textrm{ and } Q''_1=\phi'_{[-\kappa,\kappa]}(J'')$$ are flow boxes. In addition, we can assume that, for any $x\in J\cup J'\cup J''$, the diameter of the orbit segment $\phi'_{[-\kappa,\kappa]}(x)$ is smaller than $\frac{\eta}{2}$.  This in turn implies that $Q_1\cap Q_1'=\emptyset$ and $Q_1'\cap Q''_1=\emptyset$.
For each $n>1$, we denote $$Q_n=\phi'_{[-\kappa,\kappa]}([x_0,x_n]) \textrm{ and } Q'_{n}=\phi'_{[-\kappa,\kappa]}([z_0,y_n]).$$
Notice that $Q_{n+1}\subset Q_n$ and $Q'_{n+1}\subset Q'_n$, for every $n\geq 1$. To finish the notation setup, let us denote:
$$L_n^-=\phi'_{-\kappa}([x_0,x_n]),\quad L_n^+=\phi'_{\kappa}([x_0,x_n]),\quad {L'}_n^-=\phi'_{-\kappa}([z_0,y_n]) \textrm{ and } {L'}_n^+=\phi'_{\kappa}([z_0,y_n]).$$
Observe that the previous sets are all cross-sections for $\phi'$.

Let $\hat{J}\supset \interior \hat{J}\supset J$ be a slightly longer arc in $I$ which still satisfies $P(\hat J)\cap \hat J=\emptyset$ and let
$$\hat{Q}_1=\phi'_{[-\kappa,\kappa]}(\hat{J}) \text{ and }  \hat{Q}'_1=\phi'_{[-\kappa,\kappa]}(P(\hat {J})).$$
Together with the bound on the diameters of the orbit segments and \eqref{eq:eta}, this gives $\hat Q_1\cap\hat Q'_1=\emptyset$.

We construct $\rho$ in the flow-box coordinates $\hat{Q}_1\cong \hat{J}\times[-\kappa,\kappa]$ and $\hat{Q}'_1\cong P(\hat{J})\times[-\kappa,\kappa]$, in which $\phi'$ is the translation $(s,\tau)\mapsto(s,\tau+\,\cdot\,)$ and $S$ corresponds to the points $(x_m,0)_{m\ge0}$ on the $\hat{J}$-side and $(y_m,0)_{m\ge1}$, $(z_0,0)$ on the $P(\hat{J})$-side. Set
$\rho(s,\tau)=1-b(\tau)\bigl(1-g(s)\bigr)$ on the above boxes, and $\rho\equiv1$  elsewhere on $D$,
where:
\begin{itemize}
\item \(b\colon[-\kappa,\kappa]\to[0,1]\) is an even and smooth
bump function, strictly decreasing on \([0,\kappa]\), with
\(b(0)=1\), \(b(\tau)<1\) for \(\tau\ne0\), and \(b(\tau)=0\) near
\(\{-\kappa,\kappa\}\).  We also require \(b''(0)<0\).  Then, for sufficiently small
\(\tau_0>0\), there are constants \(c_b,C_b>0\) such that
\begin{equation}\label{eq:bquad}
 c_b\tau^2\le1-b(\tau)\le C_b\tau^2
 \qquad(|\tau|\le\tau_0).
\end{equation}
Since \(1-b(\tau)>0\) for \(0<|\tau|\le\kappa\), decreasing \(c_b\) if necessary, we may assume that
$c_b\tau^2\le1-b(\tau)$ for all $\tau \in [-\kappa,\kappa]$.
\item \(g\) is a smooth function on \(\hat J\), respectively
\(P(\hat J)\), with values in \([0,1]\), equal to \(1\) near both
ends of the arc, and vanishing exactly on
\(\{x_m:m\ge0\}\), respectively
\(\{y_m:m\ge1\}\cup\{z_0\}\). 

Consider the $C^\infty$ function
\[
 \psi(t)=
 \begin{cases}
 \exp\bigl(4-\frac{1}{t(1-t)}\bigr),&0<t<1,\\
 0,&t\notin(0,1),
 \end{cases}
\]
with maximum \(\psi(1/2)=1\) and values in \([0,1]\).
Use affine coordinates on \(\hat J\) and \(P(\hat J)\), oriented from
\(x_0\) to \(x_1\) and from \(z_0\) to \(z_1\), respectively.  On
\((x_{n+1},x_n)\), \(n\ge1\), define
\[
 w_n=x_n-x_{n+1},\qquad
 g_n(s)=h_n\psi\!\left(\frac{s-x_{n+1}}{w_n}\right),
 \qquad 0<h_n\le\min\{1,w_n^n\}.
\]
Similarly, on \((y_n,y_{n-1})\), \(n\ge2\), put
\[
 \widehat w_n=y_{n-1}-y_n,\qquad
 \widehat g_n(s)=\widehat h_n
 \psi\!\left(\frac{s-y_n}{\widehat w_n}\right),
 \qquad 0<\widehat h_n\le\min\{1,\widehat w_n^n\}.
\]
Define \(g\) by these functions on the bounded gaps and set \(g=0\)
on the prescribed zero sets.  On the remaining components use smooth
cutoffs, flat at the zero endpoint and equal to \(1\) near the
endpoints of the enlarged arcs.  By
\eqref{eq:zn-midpoint}, \(\widehat g_n\) has its unique maximum
\(\widehat h_n\) at \(z_n\). We will later specify how fast $\widehat h_n$ should decrease to satisfy some additional properties.
\end{itemize}
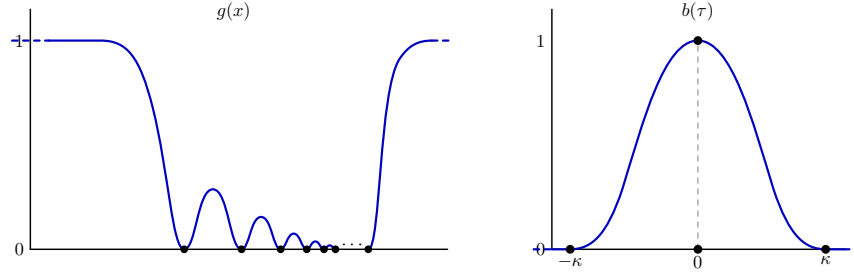
\begin{figure}[ht]
	\centering
	\resizebox{0.9\textwidth}{!}{%
	\begin{tikzpicture}[
		x=1cm,
		y=1cm,
		line cap=round,
		line join=round,
		bluegraph/.style={blue!75!black, very thick},
		point/.style={circle, fill=black, inner sep=1.7pt}
		]
\begin{scope}[xshift=0cm]
	
	\draw[thick] (0,0) -- (8.0,0);
	\draw[thick] (0,0) -- (0,4.6);
	\node[left] at (0,0) {$0$};
	\node[left] at (0,4) {$1$};
	
	\node at (3.9,4.55) {$g(x)$};
	
	\draw[bluegraph,dashed] (-0.35,4) -- (0.35,4);
	
	\draw[bluegraph]
	(0.35,4)
	-- (1.35,4)
	.. controls (1.85,4) and (2.10,3.70) .. (2.35,2.70)
	.. controls (2.65,1.50) and (2.72,0) .. (2.95,0);
	
	\draw[bluegraph]
	(2.95,0)
	.. controls (3.18,0) and (3.20,1.15) .. (3.50,1.15)
	.. controls (3.80,1.15) and (3.82,0) .. (4.05,0);
	
	\draw[bluegraph]
	(4.05,0)
	.. controls (4.20,0) and (4.22,0.62) .. (4.425,0.62)
	.. controls (4.63,0.62) and (4.65,0) .. (4.80,0);
	
	\draw[bluegraph]
	(4.80,0)
	.. controls (4.90,0) and (4.92,0.30) .. (5.05,0.30)
	.. controls (5.18,0.30) and (5.20,0) .. (5.30,0);
	
	\draw[bluegraph]
	(5.30,0)
	.. controls (5.37,0) and (5.38,0.15) .. (5.465,0.15)
	.. controls (5.55,0.15) and (5.56,0) .. (5.63,0);
	
	\draw[bluegraph]
	(5.63,0)
	.. controls (5.675,0) and (5.685,0.075) .. (5.74,0.075)
	.. controls (5.795,0.075) and (5.805,0) .. (5.85,0);
	
	\node at (6.2,0.08) {$\cdots$};
	
	\draw[bluegraph]
	(6.48,0)
	.. controls (6.72,0) and (6.68,3.05) .. (7.08,3.65)
	.. controls (7.25,3.90) and (7.42,4) .. (7.70,4);
	
	\draw[bluegraph,dashed] (7.70,4) -- (8.10,4);
	
	\foreach \x in {
		2.95,
		4.05,
		4.80,
		5.30,
		5.63,
		5.85,
		6.48
	}{
		\fill (\x,0) circle[radius=2.1pt];
	}
	
\end{scope}
		
		\begin{scope}[xshift=10cm]
			
			\draw[thick] (0,0) -- (5.6,0);
			\draw[thick] (0,0) -- (0,4.6);
			
			\node[left] at (0,0) {$0$};
			\node[left] at (0,4) {$1$};
			
			\node at (2.8,4.55) {$b(\tau)$};
			
			\draw[bluegraph]
			(0,0) -- (0.35,0);
			
			\draw[bluegraph]
			(0.35,0)
			.. controls (0.75,0) and (1.05,0.25) .. (1.35,1.15)
			.. controls (1.75,2.45) and (2.15,4) .. (2.80,4)
			.. controls (3.45,4) and (3.85,2.45) .. (4.25,1.15)
			.. controls (4.55,0.25) and (4.85,0) .. (5.25,0)
			-- (5.60,0);
			
			\draw[bluegraph,dashed] (-0.35,0) -- (0,0);
			\draw[bluegraph,dashed] (5.60,0) -- (5.95,0);
			
			\draw[dashed,gray]
			(2.80,0) -- (2.80,4);
			
			\node[point] at (0.35,0) {};
			\node[point] at (2.80,0) {};
			\node[point] at (2.80,4) {};
			\node[point] at (5.25,0) {};
			
			\node[below] at (0.35,0) {$-\kappa$};
			\node[below] at (2.80,0) {$0$};
			\node[below] at (5.25,0) {$\kappa$};
			
		\end{scope}
	\end{tikzpicture}
}
	\caption{Exemplary graphs of functions $g$ an $b$.}
	\label{fig:graphs-g-b}
\end{figure}
The equality \(b=0\) near $-\kappa,\kappa$ makes \(\rho\) extend smoothly
by \(1\) outside the boxes.
Then $\rho$ is $C^\infty$ and satisfies:
\begin{enumerate}
\item\label{con:thm1:1} $\rho(x)=1$ for every $x\in D\setminus(\hat{Q}_1\cup\hat{Q}'_1)$;
\item\label{con:thm1:2} $\rho(x)=0$ if and only if $x\in S$;
\item\label{con:thm1:3} $0\le\rho(x)\le1$ for every $x\in D$.
\end{enumerate}
Indeed, \eqref{con:thm1:3} holds because values of $b$ and $1-g$ are in $[0,1]$;
\eqref{con:thm1:1} because $g\equiv1$ near the ends of the arc and $b\equiv0$ near $\pm\kappa$; and for \eqref{con:thm1:2}, $\rho(s,\tau)=0$ iff $b(\tau)(1-g(s))=1$, i.e.\ iff $b(\tau)=1$ \emph{and} $g(s)=0$, i.e.\ iff $\tau=0$ and $s$ is a singular value --- that is, iff $(s,\tau)\in S$.

By construction, \(\rho\) is \(C^\infty\) away from the vertical
segments
\[
\{x_0\}\times[-\kappa,\kappa]\subset\widehat Q_1,
\qquad
\{z_0\}\times[-\kappa,\kappa]\subset\widehat Q'_1
\]
so it remains to verify its smoothness along them.
For every fixed \(k\) and all sufficiently large \(n\),
\[
 \|g_n\|_{C^k}\le C_k\frac{h_n}{w_n^k}
 \le C_k w_n^{n-k},
 \qquad
 \|\widehat g_n\|_{C^k}
 \le C_k\frac{\widehat h_n}{\widehat w_n^k}
 \le C_k\widehat w_n^{n-k},
\]
where \(C_k=\max_{j\le k}\|\psi^{(j)}\|_\infty\).  Both right-hand
sides tend to \(0\), so \(g\), and hence \(\rho\), is \(C^\infty\) at
the accumulation orbits.

\medskip
\noindent \emph{Crossing times.} In the flow-box coordinates, within trajectory over a fixed vertical segment at $s$ without critical values over it, the time the $\phi$-orbit  spends in the box is
\[
T(s)=\int_{-\kappa}^{\kappa}\frac{d\tau}{\rho(s,\tau)}=\int_{-\kappa}^{\kappa}\frac{d\tau}{1-b(\tau)\bigl(1-g(s)\bigr)} ,
\]
which changes in the values depending only on the function of $g(s)$. Furthermore, for each fixed $\tau$ the denominator increases with $g(s)$ (as $b(\tau)\ge0$), so the integrand decreases, hence $T$ is a strictly decreasing function of $g(s)$. It equals $2\kappa$ when $g(s)=1$ and tends to $+\infty$ as $g(s)\to0$ (the integrand then blows up at $\tau=0$, where $b=1$). In particular $T(s)$ is large precisely where $g(s)$ is small. For the future arguments, we record the estimate of the crossing time. Since $\rho(s,\tau)=(1-b(\tau))+b(\tau)g(s)$ and $\rho(s,0)=g(s)$, splitting the integral at a $\tau_0$ with $b\ge\frac12$ on $[-\tau_0,\tau_0]$ and using \eqref{eq:bquad}, we obtain that:
\begin{equation}\label{eq:Tsqrt}
T(s)\ \le\ \int_{|\tau|\le\tau_0}\frac{d\tau}{c_b\tau^{2}+\frac12 g(s)}+\int_{|\tau|\ge\tau_0}\frac{d\tau}{c_b\tau^{2}}\ \le\ C\bigl(1+g(s)^{-1/2}\bigr),
\end{equation}
with $C$ depending only on $b$ and $\kappa$.

As we announced at start, let $\phi$ be the flow generated by the vector field $\rho X$. Hence, the flow $\phi$ satisfies  $\sing(\phi)=S\cup \{0\}$. Also, the orbit of any point $x\in D\setminus \phi'(S\times \R)$ under $\phi$ retains the same orbit as its orbit under the flow $\phi'$, possibly after a change of speed. The remaining orbits correspond to the addition of singularities into each orbit under $\phi'$ that crosses $S$.
We name these branches as follows: For $p\in S\setminus\{x_0,z_0\}$ we have $W^-(p)\setminus \{p\}$, the branch negatively asymptotic to $\partial D$ and positively asymptotic to $p$, and $W^+(p)\setminus \{p\}$ the branch negatively asymptotic to $p$ and positively asymptotic to $0$. The orbit of $x_0$ splits into three parts outside fixed points: $A=W^-(x_0)\setminus \{x_0\}$, $B=W^+(x_0)\setminus \{x_0\}=W^-(z_0)\setminus\{z_0\}$ and $C=W^+(z_0)\setminus \{z_0\}$.
Observe that any point $x\notin \hat{Q}_1\cup \hat{Q}'_1$ keeps the same velocity until it enters  $\hat{Q}_1\cup \hat{Q}'_1$. So, all those points move with constant angular velocity, similarly to $\phi'$. In contrast, the points inside $\hat{Q}_1\cup \hat{Q}'_1$ have their speed changed according to the function $\rho$.

It remains to choose the heights.  Let \(T_2(s)\) denote the full
\(\phi\)-crossing time of \(\hat Q'_1\) on the vertical through
\(s\) with $g(s)>0$, and put
\[
 \widehat T_n:=T_2(z_n).
\]
Choose the heights \(\widehat h_n\), \(n\ge2\), recursively so that
\begin{equation}\label{eq:direct-slowdown}
	\widehat T_{n+1}\ge2^n\widehat T_n
	\qquad(n\ge1).
\end{equation}
Since \(T_2\) is minimized on each gap
\((y_k,y_{k-1})\) at \(z_k\), and
\((\widehat T_k)\) is increasing, we obtain
\begin{equation}\label{eq:tail-slowdown}
	\inf_{\substack{s\in[z_0,z_{n+1}]\\g(s)>0}}T_2(s)
	=\widehat T_{n+1}
	\ge2^n\widehat T_n
	\qquad(n\ge1).
\end{equation}
This recursion is compatible with
\(\widehat h_n\le\min\{1,\widehat w_n^n\}\), since
\[
T_2(z_n)\longrightarrow+\infty
\qquad\text{as }\widehat h_n\longrightarrow0.
\]
Thus, at each step, \(\widehat h_{n+1}\) can be chosen sufficiently
small to satisfy \eqref{eq:direct-slowdown} while preserving the
required bound. 
For each \(Q\in\{\hat Q_1,\hat Q'_1\}\), write
\[
\chi_Q=(s_Q,\tau_Q)\colon Q\longrightarrow
J_Q\times[-\kappa,\kappa]
\]
for its flow-box chart, and define
\[
q_Q^0(s):=q\bigl(\chi_Q^{-1}(s,0)\bigr).
\]
Since the two flow-box charts and the function \(q\) are smooth on
the two compact boxes, there exists \(L\ge1\) such that, for every
\(Q\in\{\hat Q_1,\hat Q'_1\}\) and all \(x,y\in Q\),
\[
\begin{aligned}
	|s_Q(x)-s_Q(y)|&\le Ld(x,y),\\
	|\tau_Q(x)-\tau_Q(y)|&\le Ld(x,y),\\
	|q_Q^0(s_Q(x))-q_Q^0(s_Q(y))|&\le Ld(x,y),\\
	|q(x)-q(y)|&\le Ld(x,y).
\end{aligned}
\]
Thus a jump of size at most \(\delta\) changes each of the four
quantities on the left by at most \(L\delta\).

\begin{claim}\label{cl:box-passage}
Fix \(T_0>0\) and \(A>0\) such that \(AT_0>4L\).  There exist
constants \(C,\delta_0>0\), depending only on \(T_0,A\) and the
restrictions of \(\phi\) to \(\hat Q_1\) and \(\hat Q'_1\), such that
the following holds.

Fix \(0<\delta\le\delta_0\) and let
\(P=((p_i,t_i))_{i\in\mathbb Z}\)
be a \(\delta\)-\(T_0\) pseudo-orbit such that
\(t_i\in[T_0,2T_0)\) for every \(i\).  Fix
\(Q\in\{\hat Q_1,\hat Q'_1\}\), and let \(j_1<j_2\) be such that
\[
\phi_{[0,t_i]}(p_i)\subset Q\setminus H_\delta^Q
\qquad\text{for every }j_1\le i<j_2,
\]
where
\[
 H_\delta^Q:=
 \bigl\{p\in Q:v(s_Q(p),\tau_Q(p))\le A\delta\bigr\}.
\]
Assume that the integer interval \([j_1,j_2)\) is maximal with this
property, and consider the corresponding subsegment of the pseudo-orbit
\[
P[j_1,j_2)
:=((p_i,t_i))_{j_1\le i<j_2}.
\]
Then the following assertions hold:
\begin{enumerate}
	\item\label{cl:box-passage:1} the sequence of transverse coordinates
	\[
	\tau_Q(p_{j_1}),\tau_Q(p_{j_1+1}),\ldots,
	\tau_Q(p_{j_2-1})
	\]
	is strictly increasing;
	
	\item\label{cl:box-passage:2} the length of \(P[j_1,j_2)\) satisfies $j_2-j_1\le C/\sqrt{\delta}$;
	
	\item\label{cl:box-passage:3} the total change of the transverse coordinate caused by the
	jumps within this subsegment satisfies
	\[
	\sum_{i=j_1}^{j_2-2}
	\left|
	\tau_Q(p_{i+1})-\tau_Q\bigl(\phi_{t_i}(p_i)\bigr)
	\right|
	\le C\sqrt{\delta}.
	\]
	\item\label{cl:box-passage:4}
If \(p_{j_2}\in Q\), then
\[
\sum_{i=j_1}^{j_2-1}
\left|
\tau_Q(p_{i+1})-\tau_Q(\phi_{t_i}(p_i))
\right|	\le C\sqrt{\delta}. 
\]
\end{enumerate}
\end{claim}

\begin{proof}
For \(j_1\le i<j_2-1\), the orbit segment lies outside
\(H_\delta^Q\), and therefore
\begin{equation}\label{eq:box-net-advance}
\begin{aligned}
 \tau_Q(p_{i+1})-\tau_Q(p_i)
 &\ge
 \tau_Q\bigl(\phi_{t_i}(p_i)\bigr)-\tau_Q(p_i)-L\delta\\
 &\ge (AT_0-L)\delta>3L\delta
\end{aligned}
\end{equation}
which proves \eqref{cl:box-passage:1}.

Let us now estimate the number of possible segments.  By
\eqref{eq:bquad} we have:
\begin{equation}
 v(s,\tau)\ge 1-b(\tau)\ge c_b\tau^2
 \quad (|\tau|\le\tau_0).\label{eq:bquad:cb}
\end{equation}
Fix $\delta\in (0,1)$ and divide \(0<|\tau|\le\tau_0\) into the following strips:
\[
 E_0=\{|\tau|\le K\sqrt\delta\}, \quad  E_j=\{2^{j-1}K\sqrt\delta<|\tau|
                \le2^jK\sqrt\delta\},\qquad j\ge1,
\]
where \(K\) is fixed and sufficiently large.  Estimate
\eqref{eq:box-net-advance} shows
that the number of elements in \(E_0\) is at most
\[
 2+\frac{2K\sqrt\delta}{(AT_0-L)\delta}
 \le \frac{C_0}{\sqrt\delta},
 \qquad
 C_0:=2+\frac{2K}{AT_0-L}.
\]
Put
\[
r_j=2^{j-1}K\sqrt\delta,\qquad
E_j^-=[-2r_j,-r_j),\qquad E_j^+=(r_j,2r_j].
\]
Call an orbit segment exceptional for \(E_j\) if it starts in
\(E_j^-\) and crosses the boundary at the level \(\tau=-r_j\).  Observe that there are at most two
such orbit segments.  Indeed, after the first crossing, the following jump
places the next initial point at a coordinate greater than
\(-r_j-L\delta\).  If the next segment also crosses \(-r_j\), then
\eqref{eq:box-net-advance} gives
\[
\tau_Q(p_{i+1})>-r_j+2L\delta,
\]
and the strict increase of the subsequent initial coordinates
precludes any further crossing of \(-r_j\).

For every nonexceptional orbit segment starting in \(E_j\), one has
\[
|\tau(\phi_t(p_i))|\ge r_j
\qquad(0\le t\le t_i).
\]
Consequently, for nonexceptional orbit segments we have by \eqref{eq:bquad:cb} that
\[
v\bigl(\phi_t(p_i)\bigr)
\ge c_b r_j^2
=c_bK^2\,4^{j-1}\delta
\qquad(0\le t\le t_i)
\]
which after taking jump into account gives
\[
\begin{aligned}
\tau_Q(p_{i+1})-\tau_Q(p_i)
	&\ge c_bK^2\,4^{j-1}\delta T_0-L\delta\\
	&\ge c'4^j\delta,
\end{aligned}
\]
where
\[
c'=\frac{c_bK^2T_0-L}{4}>0.
\]
Apart from the two exceptional points, the number of elements assigned to \(E_j\) is therefore at
most
\[
 \frac{2^{j+1}K\sqrt\delta}{c'4^j\delta}
 =\frac{2K}{c'2^{j-1}\sqrt\delta}.
\]
The number \(J_\delta\) of strips is at most
\(\tau_0/(K\sqrt\delta)+1\),  hence the total number of exceptional segments
is at most
\[
 2J_\delta\le
 \frac{2\tau_0/K+2}{\sqrt\delta}.
\]
and the maximal number of regular segments through all $E_j$ is bounded by:
\[
 \sum_{j=1}^{\infty}
 \frac{2K}{c'2^{j-1}\sqrt\delta}
 =\frac{4K}{c'\sqrt\delta}.
\]
Now let us provide an estimate for the region \(\{|\tau|\ge\tau_0\}\). Put $
 v_0:=\min\{v(s,\tau):|\tau|\ge\tau_0\}>0$
and note that if
\(\delta\le v_0T_0/(2L)\) then every complete step in this region increases $\tau$ by
at least \(v_0T_0/2\).  Thus the number of additional
elements arising when passing this region is at most
\[
 C_2:=2+\frac{4\kappa}{v_0T_0}
 \le\frac{C_2}{\sqrt\delta}.
\]
Introduce the following constants:
\[
\widetilde C:=C_0+\frac{2\tau_0}{K}+2+\frac{4K}{c'}+C_2,
 \qquad C:=\max\{1,L\}\widetilde C,
\]
and denote
\[
 \delta_0:=
 \min\left\{1,\frac{v_0T_0}{2L}\right\}.
\]
Every orbit segment starts either in \(E_0\), in one of the 
strips \(E_j\), or in \(\{|\tau|\ge\tau_0\}\).  Summing the estimates
for these three cases gives
\[
j_2-j_1
\le
\frac{C_0+2\tau_0/K+2+4K/c'+C_2}{\sqrt\delta}
=\frac{\widetilde C}{\sqrt\delta}
\le \frac{C}{\sqrt\delta},
\]
which proves \eqref{cl:box-passage:2}.

Put \(N=j_2-j_1\).  For every \(j_1\le i\le j_2-2\), both
\(\phi_{t_i}(p_i)\) and \(p_{i+1}\) belong to \(Q\).  Therefore,
by the Lipschitz estimate for the coordinate \(\tau_Q\),
\[
\left|
\tau_Q(p_{i+1})-\tau_Q\bigl(\phi_{t_i}(p_i)\bigr)
\right|
\le
Ld\bigl(p_{i+1},\phi_{t_i}(p_i)\bigr)
\le L\delta.
\]
The sum in \eqref{cl:box-passage:3} has \(N-1\) terms.  Hence,
using \eqref{cl:box-passage:2} we obtain:
\[
\sum_{i=j_1}^{j_2-2}
\left|
\tau_Q(p_{i+1})-\tau_Q\bigl(\phi_{t_i}(p_i)\bigr)
\right|
\le L\delta(N-1)
\le L\widetilde C\sqrt\delta
\le C\sqrt\delta.
\]

Under the hypothesis of \eqref{cl:box-passage:4}, one also has
\(p_{j_2}\in Q\).  Thus the same estimate applies to the additional
jump with index \(i=j_2-1\).  The extended sum has \(N\) terms, and
therefore
\[
\sum_{i=j_1}^{j_2-1}
\left|
\tau_Q(p_{i+1})-\tau_Q\bigl(\phi_{t_i}(p_i)\bigr)
\right|
\le L\delta N
\le L\widetilde C\sqrt\delta
\le C\sqrt\delta.
\]
This proves \eqref{cl:box-passage:3} and
\eqref{cl:box-passage:4}, completing the proof of the claim.
\end{proof}

Recall that
\[
 a\colon\operatorname{Int}D\setminus\{0\}\longrightarrow\mathbb S^1,\quad 
  q\colon\operatorname{Int}D\setminus\{0\}\longrightarrow\mathbb \R
\]
were defined by \eqref{eq:def:a} and \eqref{eq:def:q}
and that the orbit-label function $a$  is constant along the
orbits of \(\phi\).

\begin{claim}\label{cl:singular-neighbourhoods}
For every \(\varepsilon,\widetilde\varepsilon>0\) there exist \(T_0\ge1\),
\(\delta_0>0\), a closed neighborhood \(V_\partial\) of \(\partial D\), a
closed neighbourhood \(V_0\) of \(0\), closed neighbourhoods
\(W_\partial\supset V_\partial\) and \(W_0\supset V_0\), an integer
\(N\), and a family
\[
 \mathcal U=
 \{U_{x_i},U_{y_i}:1\le i\le N\}\cup\{U_{x_0},U_{z_0}\}
\]
of pairwise disjoint compact flow-box rectangles in
\(D\) such that
\[
 V_\partial\subset\operatorname{Int}W_\partial,
 \qquad V_0\subset\operatorname{Int}W_0,
\]
\[
 W_\partial\subset
 \left\{1-r<\frac{\varepsilon}{10}\right\},
 \qquad
 \operatorname{diam}(W_0)<\frac{\varepsilon}{10},
 \qquad
 (W_\partial\cup W_0)\cap
 \bigcup_{U\in\mathcal U}U=\emptyset,
\]
and the following assertions hold:
\begin{enumerate}
\item\label{cl:singular-neighbourhoods:1} every \(U\in\mathcal U\) satisfies
      \[
       \operatorname{diam}(U)<\frac{\varepsilon}{10},
       \qquad
       \operatorname{diam}_{\mathbb S^1}(a(U))
       <\frac{\widetilde\varepsilon}{10},
       \qquad
\operatorname{diam}q(U)<\frac{\pi}{10},
      \]
      and
      \[
      \begin{aligned}
       &x_i\in\operatorname{Int}U_{x_i},\qquad
       y_i\in\operatorname{Int}U_{y_i}\quad(1\le i\le N),\\
       &\{x_0\}\cup\{x_j:j>N\}\subset\operatorname{Int}U_{x_0},\\
       &\{z_0\}\cup\{y_j,z_j:j>N\}\subset\operatorname{Int}U_{z_0};
      \end{aligned}
      \]
\item\label{cl:singular-neighbourhoods:2} if \(0<\delta\le\delta_0\),
      \(P=((p_i,t_i))_{i\in\mathbb Z}\) is a
      \(\delta\)-\(T_0\) pseudo-orbit with
      \(t_i\in[T_0,2T_0)\), and \([j_1,j_2)\) is maximal subject to
      \[
       \phi_{[0,t_i]}(p_i)\subset
       R:=D\setminus
       \left(V_\partial\cup V_0\cup
       \bigcup_{U\in\mathcal U}U\right)
       \quad(j_1\le i<j_2),
      \]
      then
      \[
       q(p_{j_1})<q(p_{j_1+1})<\cdots<q(p_{j_2-1})
      \]
      and
      \[
       \sum_{i=j_1}^{j_2-2}
       d_{\mathbb S^1}\!\left(
       a(p_{i+1}),a\bigl(\phi_{t_i}(p_i)\bigr)\right)
       <\widetilde\varepsilon;
      \]
\item\label{cl:singular-neighbourhoods:3} Define
      \[
       S_0=0,\qquad
       S_i=\sum_{k=0}^{i-1}t_k\ (i>0),\qquad
       S_i=-\sum_{k=i}^{-1}t_k\ (i<0),
      \]
      and define the interpolated pseudo-orbit by
      \[
       \gamma_P(t)=\phi_{t-S_i}(p_i),
       \qquad t\in[S_i,S_{i+1}).
      \]
      For \(U\in\mathcal U\), a visit to \(U\) is a connected component of
      \(\gamma_P^{-1}(U)\).
      
      The following assertions hold:
      \begin{enumerate}[(a)]
      \item\label{cl:singular-neighbourhoods:3:a} For every \(i<j\),
      \begin{eqnarray*}
       p_i\notin V_\partial
       &\quad\Longrightarrow\quad&
       p_j\notin V_\partial,\\
       p_i\in V_0
       &\quad\Longrightarrow\quad&
       p_j\in V_0.
      \end{eqnarray*}
      Thus the sequence of initial points leaves \(V_\partial\) and
      enters \(V_0\) at most once.

      \item\label{cl:singular-neighbourhoods:3:b}
      Put
      \[
       \mathcal U_P
       :=\{U\in\mathcal U:\gamma_P(\mathbb R)\cap U\ne\emptyset\}.
      \]
      Then there are only the following three possibilities:
      \[
       \mathcal U_P=\emptyset,\qquad
       \mathcal U_P=\{U\}\ \ (U\in\mathcal U),
       \qquad\text{or}\qquad
       \mathcal U_P=\{U_{x_0},U_{z_0}\}.
      \]
      In the last case,
      \[
       \gamma_P(t)\in U_{x_0},\quad \gamma_P(t')\in U_{z_0}
       \quad\Longrightarrow\quad t<t'.
      \]
      Moreover, for \(U\in\mathcal U_P\),
      \[
      \begin{aligned}
       \gamma_P(t)\in V_\partial,
       \ \gamma_P(t')\in U&\quad\Longrightarrow\quad t<t',\\
       \gamma_P(t)\in U,
       \ \gamma_P(t')\in V_0&\quad\Longrightarrow\quad t<t'.
      \end{aligned}
      \]
Finally, no member of \(\mathcal U\) has two distinct visits
      which are respectively unbounded below and unbounded above.

      \item\label{cl:singular-neighbourhoods:3:c} There is a nonempty open arc
      \(I_P\subset\mathbb S^1\) such that
      \[
       I_P\subset a(\operatorname{Int}U)
       \qquad\text{for every \(U\in\mathcal U_P\)},
      \]
      and, for every \(\hat{a}\in I_P\),
      \[
       d_{\mathbb S^1}\bigl(a(\gamma_P(t)),\hat{a}\bigr)
       <\widetilde\varepsilon
       \qquad\text{whenever }\gamma_P(t)\in R.
      \]
      \end{enumerate}
\end{enumerate}
\end{claim}

\begin{proof}
 Choose
\[
 0<\nu_\partial<\frac1{12},\qquad
 0<\nu_0<\frac16
\]
sufficiently small, so that the closed sets
\[
 V_\partial=\{(r,\theta): 1-r \le2\nu_\partial\},
 \qquad V_0=\{(r,\theta) : r\le2\nu_0\}
\]
satisfy
\[
 V_\partial\subset\left\{1-r<\frac{\varepsilon}{10}\right\},
 \qquad \operatorname{diam}(V_0)<\frac{\varepsilon}{10}.
\]
We also require the larger neighborhoods
\[
 W_\partial=\{(r,\theta): 1-r\le3\nu_\partial\},
 \qquad W_0=\{r\le3\nu_0\}
\]
to be disjoint from \(\hat Q_1\cup\hat Q'_1\), and choose the
constants so that
\[
 W_\partial\subset
 \left\{(r,\theta) : 1-r<\frac{\varepsilon}{10}\right\},
 \qquad \operatorname{diam}(W_0)<\frac{\varepsilon}{10}.
\]

Choose \(N\ge1\) and put
\[
 X_N=\{x_0\}\cup\{x_j:j>N\},\qquad
 Z_N=\{z_0\}\cup\{y_j,z_j:j>N\},
\]
where \(N\) is so large that
\begin{equation}\label{eq:claim2-tail-smallness}
 \operatorname{diam}(X_N),\operatorname{diam}(Z_N)
 <\frac{\varepsilon}{20},\qquad
 \operatorname{diam}_{\mathbb S^1}
 \bigl(a(X_N)\cup a(Z_N)\bigr)
 <\frac{\widetilde\varepsilon}{100}.
\end{equation}
Let
\[
 \mathcal A_N=\{x_0,x_1,\ldots,x_N,y_1,\ldots,y_N,z_0\}
\]
and define
\[
 E_{x_i}=\{x_i\},\quad E_{y_i}=\{y_i\}\quad(1\le i\le N),
 \qquad E_{x_0}=X_N,\quad E_{z_0}=Z_N.
\]
For \(\alpha\in\mathcal A_N\), put \(\Lambda_\alpha=a(E_\alpha)\).
The compact sets \(\Lambda_\alpha\) and \(\Lambda_\beta\) are
disjoint whenever \(\alpha\ne\beta\) and
\(\{\alpha,\beta\}\ne\{x_0,z_0\}\).  Hence
\begin{equation}\label{eq:claim2-cluster-separation}
 \Delta:=
 \min_{\substack{\alpha\ne\beta\\
        \{\alpha,\beta\}\ne\{x_0,z_0\}}}
 \operatorname{dist}_{\mathbb S^1}
 \bigl(\Lambda_\alpha,\Lambda_\beta\bigr)>0.
\end{equation}
Set
\[
 \eta_a=\min\left\{\frac{\widetilde\varepsilon}{20},
                    \frac{\Delta}{20}\right\}.
\]

For each \(\alpha\in\mathcal A_N\), choose open arcs
\(J_\alpha^-\Subset J_\alpha^+\) such that
\begin{equation}\label{eq:claim2-label-arcs}
 \Lambda_\alpha\subset J_\alpha^-,\qquad
 \operatorname{diam}_{\mathbb S^1}(J_\alpha^+)
 <\frac{\widetilde\varepsilon}{10},
\end{equation}
and
\begin{equation}\label{eq:claim2-incompatible-arcs}
 \operatorname{dist}_{\mathbb S^1}
 \bigl(\overline{J_\alpha^+},\overline{J_\beta^+}\bigr)
 \ge\frac{\Delta}{2}
\end{equation}
for every pair occurring in the minimum in
\eqref{eq:claim2-cluster-separation}.  In view of
\eqref{eq:claim2-tail-smallness}, we may also arrange that
\begin{equation}\label{eq:claim2-common-label-arc}
 J_{x_0}^-=J_{z_0}^-=:J_\infty.
\end{equation}
Choose pairwise disjoint compact flow-box rectangles \(U_\alpha\)
such that
\begin{equation}\label{eq:claim2-neighbourhood-choice}
 E_\alpha\subset\operatorname{Int}U_\alpha,\qquad
 J_\alpha^-\subset a(\operatorname{Int}U_\alpha),\qquad
 a(U_\alpha)\subset J_\alpha^+,
 \qquad \diam (U_\alpha)<\frac{\varepsilon}{10},
 \qquad \diam q(U_\alpha)<\frac{\pi}{10}.
\end{equation}
They may be chosen so small that
\begin{equation}\label{eq:claim2-q-order}
 d_q:=\min_{U_{z_0}}q-\max_{U_{x_0}}q>0.
\end{equation}
Let
\[
 \mathcal U=
 \{U_{x_i},U_{y_i}:1\le i\le N\}
 \cup\{U_{x_0},U_{z_0}\}.
\]
This proves part \eqref{cl:singular-neighbourhoods:1}.

Define
\[
 R=D\setminus\left(V_\partial\cup V_0\cup
                    \bigcup_{U\in\mathcal U}U\right),
 \qquad K=\overline R.
\]
The compact set \(K\) contains no singularity and is contained in
\(\operatorname{Int}D\setminus\{0\}\).  Therefore
\[
 m:=\min_K\dot q>0,
 \qquad \ell:=\max_Kq-\min_Kq<\infty.
\]
Since \(a\) and \(q\) are smooth on the compact annulus
\[
 K_V=\{(r,\theta) : 1-r\ge\nu_\partial,\ r\ge\nu_0\}
\]
there exists \(\widehat L\ge1\) such that
\begin{equation}\label{eq:claim2-aq-Lipschitz}
 d_{\mathbb S^1}(a(p),a(p'))\le \widehat Ld(p,p'),
 \qquad |q(p)-q(p')|\le \widehat Ld(p,p')
\end{equation}
for all \(p,p'\in K_V\).

Fix \(T_0\ge1\) and put
\[
 B=2+\frac{2\ell}{mT_0}.
\]
Put \(\varrho=1-r\).  Since both \(r\) and \(\varrho\) are
\(1\)-Lipschitz, choose \(\delta_0>0\) so small that
\begin{equation}\label{eq:claim2-delta-choice}
 \delta_0\le
 \min\left\{1,\frac{mT_0}{2\widehat L},
                \frac{d_q}{4\widehat L}\right\},
\end{equation}
\begin{equation}\label{eq:claim2-smallness}
 \widehat LB\delta_0<\frac{\eta_a}{4},
 \qquad 4\widehat L\delta_0<\frac{\eta_a}{4},
\end{equation}
and
\begin{equation}\label{eq:claim2-collar-smallness}
\begin{aligned}
 \delta_0&<
 \min\{\nu_\partial,
        \dfrac{2(e^{T_0}-1)}{e^{T_0}+1}\nu_\partial\},\\
 \delta_0&<
 \min\{2(1-e^{-3T_0/4})\nu_0,
        (2-3e^{-3T_0/4})\nu_0,\nu_0\}.
\end{aligned}
\end{equation}

We first prove part \eqref{cl:singular-neighbourhoods:2}.  If
\([j_1,j_2)\) is as in that part and \(j_1\le i\le j_2-2\), then
both endpoints of the jump belong to \(K_V\).  Since
\(\dot q\ge m\) on the orbit segment,
\begin{equation}\label{eq:claim2-q-increase}
\begin{aligned}
 q(p_{i+1})-q(p_i)
 &=q(p_{i+1})-q(\phi_{t_i}(p_i))
   +q(\phi_{t_i}(p_i))-q(p_i)\\
 &\ge-\widehat L\delta+mT_0
 \ge\frac12mT_0>0.
\end{aligned}
\end{equation}
Thus the displayed sequence in part
\eqref{cl:singular-neighbourhoods:2} is strictly increasing.  Since
its values lie in an interval of length \(\ell\), the block contains
at most \(B\) orbit segments.  The function \(a\) is constant on
each orbit segment, so
\begin{equation}\label{eq:claim2-label-block}
 \sum_{i=j_1}^{j_2-2}
 d_{\mathbb S^1}\!\left(
 a(p_{i+1}),a(\phi_{t_i}(p_i))\right)
 \le \widehat LB\delta<\frac{\eta_a}{4}.
\end{equation}
This proves part \eqref{cl:singular-neighbourhoods:2}.

We next establish the order of visits in part
\eqref{cl:singular-neighbourhoods:3:a}.  On the two auxiliary neighborhoods $W_\partial, W_0$,
the there was no change of the vector field.  Hence
\[
 \dot\varrho=\varrho(1-\varrho)(2-\varrho)\ge\varrho
 \quad\text{on }W_\partial \quad \text{  and }
 \quad
 \dot r=-r(1-r^2)\le-\frac34r
 \quad\text{on }W_0.
\]
Consequently, if an orbit segment remains in \(W_\partial\), then
\[
\frac{d}{dt}\left(
e^{-t}\varrho(\phi_t(p_i))
\right)
=
e^{-t}\left(
\dot\varrho(\phi_t(p_i))
-
\varrho(\phi_t(p_i))
\right)
\ge0
\]
hence \(t\mapsto e^{-t}\varrho(\phi_t(p_i))\) is nondecreasing.
If we denote $f(t)=\varrho(\phi_t(p_i))$ then $f'(t)/f(t)\geq 1$ and so
$$
\log(\varrho(\phi_{t_i}(p_i)))-\log(\varrho(p_i))=\int_{0}^{t_i}\frac{f'(t)}{f(t)}dt\geq \int_{0}^{t_i} 1 dt=t_i.
$$
which gives
\begin{equation}\label{eq:claim2-boundary-growth}
 \varrho(\phi_t(p))\ge e^t\varrho(p).
\end{equation}
Similarly, an orbit segment contained in \(W_0\) satisfies
\begin{equation}\label{eq:claim2-sink-contraction}
 r(\phi_t(p))\le e^{-3t/4}r(p).
\end{equation}

Suppose \(p_i\notin V_\partial\).  If its orbit segment stays in
\(W_\partial\), then \eqref{eq:claim2-boundary-growth} and
\eqref{eq:claim2-collar-smallness} give
\[
 \varrho(p_{i+1})
 \ge e^{T_0}\varrho(p_i)-\delta
 >2\nu_\partial.
\]
If the segment leaves \(W_\partial\), or starts outside it, then
radial monotonicity gives
\(\varrho(\phi_{t_i}(p_i))\ge3\nu_\partial\), and the same
conclusion follows.  Induction proves that no later initial point
belongs to \(V_\partial\).

If a segment leaves \(V_\partial\) but the following jump returns
the next initial pseudo-orbit point to it, then
\(\varrho(p_{i+1})\ge2\nu_\partial-\delta\).  The next complete
segment therefore satisfies
\[
 \varrho(p_{i+2})
 \ge e^{T_0}(2\nu_\partial-\delta)-\delta
 >2\nu_\partial
\]
by \eqref{eq:claim2-collar-smallness}.  Thus the interpolated
pseudo-orbit can return across the inner boundary of $V_\partial$ at most once.

If \(p_i\in V_0\), then \eqref{eq:claim2-sink-contraction} gives
\[
 r(p_{i+1})
 \le2e^{-3T_0/4}\nu_0+\delta<2\nu_0.
\]
Thus every later initial point belongs to \(V_0\).  Moreover, once
an orbit segment meets \(V_0\), its next initial point lies in
\(W_0\). The following complete segment and
\eqref{eq:claim2-collar-smallness} place the next initial point in
\(V_0\).  Since \(\bigcup\mathcal U\) is disjoint from \(W_0\), no
visit to a member of \(\mathcal U\) can occur after a visit to
\(V_0\).  The analogous argument at the boundary shows that, after
a visit to a member of \(\mathcal U\), no visit to \(V_\partial\)
is possible.  This proves part
\eqref{cl:singular-neighbourhoods:3:a}, including its two assertions
about the interpolated pseudo-orbit.

For part \eqref{cl:singular-neighbourhoods:3:b}, let visits to
distinct sets \(U_\alpha,U_\beta\) be consecutive, and denote by
\(u_\alpha\in U_\alpha\) and \(u_\beta\in U_\beta\) the exit and
entry limit points.  The intervening part of the interpolated
pseudo-orbit meets neither collar, by the collar order just proved,
and meets no member of \(\mathcal U\), by consecutiveness.  Its
complete orbit segments therefore form a block in \(R\).  Formula
\eqref{eq:claim2-label-block}, together with at most two end jumps,
gives
\begin{equation}\label{eq:claim2-transition-drift}
 d_{\mathbb S^1}(a(u_\alpha),a(u_\beta))
 <\frac{\eta_a}{4}+2\widehat L\delta<\frac{\eta_a}{2}.
\end{equation}
For an incompatible pair, however,
\eqref{eq:claim2-incompatible-arcs} and
\eqref{eq:claim2-neighbourhood-choice} give the lower bound
\(\Delta/2\), a contradiction.  Thus the only possible distinct
pair is \(U_{x_0},U_{z_0}\).

This pair cannot occur in the reverse order.  Indeed, the same
decomposition, now applied to \(q\), and
\eqref{eq:claim2-q-increase} give
\[
 q(u_{x_0})-q(u_{z_0})\ge-2\widehat L\delta>-d_q,
\]
whereas \eqref{eq:claim2-q-order} gives
\(q(u_{x_0})-q(u_{z_0})\le-d_q\).  It follows that distinct visited
sets occur only in the order
\(U_{x_0}\) followed by \(U_{z_0}\), and there can be no return to
\(U_{x_0}\).  Together with the collar order, this proves all but
the final assertion of part
\eqref{cl:singular-neighbourhoods:3:b}.

For the remaining assertion, recall that every
\(U\in\mathcal U\) is a flow-box rectangle and that
\(\diam q(U)<\pi/10\).  A pseudo-orbit which leaves a
visit unbounded below exits through the positive transverse side of
\(U\). To begin a distinct visit unbounded above, it must later
enter through the negative transverse side.  A direct return across
the same box is excluded by the positive transverse direction and,
after reducing \(\delta_0\), by the positive distance between the
two transverse sides.  Any other return contains a full turn around
the origin, during which \(q\) increases by \(2\pi\).  Formula
\eqref{eq:claim2-q-increase}, together with the estimates for the two
end jumps, shows, after one further reduction of \(\delta_0\), that
the corresponding increase along the pseudo-orbit is greater than
\(\pi\).  This is impossible when both endpoints lie in a set of
\(q\)-diameter less than \(\pi/10\).  Thus the two terminal visits cannot be distinct, completing part
\eqref{cl:singular-neighbourhoods:3:b}.

It remains to prove part
\eqref{cl:singular-neighbourhoods:3:c}.  Suppose first that
\(\mathcal U_P=\emptyset\).  If the interpolated pseudo-orbit does
not meet \(R\), the required estimate is vacuous, and \(I_P\) may be
any nonempty open arc.  Otherwise, choose \(s\in\mathbb R\) such that
\(\gamma_P(s)\in R\), and put
\[
u:=\gamma_P(s).
\]
Let \(t\in\mathbb R\) be arbitrary subject to
\(\gamma_P(t)\in R\).  Interchanging \(s\) and \(t\), if necessary,
we may assume that \(s\le t\).

By part \eqref{cl:singular-neighbourhoods:3:a}, the portion of the
interpolated pseudo-orbit between the times \(s\) and \(t\) contains
at most two orbit segments involved in leaving \(V_\partial\) and
at most one orbit segment involved in entering \(V_0\).  All the
remaining complete orbit segments form a consecutive block
contained in \(R\).  Since \(a\) is constant along every orbit
segment, the exceptional segments themselves cause no change of
the orbit label.  The jumps within the block contribute at most
\(\widehat L B\delta\), by
\eqref{eq:claim2-label-block}.  Besides the jumps within this block, there are at most three further
jumps: the possible jump returning the pseudo-orbit to
\(V_\partial\), the jump following its final departure from
\(V_\partial\), and the jump preceding its entrance into \(V_0\).
Each of them changes \(a\) by at most \(\widehat L\delta\).

The buffer provided by \(W_\partial\) and \(W_0\), together with
\eqref{eq:claim2-collar-smallness}, ensures that both endpoints of
each of these additional jumps belong to \(K_V\).  Hence each such
jump changes \(a\) by at most \(\widehat L\delta\).  Consequently,
using \(\delta\le\delta_0\) and
\eqref{eq:claim2-smallness}, we obtain
\begin{equation}\label{eq:claim2-empty-label-bound}
		d_{\mathbb S^1}\bigl(a(u),a(\gamma_P(t))\bigr)
		\le \widehat L B\delta+4\widehat L\delta \le \widehat L B\delta_0+4\widehat L\delta_0 <\frac{\eta_a}{4}+\frac{\eta_a}{4}
		=\frac{\eta_a}{2}.
\end{equation}
Choose a nonempty open arc \(I_P\) about \(a(u)\) such that
\[
 d_{\mathbb S^1}(\widehat a,a(u))<\frac{\eta_a}{2}
 \qquad(\widehat a\in I_P).
\]
Then \eqref{eq:claim2-empty-label-bound} gives the required estimate.
If the interpolated pseudo-orbit does not meet \(R\),
take any nonempty open arc.

Suppose \(\mathcal U_P\ne\emptyset\), and let
\(\gamma_P(t)\in R\).  By the order already proved, the component
of \(\gamma_P^{-1}(R)\) containing \(t\) has an endpoint adjacent
to a visit to some \(U_\alpha\in\mathcal U_P\).  If
\(u_\alpha\in U_\alpha\) is the corresponding one-sided limit, then
the same block estimate gives
\begin{equation}\label{eq:claim2-label-to-neighbourhood}
 d_{\mathbb S^1}
 \bigl(a(\gamma_P(t)),a(u_\alpha)\bigr)
 <\frac{\eta_a}{4}+2\widehat L\delta<\frac{\eta_a}{2}.
\end{equation}
If \(\mathcal U_P=\{U_\alpha\}\), choose
\[
 \Lambda_\alpha\subset I_P\subset J_\alpha^-.
\]
If \(\mathcal U_P=\{U_{x_0},U_{z_0}\}\), choose
\[
 \Lambda_{x_0}\cup\Lambda_{z_0}
 \subset I_P\subset J_\infty.
\]
In either case \(I_P\subset a(\operatorname{Int}U)\) for every
\(U\in\mathcal U_P\).  Moreover, for \(\widehat a\in I_P\), both
\(a(u_\alpha)\) and \(\widehat a\) lie in \(J_\alpha^+\).  Therefore
\[
 d_{\mathbb S^1}
 \bigl(a(\gamma_P(t)),\widehat a\bigr)
 <\frac{\eta_a}{2}+\frac{\widetilde\varepsilon}{10}
 <\widetilde\varepsilon.
\]
Together with \eqref{eq:claim2-empty-label-bound}, this proves part
\eqref{cl:singular-neighbourhoods:3:c} and completes the proof.
\end{proof}

\begin{claim}\label{cl:oriented-corrected}
The flow \(\phi\) has the oriented shadowing property.
\end{claim}

\begin{proof}
Fix \(\varepsilon>0\).  We first record a direct estimate in the
orbit coordinates.  Since
\[
 r(q)=\frac{1}{\sqrt{1+e^{2q}}},\qquad
 \theta=a+q\pmod {2\pi},
\]
and \(\lvert r'(q)\rvert=r(q)(1-r(q)^2)\le1\), the mean value
theorem gives, for all
\(x,y\in\operatorname{Int}D\setminus\{0\}\),
\begin{equation}\label{eq:claim3-coordinate-distance}
\begin{aligned}
	d(x,y)
	&=|r_xe^{i\theta_x}-r_ye^{i\theta_y}|\le |r_x-r_y|
	+r_y|e^{i\theta_x}-e^{i\theta_y}|\\
	&\le |r_x-r_y|
	+d_{\mathbb S^1}(\theta_x,\theta_y)\\
&\le d_{\mathbb S^1}(a(x),a(y))
	+2|q(x)-q(y)|.
\end{aligned}
\end{equation}

Using the constant \(L\) fixed before Claim \ref{cl:box-passage}, put
$
 \varepsilon_1:=\min\left\{\varepsilon,
                    \frac{\varepsilon}{10L}\right\}.
$
Choose
$
 0<\widetilde\varepsilon<\frac{\varepsilon}{20}
$
and apply Claim \ref{cl:singular-neighbourhoods} with
\(\varepsilon_1\) and \(\widetilde\varepsilon\).  In particular,
for every \(U\in\mathcal U\),
\begin{equation}\label{eq:claim3-U-smallness}
 \diam (U)<\frac{\varepsilon}{10},\qquad
 \diam q(U)
 \le L\diam (U)<\frac{\varepsilon}{100}.
\end{equation}
Since \(q\) is smooth on the compact set
\[
 D\setminus
 \bigl(\operatorname{Int}V_\partial\cup
       \operatorname{Int}V_0\bigr),
\]
and \(\theta\) is smooth on the compact collar \(W_\partial\), there
exists \(\widehat L\ge1\) such that \(q\) is \(\widehat L\)-Lipschitz
on the first set and
\[
 d_{\mathbb S^1}(\theta(p),\theta(p'))
 \le \widehat Ld(p,p')
 \qquad \text{for }p,p'\in W_\partial.
\]
Choose \(A>0\) with \(AT_0>4L\), and let decrease $\delta_0$ when necessary, so that it is smaller than
\(\delta_0>0\) supplied by Claim
\ref{cl:box-passage}.  Decrease \(\delta_0\), if
necessary again, so that
\begin{equation}\label{eq:claim3-smallness}
 20\widehat L\delta_0<\frac{\varepsilon}{10},\qquad
 \widehat L\delta_0<
 \min\left\{\frac{T_0}{2},\frac{\varepsilon}{10}\right\},
 \qquad 
 \delta_0<\frac{\varepsilon}{10}.
\end{equation}

Fix \(0<\delta\le\delta_0\), and let
\(\mathcal P=((p_i,t_i))_{i\in\mathbb Z}\) be a
\(\delta\)-\(T_0\) pseudo-orbit.  Subdivide the segment of duration
\(t_i\) into
\[
 m_i:=\left\lfloor\frac{t_i}{T_0}\right\rfloor
\]
segments of duration \(t_i/m_i\).  Since
\[
 T_0\le\frac{t_i}{m_i}<2T_0,
\]
the refined sequence has all durations in \([T_0,2T_0)\). The
inserted jumps are zero and so the interpolated pseudo-orbit is still
\(\gamma_{\mathcal P}\).  We use this refinement below.
If \(\gamma_{\mathcal P}(\mathbb R)\subset V_0\), then the fixed
point \(0\) traces \(\mathcal P\) with \(h(t)=t\), because
\(\diam (V_0)<\varepsilon/10\).  We henceforth exclude
this case.

We next select the tracing orbit.  A visit to \(U\in\mathcal U\) is
called past terminal, respectively future terminal, if the
corresponding component of \(\gamma_{\mathcal P}^{-1}(U)\) is
unbounded below, respectively above.  By part \eqref{cl:singular-neighbourhoods:3:b} of Claim
\ref{cl:singular-neighbourhoods}, the visited sets form one of the choices
$
 \emptyset,\qquad \{U\},\qquad
 \{U_{x_0},U_{z_0}\},
$
and in the last case every visit to \(U_{x_0}\) precedes every visit
to \(U_{z_0}\).

The set \(E:=a(S)\) is countable.  If there is no terminal visit,
choose
\[
 \widehat a\in I_{\mathcal P}\setminus E
\]
and take the regular orbit with label \(\widehat a\).  If only one
set \(U\) is visited and a visit is terminal only in the past, only
in the future, or the same visit is terminal in both directions,
take, respectively, the outgoing branch, the incoming branch, or the fixed point of a
singularity in \(U\).  In a tail set \(U_{x_0}\) or \(U_{z_0}\), use
a sufficiently far member of the corresponding accumulating family.
The final assertion of part
\eqref{cl:singular-neighbourhoods:3:b} excludes distinct past-terminal and
future-terminal visits to the same \(U\).

Suppose now that both \(U_{x_0}\) and \(U_{z_0}\) are visited.  If
neither extreme visit is terminal, use the regular orbit just
described.  If the visit to \(U_{x_0}\) is past terminal but the
visit to \(U_{z_0}\) is not future terminal, use the outgoing branch
of \(x_n\) for sufficiently large \(n\). This branch meets
\(U_{z_0}\) at \(z_n=P(x_n)\).  If the visit to \(U_{z_0}\) is future
terminal but the visit to \(U_{x_0}\) is not past terminal, use the
incoming branch of \(y_n\) for sufficiently large \(n\), so that its preceding
intersection with \(J\) is \(P^{-1}(y_n)\in U_{x_0}\).  If both
extreme visits are terminal, use the trajectory
$$
B=W^+(x_0)\setminus\{x_0\}
   =W^-(z_0)\setminus\{z_0\}.
   $$
These cases are exhaustive.  Denote the selected orbit by
\(\mathcal O\).  It meets the visited sets in their prescribed order
and has the required asymptotic endpoint at every terminal visit.

If \(\mathcal O\) is fixed, its visit is terminal in both directions.
Hence the corresponding component of
\(\gamma_{\mathcal P}^{-1}(U)\) is all of \(\mathbb R\), so
\(\gamma_{\mathcal P}(\mathbb R)\subset U\).  The singular point in
\(U\) then traces \(\mathcal P\) with \(h(t)=t\), by
\eqref{eq:claim3-U-smallness}.  We may therefore assume that
\(\mathcal O\) is not a fixed point, and choose \(z\in\mathcal O\).  If \(\mathcal O\) is the regular orbit with label
\(\widehat a\in I_{\mathcal P}\), then part \eqref{cl:singular-neighbourhoods:3:c} of Claim
\ref{cl:singular-neighbourhoods} gives
\[
d_{\mathbb S^1}
\bigl(a(\gamma_{\mathcal P}(t)),a(z)\bigr)
<\widetilde\varepsilon
\qquad\text{whenever }\gamma_{\mathcal P}(t)\in R.
\]
Suppose instead that \(\mathcal O\) is an incoming or outgoing
singular branch, or the branch \(B\) from \(x_0\) to \(z_0\).
Choose any \(\widehat a\in I_{\mathcal P}\).  For every neighbourhood
\(U\in\mathcal U_{\mathcal P}\) met by \(\mathcal O\), both
\(a(z)\) and \(\widehat a\) belong to \(a(U)\).  Hence
$
d_{\mathbb S^1}\bigl(a(z),\widehat a\bigr)
<\frac{\widetilde\varepsilon}{10}.
$
Therefore, whenever \(\gamma_{\mathcal P}(t)\in R\),
\begin{equation}\label{eq:claim3-label-bound}
	d_{\mathbb S^1}
	\bigl(a(\gamma_{\mathcal P}(t)),a(z)\bigr)
   \le
	d_{\mathbb S^1}
	\bigl(a(\gamma_{\mathcal P}(t)),\widehat a\bigr)
	+
	d_{\mathbb S^1}\bigl(\widehat a,a(z)\bigr)<
	\widetilde\varepsilon+\frac{\widetilde\varepsilon}{10}
	<2\widetilde\varepsilon.
\end{equation}

We now construct the reparametrization on the interval between the
neighborhoods $V_0$ and $V_\partial$.  The case \(p_i\in V_\partial\) for every \(i\) will be
treated separately below.  Otherwise, part
\eqref{cl:singular-neighbourhoods:3:a} shows that the set
\(\{i:p_i\in V_\partial\}\) is either empty or has the largest element
\(i_\partial\).  In the latter case, let \(b_\partial\) be the first
time at or after \(S_{i_\partial}\) at which the interpolated
pseudo-orbit is outside \(V_\partial\). In the former case put
\(b_\partial=-\infty\).
Similarly, let \(b_0\) be the first time at which the interpolated
pseudo-orbit meets \(V_0\), and put \(b_0=+\infty\) if it never does.
The $V_0$ case in the proof of Claim
\ref{cl:singular-neighbourhoods} gives
\(\gamma_{\mathcal P}([b_0,+\infty))\subset W_0\).  Define
\[
 \mathcal T=[b_\partial,b_0],
\]
with an infinite endpoint case omitted at this point.  The estimates in the
proof of Claim \ref{cl:singular-neighbourhoods} show that
\(b_\partial<b_0\).  Apart from its endpoints, \(\mathcal T\) is
decomposed, in chronological order, into the components on which
\(\gamma_{\mathcal P}\) lies in \(R\) and the visits to members of
\(\mathcal U\).

We shall use the following consequence of Claims
\ref{cl:box-passage} and \ref{cl:singular-neighbourhoods}.  Each
\(U\in\mathcal U\) is contained in one of the two fixed boxes, so we denote
the corresponding box by \(Q(U)\).  Choose a compact flow-box rectangle \(U^-\)
such that
\[
 \operatorname{Sing}(\phi)\cap U
 \subset\operatorname{Int}U^-
 \subset U^-\subset\operatorname{Int}U.
\]
The set
\[
 \overline{U\setminus U^-}
\]
is compact and contains no singularity.  Hence \(\dot q\) has a
positive minimum there.  Since each
\(H_\delta^{Q(U)}\) decreases to the zero set of the vector field in
\(Q(U)\) as \(\delta\to0\), we may reduce \(\delta_0\) again, so that
\[
 H_\delta^{Q(U)}\cap U\subset\operatorname{Int}U^-,
 \qquad
 L\delta_0<
 \frac12\operatorname{dist}(U^-,D\setminus U).
\]
On the set \(U\setminus\operatorname{Int}U^-\), every orbit
segment lies outside \(H_\delta^{Q(U)}\).  Part
\eqref{cl:box-passage:1} gives a positive transverse advance at each
complete step there, while parts \eqref{cl:box-passage:2}--\eqref{cl:box-passage:4} bound the total number and total transverse
increase of the intervening jumps.  In particular, any pseudo-orbit can cross the
set  \(U\setminus\operatorname{Int}U^-\), only finitely many times.  It follows that the number of crossing of orbit segments with 
\(U\) is infinite, then the pseudo-orbit remains in \(U\) in one
time direction and the corresponding visit is terminal.

Denote
\[
 d_U:=\max_{U\in\mathcal U}\diam q(U)
 <\frac{\varepsilon}{100}.
\]
On a component \(\mathcal I\) of
\(\gamma_{\mathcal P}^{-1}(R)\), every internal jump satisfies
\begin{equation}\label{eq:claim3-q-jump}
 \left|q(p_{i+1})-
 q\bigl(\phi_{t_i}(p_i)\bigr)\right|\le \widehat L\delta.
\end{equation}
Furthermore, part {\rm(2)} of Claim
\ref{cl:singular-neighbourhoods} gives
\(q(p_{i+1})>q(p_i)\) and a positive lower bound, independent of
\(\delta\), for the \(q\)-increase on every complete orbit segment
contained in \(R\).

For every internal jump in \(\mathcal I\), choose disjoint one-sided
time intervals on which \(q\circ\gamma_{\mathcal P}\) changes by
\(2\widehat L\delta\).  This is possible after a further reduction of
\(\delta_0\).  Replace \(q\circ\gamma_{\mathcal P}\) on the union of
the two intervals by the affine function joining their outer endpoint
values.  The value at the right endpoint exceeds the value at the
left endpoint by at least
\[
 -\widehat L\delta+4\widehat L\delta=3\widehat L\delta.
\]
Thus the corrected function is continuous and strictly increasing on
\(\mathcal I\), and
\begin{equation}\label{eq:claim3-q-correction}
 \left|\widetilde q(t)-
 q\bigl(\gamma_{\mathcal P}(t)\bigr)\right|\le5\widehat L\delta
 \qquad(t\in\mathcal I)
\end{equation}
away from the two ends of \(\mathcal I\).

We choose the endpoint corrections and define \(\widetilde q\) on
the visits as follows.  If \(\mathcal J\) is a visit to
\(U\in\mathcal U\), then
\[
 q\bigl(\gamma_{\mathcal P}(\mathcal J)\bigr)
 \subset
 [\min q(U),\max q(U)].
\]
The values on the regular components immediately adjacent to
\(\mathcal J\) are within \(\widehat L\delta\) of this interval.  If two visits to
the same \(U\) are separated by a regular component, the corrected
values on that component increase from the \(\widehat L\delta\)-neighbourhood of
\(q(U)\) to the same \(\widehat L\delta\)-neighbourhood.  Hence all these values
belong to
\[
 \bigl(\min q(U)-2\widehat L\delta,
        \max q(U)+2\widehat L\delta\bigr).
\]
Choose successive endpoint values in this interval and join them by
continuous strictly increasing functions on the intervening visits.
The local finiteness proved above makes these choices compatible.  At
a terminal visit to a singularity \(s\), choose the values to converge
monotonically to \(q(s)\).

Only \(U_{x_0}\) and \(U_{z_0}\) can both occur.  They occur in this
order, and the definition of \(d_q\) in the proof of Claim
\ref{cl:singular-neighbourhoods} gives
\[
 \max q(U_{x_0})<\min q(U_{z_0}).
\]
Reducing \(\delta_0\) so that \(4\widehat L\delta_0<d_q\), the two enlarged intervals
used above are still disjoint and occur in the same order.  It follows
that the preceding local definitions combine into a continuous
strictly increasing function
\[
 \widetilde q\colon\mathcal T\longrightarrow\mathbb R
\]
satisfying
\begin{equation}\label{eq:claim3-global-q-bound}
 \left|\widetilde q(t)-
 q\bigl(\gamma_{\mathcal P}(t)\bigr)\right|
 \le d_U+5\widehat L\delta
 \qquad(t\in\mathcal T).
\end{equation}

On the selected nonfixed orbit, define
\[
 Q_z(u):=q(\phi_u(z)).
\]
This is continuous and strictly increasing.  Its image is
\(\mathbb R\) for a regular orbit,
\((-\infty,q(s))\) for an incoming branch ending at \(s\),
\((q(s),+\infty)\) for an outgoing branch starting at \(s\), and
\((q(x_0),q(z_0))\) for the branch \(B\).
In the branch cases, the endpoint values in the preceding
construction are chosen on the corresponding side of the singular
\(q\)-coordinate and converge to that coordinate at a terminal end.
Therefore
\[
 \widetilde q(\mathcal T)\subset Q_z(\mathbb R).
\]
We may now define, on the whole middle interval,
\begin{equation}\label{eq:claim3-global-time}
 H(t):=Q_z^{-1}(\widetilde q(t)),
 \qquad t\in\mathcal T.
\end{equation}
The function \(H\) is continuous and strictly increasing.  Moreover,
\[
 q\bigl(\phi_{H(t)}(z)\bigr)=\widetilde q(t).
\]
If \(\gamma_{\mathcal P}(t)\in R\), then
\eqref{eq:claim3-label-bound},
\eqref{eq:claim3-coordinate-distance}, and
\eqref{eq:claim3-global-q-bound} give
\[
\begin{aligned}
 d\bigl(\gamma_{\mathcal P}(t),\phi_{H(t)}(z)\bigr)
 &\le
 2\widetilde\varepsilon+2(d_U+5\widehat L\delta)\\
 &<\frac{\varepsilon}{5}.
\end{aligned}
\]
If \(\gamma_{\mathcal P}(t)\in U\in\mathcal U\), then
\(a(z)\) and \(a(\gamma_{\mathcal P}(t))\) belong to \(a(U)\).
Consequently,
\[
\begin{aligned}
 d\bigl(\gamma_{\mathcal P}(t),\phi_{H(t)}(z)\bigr)
 &\le
 \frac{\widetilde\varepsilon}{10}
 +2(d_U+5\widehat L\delta)\\
 &<\frac{\varepsilon}{5}.
\end{aligned}
\]
These estimates also apply on a terminal visit.  At such an end,
\(\widetilde q(t)\to q(s)\), and hence
\(H(t)\to-\infty\) on an outgoing branch and
\(H(t)\to+\infty\) on an incoming branch.  For the branch \(B\), the
corresponding limits hold at its two ends.

At an unbounded end of \(\mathcal T\) which is not terminal, the
uniform positiveadvance of values of $q$ on regular blocks implies
\(\widetilde q(t)\to-\infty\) in backward time, respectively
\(\widetilde q(t)\to+\infty\) in forward time.  Thus the same limits
hold for \(H(t)\) at every unbounded nonterminal end of
\(\mathcal T\).

It remains to extend \(H\) through the collars.  Suppose first that
\(b_\partial> -\infty\).  Every
earlier
complete segment has both endpoints within \(\delta\) of
\(V_\partial\), and radial monotonicity therefore keeps it in
\(W_\partial\).  We use these complete segments and truncate the
final one at \(b_\partial\).  Since the vector field change is absent on
\(W_\partial\), one has \(\dot\theta=1\) throughout this initial
part.  Choose successive lifts
\(\Sigma_i\in\mathbb R\) of the angular coordinates of the initial
points so that
\[
 \left|\Sigma_{i+1}-(\Sigma_i+t_i)\right|
 \le \widehat L\delta.
\]
Then
\[
 \Sigma_{i+1}-\Sigma_i
 \ge T_0-\widehat L\delta>\frac{T_0}{2}.
\]
On each complete step define
\[
 \widetilde\Sigma(S_i+u)
 :=\Sigma_i+\frac{u}{t_i}(\Sigma_{i+1}-\Sigma_i),
 \qquad 0\le u\le t_i.
\]
If the final segment exits \(V_\partial\) before its endpoint, use
the exact lifted phase up to the exit time. If the exit occurs at the
following jump, use the same affine formula on the complete segment.
These formulas define a continuous strictly increasing function on
\((-\infty,b_\partial]\), with
\begin{equation}\label{eq:claim3-phase-bound}
 d_{\mathbb S^1}\bigl(
 \widetilde\Sigma(t),\theta(\gamma_{\mathcal P}(t))\bigr)
 \le \widehat L\delta.
\end{equation}

Retain the point \(z\in\mathcal O\) selected above.  In choosing the
left endpoint value of \(\widetilde q\), require that
\(\phi_{H(b_\partial)}(z)\) be the point at which \(\mathcal O\) leaves
\(V_\partial\).  This is compatible with
\eqref{eq:claim3-global-q-bound}, since \(q\) depends only on \(r\). If
the pseudo-orbit leaves \(V_\partial\) at a jump, the resulting
error in \(q\) is at most \(\widehat L\delta\).

For \(t\le b_\partial\), put
\[
H(t):=
H(b_\partial)
+\widetilde\Sigma(t)-\widetilde\Sigma(b_\partial).
\]
Then both trajectories lie in \(W_\partial\) on this past interval,
and the selected orbit satisfies \(\dot\theta=1\).  At \(b_\partial\),
the relation \(\theta=a+q\) and the middle-interval estimates give
\[
d_{\mathbb S^1}\bigl(
\theta(\gamma_{\mathcal P}(b_\partial)),
\theta(\phi_{H(b_\partial)}(z))\bigr)
\le
2\widetilde\varepsilon+d_U+5\widehat L\delta.
\]
Consequently, for \(t\le b_\partial\),
\[
\begin{aligned}
	d_{\mathbb S^1}\bigl(
	\theta(\gamma_{\mathcal P}(t)),
	\theta(\phi_{H(t)}(z))\bigr)
	&\le
	2\widetilde\varepsilon+d_U+7\widehat L\delta.
\end{aligned}
\]
Together with the radial width of \(W_\partial\), this gives
\[
d\bigl(\gamma_{\mathcal P}(t),\phi_{H(t)}(z)\bigr)
<\frac{\varepsilon}{2}
\qquad(t\le b_\partial).
\]
Moreover, \(H(t)\to-\infty\) as \(t\to-\infty\).

If \(p_i\in V_\partial\) for every \(i\), the preceding radial
estimate gives \(\gamma_{\mathcal P}(\mathbb R)\subset W_\partial\).
Choose the lifts \(\Sigma_i\) and define
\(\widetilde\Sigma\colon\mathbb R\to\mathbb R\) by the same affine
formula on every step.  The estimates above show that
\(\widetilde\Sigma\) is a strictly increasing homeomorphism and that
\eqref{eq:claim3-phase-bound} holds on \(\mathbb R\).
Choose a point \(z\) on the boundary periodic orbit whose phase $\theta(z)$
agrees with $\theta(\gamma_{\mathcal P}(0))$, and put
\[
 H(t):=\widetilde\Sigma(t)-\widetilde\Sigma(0).
\]
Then \(H\) is an increasing homeomorphism and the same estimate
applies.

Finally, suppose that \(b_0<+\infty\).  Then the pseudo-orbit has no
future-terminal visit to a member of \(\mathcal U\).  By the choice of
the orbit \(\mathcal O\), it follows that
\[
\lim_{t\to+\infty}\phi_t(z)=0.
\]
Choose the endpoint value in
\eqref{eq:claim3-global-time} so that
\(\phi_{H(b_0)}(z)\in V_0\), and extend \(H\) from
\([b_0,+\infty)\) onto \([H(b_0),+\infty)\) by any increasing
homeomorphism.  The selected orbit remains in \(V_0\), while
\(\gamma_{\mathcal P}([b_0,+\infty))\subset W_0\).  Hence
\[
 d\bigl(\gamma_{\mathcal P}(t),\phi_{H(t)}(z)\bigr)
 <\frac{\varepsilon}{10}
 \qquad(t\ge b_0).
\]
The cases in which the whole pseudo-orbit lies in \(V_0\) or in
\(V_\partial\) have already been treated.  In every remaining case,
Claim \ref{cl:singular-neighbourhoods} gives
\(b_\partial<b_0\).  The definition on the middle interval
\(\mathcal T=[b_\partial,b_0]\) agrees at \(b_\partial\) with the
past trajectory extension and at \(b_0\) with the future trajectory extension.  Hence these
definitions combine into a continuous strictly increasing map
\(H\colon\mathbb R\to\mathbb R\).  The limits established above give
\[
\lim_{t\to-\infty}H(t)=-\infty,\qquad
\lim_{t\to+\infty}H(t)=+\infty,
\]
so \(H\) is an increasing homeomorphism of \(\mathbb R\).

The tracing error is less than \(\varepsilon/2\) on$V_\partial$
part, less than \(\varepsilon/5\) on \(\mathcal T\), and less than
\(\varepsilon/10\) on the $V_0$ part.  Therefore
\[
d\bigl(\gamma_{\mathcal P}(t),\phi_{H(t)}(z)\bigr)
<\varepsilon
\qquad(t\in\mathbb R).
\]
Finally, set
\[
 \bar z:=\phi_{H(0)}(z),\qquad
 h(t):=H(t)-H(0).
\]
Then \(h\in\operatorname{Rep}\) and
\(\phi_{h(t)}(\bar z)=\phi_{H(t)}(z)\) for every \(t\).  Hence every
\(\delta\)-\(T_0\) pseudo-orbit is \(\varepsilon\)-traced.  By the
equivalence of the \(\delta\)-\(T_0\) and \(\delta\)-\(1\)
definitions, \(\phi\) has the oriented shadowing property.
\end{proof}

\begin{claim}\label{cl:long-residence}
Let \(U\) be an open neighbourhood of \(J\) such that
\(\overline U\subset\hat Q_1\), and fix \(n\ge1\).  Let
\(\mathcal I_1,\ldots,\mathcal I_n,\mathcal I_\infty\subset\hat J\)
be pairwise disjoint compact intervals with
\(x_i\in\operatorname{Int}\mathcal I_i\), and suppose that
\[
 [x_0,x_{n+1}]
 \subset\operatorname{Int}\mathcal I_\infty.
\]
Assume that all their endpoints lie outside the set of zeros of \(g\).
Then there exists \(R_n<\infty\) such that
\[
 \phi_{[\alpha,\beta]}(p)\subset U,\qquad
 \beta-\alpha>R_n
 \quad\Longrightarrow\quad
 s_{\hat Q_1}(\phi_\alpha(p))\in
 \mathcal I_1\cup\cdots\cup\mathcal I_n\cup\mathcal I_\infty,
\]
for every orbit segment contained in \(U\).
\end{claim}

\begin{proof}
Put
$K_n=\hat J\setminus
 \operatorname{Int}\bigl(
 \mathcal I_1\cup\cdots\cup\mathcal I_n\cup\mathcal I_\infty
 \bigr).$
If \(K_n=\emptyset\), the conclusion is immediate.  Otherwise, the
interiors of the displayed intervals contain every zero of the
function \(g\).  Hence
\[
 m_n:=\min_{s\in K_n}g(s)>0.
\]
Since the transverse coordinate \(s_{\hat Q_1}\) is constant along
every orbit segment in \(\hat Q_1\), a segment whose transverse
coordinate belongs to \(K_n\) remains in \(U\) for at most the full crossing time of the box
\[
 R_n:=
 \int_{-\kappa}^{\kappa}
 \frac{d\tau}{1-b(\tau)(1-m_n)}<\infty.
\]
The stated implication follows.
\end{proof}

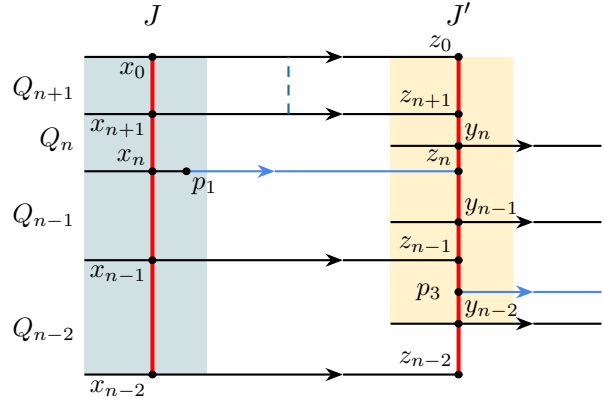
\begin{figure}[h]
        \centering
        \begin{tikzpicture}[yscale=0.6*1.4, xscale=0.6*1.5]
        	
        	\fill[boxblue] (-1.0, 0) rectangle (0.8, 5.0);
        	\fill[boxyellow] (3.5, 0.8) rectangle (5.3, 5.0);
        	
        	\draw[red, line width=1.5pt] (0, 0) -- (0, 5.0);
        	\draw[red, line width=1.5pt] (4.5, 0) -- (4.5, 5.0);
        	
        	\node[above=0.3cm] at (0, 5.0) {\large $J$};
        	\node[above=0.3cm] at (4.5, 5.0) {\large $J'$};
        	
        	\draw[thick, -{Stealth[scale=1.0]}] (-1.0, 5.0) -- (2.8, 5.0);
        	\draw[thick] (2.8, 5.0) -- (4.5, 5.0);
        	
        	\node[left] at (-1.0, 4.5) {$Q_{n+1}$};
        	\draw[thick, -{Stealth[scale=1.0]}] (-1.0, 4.1) -- (2.8, 4.1);
        	\draw[thick] (2.8, 4.1) -- (4.5, 4.1);
        	
        	\draw[thick, -{Stealth[scale=1.0]}] (3.5, 3.6) -- (5.6, 3.6);
        	\draw[thick] (5.6, 3.6) -- (6.6, 3.6);
        	
        	\node[left] at (-1.0, 3.7) {$Q_n$};
        	\draw[thick] (-1.0, 3.2) -- (0.5, 3.2);
        	
        	\draw[thick, arrowblue, -{Stealth[scale=1.0]}] (0.5, 3.2) -- (1.8, 3.2);
        	\draw[thick, arrowblue] (1.8, 3.2) -- (4.5, 3.2);
        	
        	\draw[thick, -{Stealth[scale=1.0]}] (3.5, 2.4) -- (5.6, 2.4);
        	\draw[thick] (5.6, 2.4) -- (6.6, 2.4);
        	
        	\node[left] at (-1.0, 2.5) {$Q_{n-1}$};
        	\draw[thick, -{Stealth[scale=1.0]}] (-1.0, 1.8) -- (2.8, 1.8);
        	\draw[thick] (2.8, 1.8) -- (4.5, 1.8);
        	
        	\draw[thick, arrowblue, -{Stealth[scale=1.0]}] (4.5, 1.3) -- (5.6, 1.3);
        	\draw[thick, arrowblue] (5.6, 1.3) -- (6.6, 1.3);
        	
        	\draw[thick, -{Stealth[scale=1.0]}] (3.5, 0.8) -- (5.6, 0.8);
        	\draw[thick] (5.6, 0.8) -- (6.6, 0.8);
        	
        	\node[left] at (-1.0, 0.7) {$Q_{n-2}$};
        	\draw[thick, -{Stealth[scale=1.0]}] (-1.0, 0.0) -- (2.8, 0.0);
        	\draw[thick] (2.8, 0.0) -- (4.5, 0.0);
        	
        	\draw[dashblue, thick, dash pattern=on 4pt off 3pt] (2.0, 4.1) -- (2.0, 5.0);
        	
        	\filldraw (0, 5.0) circle (1.5pt) node[anchor=north east, inner sep=2pt] {$x_0$};
        	\filldraw (0, 4.1) circle (1.5pt) node[anchor=north east, inner sep=2pt] {$x_{n+1}$};
        	\filldraw (0, 3.2) circle (1.5pt) node[anchor=south east, inner sep=2pt] {$x_n$};
        	\filldraw (0.5, 3.2) circle (1.5pt) node[anchor=north west, inner sep=2pt] {$p_1$};
        	\filldraw (0, 1.8) circle (1.5pt) node[anchor=north east, inner sep=2pt] {$x_{n-1}$};
        	\filldraw (0, 0.0) circle (1.5pt) node[anchor=north east, inner sep=2pt] {$x_{n-2}$};
        	
        	\filldraw (4.5, 5.0) circle (1.5pt) node[anchor=south east, inner sep=2pt] {$z_0$};
        	\filldraw (4.5, 4.1) circle (1.5pt) node[anchor=south east, inner sep=2pt] {$z_{n+1}$};
        	\filldraw (4.5, 3.6) circle (1.5pt) node[anchor=south west, inner sep=2pt] {$y_n$};
        	\filldraw (4.5, 3.2) circle (1.5pt) node[anchor=south east, inner sep=2pt] {$z_n$};
        	\filldraw (4.5, 2.4) circle (1.5pt) node[anchor=south west, inner sep=2pt] {$y_{n-1}$};
        	\filldraw (4.5, 1.8) circle (1.5pt) node[anchor=south east, inner sep=2pt] {$z_{n-1}$};
        	\filldraw (4.5, 1.3) circle (1.5pt) node[left=0.1cm] {$p_3$};
        	\filldraw (4.5, 0.8) circle (1.5pt) node[anchor=south west, inner sep=2pt] {$y_{n-2}$};
        	\filldraw (4.5, 0.0) circle (1.5pt) node[anchor=south east, inner sep=2pt] {$z_{n-2}$};
        	
        \end{tikzpicture}
        \caption{The idea behind the proof of Claim~\ref{cl:standard-corrected}}
    \label{fig:clam2}
    \end{figure}
    
\begin{claim}\label{cl:standard-corrected}
The flow \(\phi\) does not have the standard shadowing property.
\end{claim}

\begin{proof}
Choose once and for all
\begin{equation}\label{eq:standard-accuracy}
 0<\varepsilon_0<
 \min\left\{\frac1{10},\frac{\eta}{4}\right\}
\end{equation}
so small that the \(2\varepsilon_0\)-neighbourhood of
\(\overline{Q'_1}\) is disjoint from
\(\hat Q_1\cup Q''_1\cup W_0\cup W_\partial\), and so that
\(\overline{B_{\varepsilon_0}(J)}\subset\hat Q_1\).
Assume, for a contradiction, that standard shadowing holds, and let
\(\delta>0\) trace every \(\delta\)-\(1\) pseudo-orbit with accuracy
\(\varepsilon_0\) and a reparametrization in
\(\operatorname{Rep}(\varepsilon_0)\).

For \(m\ge1\), put
\[
 N_m=N_{\varepsilon_0}(\overline{Q'_m}),\qquad
 \hat{N}_m=N_{2\varepsilon_0}(\overline{Q'_m}).
\]
The time spent in \(\hat{N}_m\setminus\hat Q'_1\) is bounded uniformly
in \(m\).  Indeed, these sets are contained in the compact 
set
\[
 R_C=\overline{\hat{N}_1\setminus\hat Q'_1},
\]
and every orbit crosses \(R_C\) in the positive \(q\)-direction.
Put
\[
 \ell_c=\max_{R_C}q-\min_{R_C}q,\qquad
 \mu_C=\min_{R_C}\dot q>0,\qquad C=\frac{\ell_c}{\mu_C}.
\]
Since \(q\) is strictly increasing along every regular orbit, each
orbit spends at most \(C\) units of time in \(R_C\).  For an orbit
\(\mathcal O\), define
\[
 \Theta_m(\mathcal O)
 :=\operatorname{Leb}\{t\in\mathbb R:\phi_t(z)\in N_m\},
 \qquad z\in\mathcal O.
\]
This is independent of the choice of \(z\in\mathcal O\).  If
\(\mathcal O\) meets \(J'\) at a regular point with transverse
coordinate \(s\), then
\begin{equation}\label{eq:theta-corrected}
 \Theta_{n-1}(\mathcal O)\le T_2(s)+C,
 \qquad
 T_2(s)\le\Theta_{n-1}(\mathcal O)
 \quad\text{if }s\in[z_0,y_{n-1}].
\end{equation}
and \(\Theta_{n-1}(\mathcal O)=+\infty\) when the corresponding
branch is asymptotic to a singularity in
\(\overline{Q'_{n-1}}\).  If \(\mathcal O\) does not meet
\(N_{n-1}\), then \(\Theta_{n-1}(\mathcal O)=0\).

Choose \(n\) so large that
\begin{enumerate}
\item \(d(x_n,x_0)<\delta/4\) and
      \(d(z_n,y_n)<\delta/4\);
\item \(\widehat T_n\ge C\);
\item \(2^{n-1}\ge64\).
\end{enumerate}
Put
\begin{equation}\label{eq:Tprime-corrected}
 T'=4(\widehat T_n+C).
\end{equation}
On the gap \((y_n,z_n)\), the function \(T_2\) is continuous, equals
\(\widehat T_n\) at \(z_n\), and tends to \(+\infty\) at \(y_n\).
Choose
\begin{equation}\label{eq:p3-corrected}
 p_3\in(y_n,z_n),\qquad T'<T_2(p_3)<2T'.
\end{equation}
In particular \(d(p_3,y_n)<\delta\).

Choose an open neighbourhood \(U\) of \(J\) such that
\[
 \overline{B_{\varepsilon_0}(J)}\subset U\subset \overline U\subset\hat Q_1.
\]
Choose pairwise disjoint compact intervals
\(\mathcal I_i\subset\hat J\), with
\(x_i\in\operatorname{Int}\mathcal I_i\) for \(1\le i\le n\), so
small that
\begin{equation}\label{eq:bounded-crossing-neighborhood}
 s\in P(\mathcal I_i)
 \quad\Longrightarrow\quad
 T_2(s)\le
 \left(1+\frac{\varepsilon_0}{2}\right)T_2(z_i).
\end{equation}
By \eqref{eq:direct-slowdown}, every
\(s\in[z_0,z_{n+1}]\) with \(g(s)>0\) has crossing time at least
\(2^n\widehat T_n\).  Since the crossing time tends to \(+\infty\)
at \(z_0\) and is continuous at the regular point \(z_{n+1}\),
there is a compact interval
\(\mathcal I_\infty\subset\hat J\), disjoint from
\(\mathcal I_1,\ldots,\mathcal I_n\), such that
\([x_0,x_{n+1}]\subset\operatorname{Int}\mathcal I_\infty\) and
\begin{equation}\label{eq:tail-interval}
 s\in P(\mathcal I_\infty),\quad g(s)>0
 \quad\Longrightarrow\quad
 T_2(s)\ge2^{n-1}\widehat T_n.
\end{equation}
Choose all interval endpoints outside the zero set of \(g\), and let
\(R_n\) be supplied by Claim \ref{cl:long-residence}.  Finally,
choose \(T>0\) so that
\begin{equation}\label{eq:choice-of-T}
 (1-\varepsilon_0)T>R_n.
\end{equation}

We now define the promised bi-infinite pseudo-orbit completely.
Let \(W^-(x_n)\) be the branch coming from the boundary and
asymptotic to \(x_n\).  Choose \(p_{-1}\in W^-(x_n)\) so far forward
that
\[
 d(\phi_1(p_{-1}),x_n)<\delta.
\]
For \(i\le-2\), set \(t_i=1\) and choose
\(\phi_1(p_i)=p_{i+1}\).  Then \(p_i\) approaches the boundary as
\(i\to-\infty\).  At index \(-1\), take \(t_{-1}=1\), followed by the
jump to \(x_n\).

Put
\begin{equation}\label{eq:standard-pseudo-wait}
 p_0=x_n,\qquad t_0=T
\end{equation}
where \(T\) will be chosen below.
Let \(W^+(x_n)\setminus\{x_n\}\) be the outgoing branch.  Choose
\(p_1\in W^+(x_n)\setminus \{x_n\}\) with \(d(p_1,x_n)<\delta\), so close to \(x_n\)
that the time \(t_1\) determined by
\[
 \phi_{t_1}(p_1)=z_n
\]
satisfies \(t_1\ge1\).  Put \(p_2=y_n\) and \(t_2=T'\). The jump
from \(\phi_{t_1}(p_1)=z_n\) to \(p_2\), and the jump from
\(\phi_{t_2}(p_2)=y_n\) to \(p_3\), both have size less than
\(\delta\).  From \(p_3\)
onwards partition its genuine forward orbit into unit-time pieces:
\[
 t_i=1,\qquad p_{i+1}=\phi_1(p_i)\quad(i\ge3).
\]
This is a bi-infinite \(\delta\)-\(1\) pseudo-orbit, denoted by
\(\xi\).

Let \(\gamma_\xi\) be the interpolation associated with \(\xi\),
defined from its cumulative times as in part {\rm(3)} of Claim
\ref{cl:singular-neighbourhoods}.  Let \([c,d]\) be the maximal
interval for which \(\gamma_\xi(t)\) lies in
\(\overline{Q'_{n-1}}\).
The portion of the orbit through \(z_n\) takes at most time
\(\widehat T_n\), the constant part at \(y_n\) takes \(T'\), and
the orbit portion beginning at \(p_3\) takes less than
\(T_2(p_3)<2T'\).  Therefore
\begin{equation}\label{eq:dc-corrected}
 T'\le d-c\le\widehat T_n+T'+2T'
       \le\frac{13}{4}T'.
\end{equation}

Suppose that \(z\) strongly \(\varepsilon_0\)-traces \(\xi\) via
\(h\in\operatorname{Rep}(\varepsilon_0)\), and denote the orbit of
\(z\) by \(\mathcal O_z\).  Since
\(\phi_{h(t)}(z)\in N_{n-1}\) for every \(t\in[c,d]\),
\begin{equation}\label{eq:theta-lower-corrected}
 \Theta_{n-1}(\mathcal O_z)
 \ge h(d)-h(c)
 \ge(1-\varepsilon_0)(d-c)>\frac23T'.
\end{equation}
Before time \(c\) and after time \(d\), the interpolation follows
genuine orbit pieces.  Each such orbit spends at most \(C\) units
of time in \(\hat{N}_{n-1}\setminus\hat Q'_1\); hence
\[
 \gamma_\xi(t)\notin \hat{N}_{n-1}
 \qquad(t<c-C\text{ or }t>d+C).
\]
At those times its tracing orbit cannot belong to \(N_{n-1}\), because
the tracing error is at most \(\varepsilon_0\).  Since \(h\) is onto,
all times for which \(\phi_{h(t)}(z)\in N_{n-1}\) therefore lie in
\([c-C,d+C]\).  It follows that
\begin{equation}\label{eq:theta-upper-corrected}
 \Theta_{n-1}(\mathcal O_z)
 \le(1+\varepsilon_0)\bigl((d-c)+2C\bigr)
 \le5T'.
\end{equation}

For \(0\le t\le T\), one has
\(\gamma_\xi(t)=x_n\).
Thus \(\phi_{h(t)}(z)\in B_{\varepsilon_0}(x_n)\) for
\(0\le t\le T\), an actual time interval of length at least
\((1-\varepsilon_0)T\).  By \eqref{eq:choice-of-T}, this interval
has length greater than \(R_n\).

Since \(h(0)=0\) and \(h\in\operatorname{Rep}(\varepsilon_0)\),
\[
 \phi_{[0,h(T)]}(z)\subset U,\qquad
 h(T)\ge(1-\varepsilon_0)T>R_n.
\]
Claim \ref{cl:long-residence} therefore gives
\[
 s_{\hat Q_1}(z)\in
 \mathcal I_1\cup\cdots\cup\mathcal I_n\cup\mathcal I_\infty.
\]

If \(\Theta_{n-1}(\mathcal O_z)=0\), then
\eqref{eq:theta-lower-corrected} is contradicted and if
\(\Theta_{n-1}(\mathcal O_z)=+\infty\), then
\eqref{eq:theta-upper-corrected} is contradicted.  We may therefore
assume that \(0<\Theta_{n-1}(\mathcal O_z)<\infty\).  In particular,
\(\mathcal O_z\) reaches the second box and crosses it in finite
time.

If \(s_{\hat Q_1}(z)\in\mathcal I_i\) for some
\(1\le i\le n\), then \eqref{eq:theta-corrected},
\eqref{eq:bounded-crossing-neighborhood}, and the monotonicity of
\((\widehat T_i)\) give
\[
 \Theta_{n-1}(\mathcal O_z)
 \le(1+\varepsilon_0/2)\widehat T_n+C
 \le\frac76(\widehat T_n+C)
 =\frac7{24}T'
 <\frac23T',
\]
contradicting \eqref{eq:theta-lower-corrected}.  This also covers a
branch issuing from a singularity \(x_i\).

If \(s_{\hat Q_1}(z)\in\mathcal I_\infty\), then
\eqref{eq:tail-interval} gives the following bound when the
corresponding orbit portion in the second box is regular:
\[
 \Theta_{n-1}(\mathcal O_z)\ge2^{n-1}\widehat T_n
 \ge64\widehat T_n
 \ge32(\widehat T_n+C)=8T'>5T',
\]
contradicting \eqref{eq:theta-upper-corrected}.  If that second-box
orbit portion is asymptotic to a singularity in the second box,
the tracing orbit remains in \(N_{n-1}\) for infinite forward time,
which contradicts \eqref{eq:theta-upper-corrected} directly.  The
finite intervals and the tail interval exhaust all possible tracing
orbits.  Hence \(\xi\) is not strongly
\(\varepsilon_0\)-traced, contradicting standard shadowing.
\end{proof}

The proof is completed by combining Claim~\ref{cl:oriented-corrected} with Claim~\ref{cl:standard-corrected}
\end{proof}

We observe that the previous construction can be carried out straightforwardly in $A$ as follows. We begin by considering the initial flow $\phi': \mathbb{R} \times A \to A$ satisfying the following properties:
\begin{enumerate}
    \item The circle $O_1 = \{ x \in \mathbb{R} : \|x\| = 1 \}$ is a hyperbolic attracting periodic orbit for $\phi'$.
    \item The circle $O_2 = \{ x \in \mathbb{R} : \|x\| = 2 \}$ is a hyperbolic repelling periodic orbit for $\phi'$.
    \item Every point $x \notin O_1 \cup O_2$ converges to $O_1$ in forward time and to $O_2$ in backward time, while spiraling around $O_1$ with uniform angular speed.
\end{enumerate}
From this point onward, the construction is completely analogous to the one carried out in Theorem \ref{thm:disc}. We thus obtain the following result.

\begin{theorem}\label{thm:annulus}
There is a $C^\infty$-flow in  $A$ with the oriented shadowing property, but without the standard shadowing property.
\end{theorem}

\begin{remark}\label{rmk:annulus}
    We would like to emphasize that we have chosen to present the constructions in $D$ and $A$ only for simplicity. In fact, our methods work equally well in arbitrarily small discs, arbitrarily small annuli, or arbitrarily thin annuli. Additionally, by modifying the initial flows, one can easily obtain  flows with exactly same properties, but  with the attracting and repelling roles reversed.
\end{remark}

We are now in a position to provide the reader with a proof for Theorem \ref{thm:denseflows}.

\begin{proof}[Proof of Theorem \ref{thm:denseflows}]

We first construct one example. Choose a $C^\infty$ Morse--Smale flow
$\psi'$ on $M$ with a hyperbolic sink $\sigma$; see
\cite[Chapter~4]{Palis}. By a $C^\infty$ modification supported in an
arbitrarily small neighborhood of $\sigma$, we may arrange that
$\sigma$ is surrounded by a hyperbolic repelling periodic orbit $O$.
Thus $O$ bounds a closed disc $D$ containing $\sigma$, and the vector
field points away from $O$ on both sides. We further choose the
modification so that, in smooth coordinates on a collar of $O$, the
restriction $\psi'|_D$ agrees with the boundary collar of the base
flow used in Theorem~\ref{thm:disc}. The modification introduces only
hyperbolic critical elements and may be performed without creating
saddle connections. Consequently, the resulting flow is still
Morse--Smale and hence has the standard shadowing property by
\cite[Theorem~1.2]{Mu1}. This local configuration is depicted in
Figure~\ref{morse-smale}.

 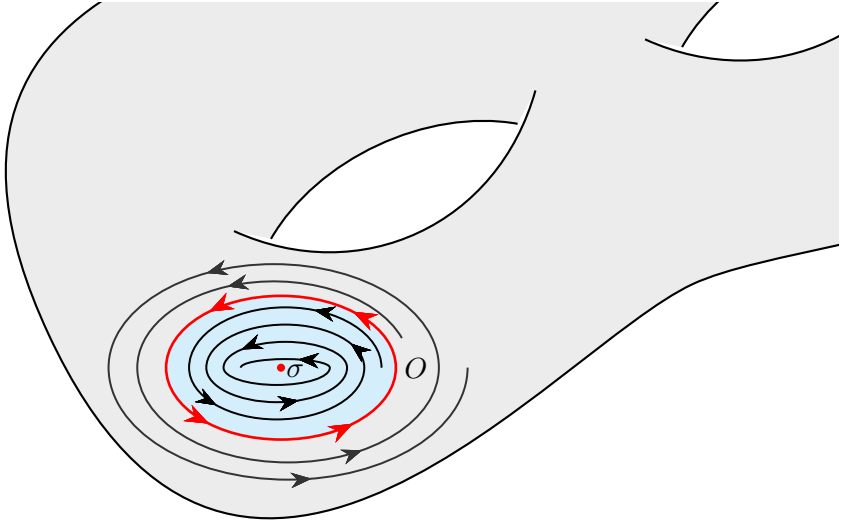
\begin{figure}[h]
        \centering
        \begin{center}
\begin{tikzpicture}[rotate=25, yscale=2.0, xscale=2.0]

    \clip[rotate around={-25:(0,0)}] (-3.95, -2.95) rectangle (1.6, 0.5);

    \draw[thick, fill=gray!15] 
        (-4, 0) .. controls (-4, 3) and (-1, 1.5) .. (0, 1.5)  .. controls (1, 1.5) and (4, 3) .. (4, 0)      .. controls (4, -3) and (1, -1.5) .. (0, -1.5) .. controls (-1, -1.5) and (-4, -3) .. (-4, 0);
    \fill[white] 
        (-2.6, 0.1) .. controls (-2.0, -0.6) and (-1.0, -0.6) .. (-0.4, 0.1) -- (-0.6, -0.05) .. controls (-1.1, 0.3) and (-1.9, 0.3) .. (-2.4, -0.05) -- cycle; 
    \draw[thick] (-2.6, 0.1) .. controls (-2.0, -0.6) and (-1.0, -0.6) .. (-0.4, 0.1);
    \draw[thick] (-2.4, -0.05) .. controls (-1.9, 0.3) and (-1.1, 0.3) .. (-0.6, -0.05);

    \fill[white] 
        (0.4, 0.1) .. controls (1.0, -0.6) and (2.0, -0.6) .. (2.6, 0.1)
        -- (2.4, -0.05) .. controls (1.9, 0.3) and (1.1, 0.3) .. (0.6, -0.05)
        -- cycle;

    \draw[thick] (0.4, 0.1) .. controls (1.0, -0.6) and (2.0, -0.6) .. (2.6, 0.1);
    \draw[thick] (0.6, -0.05) .. controls (1.1, 0.3) and (1.9, 0.3) .. (2.4, -0.05);

\begin{scope}[shift={(-2.7, -0.85)}, scale=0.19, rotate=-25]

    \fill[cyan!15] (0, 0) ellipse (4 and 2.5);
    \fill[pattern=none, pattern color=blue!60!black] (0, 0) ellipse (4 and 2.5);

    \draw[red, line width=1pt, 
          mid arrow=0.125, mid arrow=0.375, 
          mid arrow=0.625, mid arrow=0.875] 
          (4,0) arc (0:360:4 and 2.5);

    \draw[line width=0.75pt, color=black, 
          mid arrow=0.08, mid arrow=0.25, mid arrow=0.42, 
          mid arrow=0.65, mid arrow=0.8, mid arrow=0.95
          ]
        (3.5, 0) 
        arc (0:180:3.35 and 2.1)   arc (180:360:3.05 and 1.8) arc (0:180:2.75 and 1.5)   arc (180:360:2.45 and 1.2) arc (0:180:2.15 and 0.9)   arc (180:360:1.85 and 0.6) arc (0:180:1.55 and 0.3);  
    \fill[red] (0, 0) circle (4pt);

    \node at (4.7, 0) {\large $O$};
    \node at (0.5, -0.1) {\large $\sigma$};

    \draw[line width=0.75pt, color=black!80, 
          mid arrow=0.12, mid arrow=0.37, mid arrow=0.62, mid arrow=0.87]
        (4.21, 1.03) 
        arc (20:180:4.75 and 3.0)  arc (180:360:5.25 and 3.3) arc (0:180:5.75 and 3.6)
        arc (180:360:6.25 and 3.9);
\end{scope}
\end{tikzpicture}
\end{center}
        \caption{The flow $\psi'$, The orbit $O$ and the singularity $\sigma$ depicted in red.}
       \label{morse-smale}
    \end{figure}
The construction in Theorem~\ref{thm:disc} changes only the speed of
the base flow, as expressed in \eqref{eq:time-changed-coordinates},
and its multiplier $\rho$ is identically one on a collar of the
boundary. After transporting that construction to $D$, we therefore
obtain a $C^\infty$ flow $\phi$ on $D$ which agrees with $\psi'|_D$
on a neighborhood of $O$, has the oriented shadowing property, and
does not have the standard shadowing property. Define
$\psi\colon\mathbb R\times M\to M$ by
\begin{equation}\label{eq:surface-insertion}
 \psi_t(x)=
 \begin{cases}
  \phi_t(x),&x\in D,\\
  \psi'_t(x),&x\in M\setminus\operatorname{Int}D.
 \end{cases}
\end{equation}
The two definitions agree on a collar of $O$, so \eqref{eq:surface-insertion}
defines a $C^\infty$ flow. Both
\[
 D\qquad\text{and}\qquad E:=M\setminus\operatorname{Int}D
\]
are invariant, and their common boundary is $O$.

We prove that $\psi$ has the oriented shadowing property. The
restriction $\psi|_D$ has oriented shadowing by
Theorem~\ref{thm:disc}. The restriction $\psi|_E=\psi'|_E$ has the
standard shadowing property. Indeed, $\psi|_E$ is a Morse--Smale flow
on a compact invariant surface whose boundary is the hyperbolic
periodic orbit $O$. The proof of the shadowing theorem for
Morse--Smale flows applies verbatim in this setting: the spectral
decomposition consists of the critical elements of $\psi'$ contained
in $E$, including $O$, and all local tracing constructions take place
in invariant neighborhoods of these elements. Thus the conclusion
follows by the same argument as \cite[Theorem~1.2]{Mu1}.

Fix $\varepsilon>0$ and choose
\begin{equation}\label{eq:pasting-accuracy}
 0<\varepsilon_1<\frac{\varepsilon}{3}.
\end{equation}
Let $\delta_D$ and $\delta_E$ be oriented shadowing constants for
$\psi|_D$ and $\psi|_E$, respectively, with tracing accuracy
$\varepsilon_1$, and put
\begin{equation}\label{eq:pasting-shadowing-constant}
 \delta_0=\min\{\delta_D,\delta_E\}.
\end{equation}

Because $O$ is a hyperbolic repeller, it has an arbitrarily thin
closed repelling collar $N$. We choose $N$ so that there is a
continuous projection $\pi\colon N\to O$ along its transverse fibers
satisfying
\begin{equation}\label{eq:collar-projection-estimates}
 \begin{split}
 d(x,\pi(x))&<\min\left\{\varepsilon_1,\frac{\delta_0}{3}\right\},\\
 d\bigl(\psi_t(x),\psi_t(\pi(x))\bigr)
 &<\min\left\{\varepsilon_1,\frac{\delta_0}{3}\right\}
 \qquad(0\le t\le2)
 \end{split}
\end{equation}
for every $x\in N$. Shrinking $N$ once more, if necessary, the
repelling property gives
\begin{equation}\label{eq:repelling-collar-separation}
 a_N:=\inf\left\{
 d\bigl(\psi_t(x),N\bigr):
 x\in M\setminus\operatorname{Int}N, 1\le t\le2
 \right\}>0.
\end{equation}
Choose
\begin{equation}\label{eq:pasting-pseudo-orbit-constant}
 0<\gamma<\min\left\{a_N,\frac{\delta_0}{3}\right\}.
\end{equation}

Let $P=(x_i,t_i)_{i\in\mathbb Z}$ be a
$\gamma$-$1$-pseudo-orbit for $\psi$. We may assume that
$1\le t_i<2$. It follows from
\eqref{eq:repelling-collar-separation} and
\eqref{eq:pasting-pseudo-orbit-constant} that once $P$ has a term in
$M\setminus\operatorname{Int}N$, all its subsequent terms remain
outside $N$ and in the same component of $M\setminus O$. Consequently,
if $P$ has terms outside $N$, all of them lie on the same side of
$O$. Denote the corresponding invariant side by $S$, where $S=D$ or
$S=E$. If every term of $P$ belongs to $N$, set $S=D$.

Define
\begin{equation}\label{eq:projected-pseudo-orbit}
 y_i=
 \begin{cases}
  x_i,&x_i\in S,\\
  \pi(x_i),&x_i\notin S.
 \end{cases}
\end{equation}
In the case in which every $x_i$ belongs to $N$, we instead put
$y_i=\pi(x_i)$ for all $i$. In either case $y_i\in S$. Moreover,
\eqref{eq:collar-projection-estimates} and
\eqref{eq:pasting-pseudo-orbit-constant} give
\begin{align}
 d\bigl(\psi_{t_i}(y_i),y_{i+1}\bigr)
 &\le d\bigl(\psi_{t_i}(y_i),\psi_{t_i}(x_i)\bigr)
   +d\bigl(\psi_{t_i}(x_i),x_{i+1}\bigr)
   +d(x_{i+1},y_{i+1}) \notag\\
 &<\delta_0.\label{eq:projected-pseudo-orbit-error}
\end{align}
Thus $(y_i,t_i)_{i\in\mathbb Z}$ is a
$\delta_0$-$1$-pseudo-orbit contained in $S$. The same estimates,
applied for every $0\le t\le t_i$, show that the interpolations of
$(x_i,t_i)$ and $(y_i,t_i)$ remain $\varepsilon_1$-close. By the
choice of $\delta_0$, the latter pseudo-orbit is
$\varepsilon_1$-traced by an orbit in $S$, with an increasing
reparametrization. It follows from \eqref{eq:pasting-accuracy} that
the original pseudo-orbit is $2\varepsilon_1$-traced and hence is
$\varepsilon$-traced. Therefore $\psi$ has the oriented shadowing
property.

We now show that $\psi$ does not have the standard shadowing property.
Let $\varepsilon_0$ be the number fixed in
\eqref{eq:standard-accuracy}. For each $\delta>0$, let $\xi^\delta$
be the $\delta$-$1$-pseudo-orbit constructed in the proof of
Claim~\ref{cl:standard-corrected}, with the choices made in
\eqref{eq:Tprime-corrected}, \eqref{eq:p3-corrected}, and
\eqref{eq:choice-of-T}. The alternatives
\eqref{eq:bounded-crossing-neighborhood} and
\eqref{eq:tail-interval}, together with
\eqref{eq:theta-lower-corrected} and
\eqref{eq:theta-upper-corrected}, show that $\xi^\delta$ is not
strongly $\varepsilon_0$-traced by any orbit of $\phi$. Since the
compact interval $J$ lies in
$\operatorname{Int}D$, define
\begin{equation}\label{eq:global-nonshadowing-accuracy}
 \widehat\varepsilon=
 \min\left\{\varepsilon_0,\frac12d(J,O)\right\}>0.
\end{equation}

Suppose that $\xi^\delta$ were strongly $\widehat\varepsilon$-traced
for $\psi$ by a point $z\in M$. By
\eqref{eq:standard-pseudo-wait}, its zeroth term is
$p_0=x_n\in J$. Since every reparametrization under consideration
satisfies $h(0)=0$, tracing at time zero and
\eqref{eq:global-nonshadowing-accuracy} imply that
$z\in\operatorname{Int}D$. The invariance of $D$ then gives
$\psi_t(z)\in D$ for every $t\in\mathbb R$. By
\eqref{eq:surface-insertion}, the orbit of $z$ is an orbit of $\phi$
and strongly $\widehat\varepsilon$-traces $\xi^\delta$. Since
$\widehat\varepsilon\le\varepsilon_0$, we have
$\operatorname{Rep}(\widehat\varepsilon)\subset
\operatorname{Rep}(\varepsilon_0)$, and the same orbit would strongly
$\varepsilon_0$-trace $\xi^\delta$. This contradicts
Claim~\ref{cl:standard-corrected}. Hence $\psi$ does not have the
standard shadowing property.

It remains to prove the density assertion. Let $X$ be a $C^1$ vector
field on $M$ with a zero $\sigma$, and let $\mathcal U$ be a
$C^0$-neighborhood of $X$. Choose $\eta_0>0$ such that
\begin{equation}\label{eq:density-neighborhood}
 \|Y-X\|_{C^0}<\eta_0\quad\Longrightarrow\quad Y\in\mathcal U.
\end{equation}
Since $X(\sigma)=0$, there is a coordinate disc $B$ centered at
$\sigma$ such that
\begin{equation}\label{eq:density-small-vector-field}
 \|X(x)\|<\frac{\eta_0}{8}
 \qquad(x\in B).
\end{equation}
Choose closed coordinate discs
\[
 B_0\subset\operatorname{Int}B_1
 \subset B_1\subset\operatorname{Int}B.
\]

On $B_1$ take a smooth vector field whose restriction to $B_0$ is a
rescaled copy of the local Morse--Smale configuration used in the
first part of the proof: a hyperbolic sink
surrounded by a hyperbolic repelling periodic orbit $O$, with $O$
bounding a disc $D\subset\operatorname{Int}B_0$. Spatial rescaling
and a constant change of speed allow the vector field $V$ on $B_1$
to be chosen so that
\begin{equation}\label{eq:density-small-local-model}
 \|V\|_{C^0}<\frac{\eta_0}{8}.
\end{equation}
We arrange that $V|_D$ is the transported base vector field used in
Theorem~\ref{thm:disc}.

Choose a smooth function $\chi\colon M\to[0,1]$ supported in
$\operatorname{Int}B_1$ and identically one on a neighborhood of
$B_0$, and define
\[
 X_0=\chi V+(1-\chi)X\quad\text{on }B_1,
 \qquad X_0=X\quad\text{on }M\setminus B_1.
\]
In view of \eqref{eq:density-small-vector-field} and
\eqref{eq:density-small-local-model}, this $C^1$ vector field
satisfies
\begin{equation}\label{eq:density-local-interpolation}
 X_0=V\text{ near }B_0,\qquad
 X_0=X\text{ on }M\setminus B_1,\qquad
 \|X_0-X\|_{C^0}<\frac{\eta_0}{3}.
\end{equation}
The interpolation may create additional or non-hyperbolic zeros in
$B\setminus B_0$. No separate index adjustment is needed: the
relative perturbation in the next paragraph is performed on this
transition region while the prescribed configuration in $B_0$ is
kept fixed.

Apply smooth approximation relative to a neighborhood of $B_0$ and
then the Morse--Smale density theorem on the complement of that
neighborhood; see \cite[Chapter~4]{Palis}. The perturbations in the
surface proof are local and can therefore be supported away from the
protected neighborhood of $B_0$. Combining this with
\eqref{eq:density-local-interpolation}, we obtain a $C^\infty$
Morse--Smale vector field $Y$ satisfying
\begin{equation}\label{eq:density-relative-morse-smale}
 Y=V\text{ near }B_0,\qquad
 \|Y-X\|_{C^0}<\frac{2\eta_0}{3}.
\end{equation}
In particular, the sink, the repelling periodic orbit $O$, the disc
$D$, and the prescribed collar remain unchanged.

Finally, replace $Y|_D$ by the slowed-down model of
Theorem~\ref{thm:disc}. By \eqref{eq:time-changed-coordinates}, its
generating vector field on $D$ is $\rho V$, where $0\le\rho\le1$ and
$\rho\equiv1$ on the collar of $O$. It therefore glues smoothly to
$Y$ on $M\setminus D$, and
\begin{equation}\label{eq:density-final-perturbation}
 \|\rho V-V\|_{C^0}
 \le\|V\|_{C^0}<\frac{\eta_0}{8}.
\end{equation}
Let $Z$ denote the resulting $C^\infty$ vector field.
Equations~\eqref{eq:density-relative-morse-smale} and
\eqref{eq:density-final-perturbation} give
\begin{equation}\label{eq:density-final-estimate}
 \|Z-X\|_{C^0}
 \le\|Z-Y\|_{C^0}+\|Y-X\|_{C^0}
 <\left(\frac18+\frac23\right)\eta_0<\eta_0.
\end{equation}
Thus $Z\in\mathcal U$ by \eqref{eq:density-neighborhood}.

The argument based on
\eqref{eq:collar-projection-estimates}--\eqref{eq:projected-pseudo-orbit-error}
shows that the flow generated by $Z$ has the oriented shadowing
property. The argument based on
\eqref{eq:standard-pseudo-wait} and
\eqref{eq:global-nonshadowing-accuracy} shows that it does not have
the standard shadowing property. Since $X$ and $\mathcal U$ were
arbitrary, the required $C^\infty$ flows are $C^0$-dense among the
$C^1$ flows possessing a singularity.
\end{proof}

\begin{corollary}\label{cor:dense-all-surface}
Let $M$ be a compact, boundaryless and oriented surface with
$\chi(M)\neq0$. Then the set of $C^\infty$ flows on $M$ with the
oriented shadowing property but without the standard shadowing
property is $C^0$-dense in the space of all $C^1$ flows on $M$, where
the $C^0$ topology is induced by their generating vector fields.
\end{corollary}

\begin{proof}
By the Poincar\'e--Hopf theorem, every $C^1$ vector field on $M$ has a
zero and hence its flow has a singularity. The conclusion therefore
follows directly from the density assertion in
Theorem~\ref{thm:denseflows}.
\end{proof}

\section{Regularity of Reparametrizations}\label{sec: regularity}

This section is devoted to studying the regularity of the reparametrizations involved in shadowing. In particular, we show that for a flow with the oriented shadowing property, one may assume strong tracing of pseudo-orbits, provided they remain away from the set of singularities. Before delving into this analysis, we first lay the groundwork by obtaining some properties of cross-sections at regular points of the flow.

 \begin{lemma}\label{lemma_t_0}
Let $U\subset X$ be a non-empty open subset such that 
$\overline{U}\cap \Sing(\phi)=\emptyset$. Then there exist 
$t_0>0$ and $\delta>0$ such that:

\begin{enumerate}
\item Every $x\in \overline{U}$ admits a cross-section $S_x$ 
of time $t_0$ such that 
$B_{\frac{\delta}{2}}(x)\subset \phi_{(-t_0,t_0)}(S_x)$.

\item For every $x\in \overline{U}$,
\[
B_{\delta/2}(x)\subset \phi_{[-t_0,t_0]}(S_x).
\]

\item Any periodic orbit intersecting $\overline{U}$ has period 
greater than $t_0$.
\end{enumerate}
\end{lemma}

\begin{proof}
Since $\overline{U}$ is compact and disjoint from $\Sing(\phi)$, there exists $\eta>0$ such that
\[
d(\overline{U},\Sing(\phi))>\eta.
\]
By Lemma~\ref{lemma_cross-sections}, for each $x\in\overline{U}$ there exist $t_x>0$ and a local cross-section $S'_x$ through $x$ of time $t_x$.  Observe that for any $x\in \overline{U}$ and every $0<t<t_x$, $S_x$ is also a cross-section of time $t$. Lets start by proving the existence of $t_0$ in item (1).   Assume by contradiction that no $t_0>0$ can satisfy the Lemma. Hence, for every $n>0$, there is a point $x_n\in \overline{U}$ and a decreasing sequence $t_n\to 0$ such that such that:

\begin{equation}\label{condition:cross-sec}
    S_{x_n} \textrm{ is a cross-section of time } t_n, \textrm{ but not of time } t , \textrm{ for every } t>t_n. 
\end{equation}

Assume, by compactness of $\overline U$, that $x_n\to x\in \overline{U}$ and observe that $S_x$ is a cross-section of time $t_x>0$. Moreover, $\phi_{[-t_x,t_x]}(S_x)$ is a closed neighborhood of $x$. Fix $\eps>0$ such that $$B_{2\eps}(x)\subset \phi_{[-\frac{t_x}{3},\frac{t_x}{3}]}(S_x)$$ and $n_0>0$ such that  $x_n\in B_{\eps}(x)$ and $t_{x_n}\leq \frac{t_x}{3}$, for any $n\geq n_0$.  So, by taking $n>n_0$ and considering $S'_{x_n}=S_{x_n}\cap \overline{B_{\eps}(x)}$, we conclude that $S'_{x_n}$ is a cross-section through $x_n$ of time at least $\frac{t_x}{3}$. Indeed, this is because $S'_{x_n}\subset \phi_{[-\frac{t_x}{3},\frac{t_x}{3}]}(S_x)\subset \phi_{[t_x,t_x]}(S_x)$. But this contradicts \eqref{condition:cross-sec}, proving the existence of $t_0$.  

Now, observe that since $\overline U$ is compact and the family $$\{\phi_{[-t_0,t_0]}(S_x); x\in \overline{U}\}$$
forms a cover for $\overline U$ formed by closed neighborhoods. So, the compactness of $\overline{U}$ gives $\delta>0$ such that $$B_{\frac{\delta}{2}}(x)\subset \phi_{(-t_0,t_0)}(S_x),$$ for every $x\in \overline{U}$.  This simultaneously completes the proofs of items (1) and (2).

Finally, let $p\in\overline{U}$ be a periodic point with period $\pi(p)$. If $\pi(p)\le t_0$, then the orbit segment $\{\phi_t(p): t\in[-t_0,t_0]\}$ intersects the cross-section $S_p$ at least twice, contradicting the definition of a cross-section. Hence, $\pi(p)>t_0$, proving (3).
\end{proof}

\begin{remark}
    Observe that both $t_0$ and $\delta$ in the previous lemma depend on  $d(Sing(\phi),U)$. In other words, $t_0$ and $\delta$ may shrink as long as $U$ gets closer to $Sing(\phi)$.  
\end{remark}

We now turn our attention to the sets $Rep$ and $Rep(\eps)$.

\begin{lemma}\label{lemma_proprep}
    Let \(K\subset X\backslash Sing(\phi)\) be a compact set, then there is a \(T_0>0\) which satisfies the following:
    \begin{enumerate}
        \item\label{lemma_proprep:1} For any \(\ka\in (0,T_0)\), there is a \(\lambda>0\) such that for \(x\in K\) and \(0\leq s,t\leq T_0,\ d(\phi_s(x), \phi_t(x))<\lambda\) implies \(|s-t|<\ka\).
        \item\label{lemma_proprep:2} For any \(T\in (0,T_0)\), there is \(\gamma>0\) such that for 
         any $x\in K$,
        \(\phi|_{[0,t]}(x)\subset B_\gamma(x)\) implies that \(t\in [0,T]\).
        \item\label{lemma_proprep:3} for any \(\eps\in (0,1)\) and \(T\in (0,T_0)\), there is \(\delta>0\) such that for any \(x,y\in 
         K\), if \(g\in Rep\) satisfies \(d(\phi_t(x), \phi_{g(t)}(y))\leq \delta\) for all \(t\in[0,T]\), then 
        \[
        \left|\frac{g(T)}{T}-1\right|\leq \eps.
        \]
    \end{enumerate}
\end{lemma}

\begin{proof}
To begin with, notice that by assumption $K\subset X\setminus Sing(\phi)$ and hence there is an open neighborhood $U$ of $K$ such that $\overline{U}\cap Sing(\phi)=\emptyset$.  So, let $t_0,\delta>0$ be given by Lemma \ref{lemma_t_0}, with respect to $U$, and fix $T_0=\frac{t_0}{2}$. By possibly reducing $T_0$, we can assume $\phi_{[-T_0,T_0]}(K)\subset \overline{U}$.  

If \eqref{lemma_proprep:1} does not hold, there are $\kappa>0$, $\lambda_n\to 0$, $x_n\in K$ and  $s_n,t_n\in [0,T_0]$ such that $d(\phi_{s_n}(x_n),\phi_{t_n}(x_n))\leq\lambda_n$ and $|s_n-t_n|>\kappa$. Assume,  $x_n\to x\in K$, $s_n\to s$ and $t_n\to t$,  going to subsequences if needed. In this case we have $s\neq t$, but $\phi_t(x)=\phi_s(x)$. Since $K\cap Sing(\phi)=\emptyset$, we conclude $x$ is a periodic point with period $\pi(x)\leq 2T_0=t_0$. But this contradicts the choice of $t_0$ 
 in Lemma~\ref{lemma_t_0}, therefore \eqref{lemma_proprep:1} holds.

Next, suppose \eqref{lemma_proprep:2} does not hold. There exist $0<T<T_0$, 
 $t_n>0$, 
$\gamma_n\to 0$  and $x_n\in K$ so that $\phi_{[0,t_n]}(  
 x_n
)\subset B_{\gamma_n}(x_n)$, but $t_n\geq T$. Assume $x_n\to x\in K$. In this case, 
$$diam(\phi_{[0,T]}(
 x))=\lim\limits_{n\to\infty}diam(\phi_{[0,T]}(x_n)) \leq \lim\limits_{n\to\infty}diam(\phi_{[0,t_n]}(x_n))\leq\lim\limits_{n\to\infty} 2\gamma_n =0,$$ 
 that is, \(\phi_t(x)=x\) for every \(t\in[0,T]\) and hence $x\in Sing(\phi)$ contradicting $K\cap Sing(\phi)=\emptyset$. So, \eqref{lemma_proprep:2} holds.   

Now, let us prove \eqref{lemma_proprep:3}. Fix $0<\eps<1$ and $0<T<T_0$. \(\overline U\) being compact, we apply item \eqref{lemma_proprep:2} to $\overline{U}$, that is for \(\eps T\), there is $\gamma>0$ so  that if $\phi_{[0,
t
]}(x)\subset B_\gamma(x)$ 
and $x\in \overline{U}$, then $t\leq  \eps T$.
There exists \(\eta\in(0,\eps T)\) such that the continuity of $\phi$ and compactness of $X$ implies 
$d(x,\phi_t(x))\leq \frac{\gamma}{3}$, for every $x\in X$ and $t\in [0,\eta]$. 

By item \eqref{lemma_proprep:1}, there is $\lambda>0$ such that if $x\in K$ and $d(\phi_t(x),\phi_s(x))\leq \lambda$, then $|s-t|\leq \eta$. Fix $$\alpha=\min\left\{\frac{\lambda}{2},\frac{\gamma}{3}\right\}.$$

Let $0<\eps'\leq\alpha$ be so that 
$$
d(\phi_{t}(x),\phi_{t}(y))\leq \alpha, \textrm{ whenever } d(x,y)\leq \eps'
 \text{ and }|t|\leq T.
$$ 
Let $x,y\in K$ and suppose $g\in Rep$ satisfies $$d(\phi_t(x),\phi_{g(t)}(y))\leq \eps',$$ for every $t\in [0,T]$. 
 Note that in particular $d(x,y)\leq \eps'$.
If $g(T)\leq T$, we obtain: $$
d(\phi_{T}(x),\phi_{g(T)}(x))\leq d(\phi_T(x),\phi_{g(T)}(y))+(\phi_{g(T)}(y),\phi_{g(T)}(x))\leq \eps'+\alpha\leq\lambda,
$$
and hence $|g(T)-T|\leq  \eta<\eps T
$.

On the other hand, if $g(T)\geq T$, then  $g^{-1}(T)\leq T$ and therefore
$$d(\phi_{g^{-1}(T)}(x),\phi_T(x))\leq d(\phi_{g^{-1}(T)}(x),\phi_{T}(y))+ d(\phi_{T}(y),\phi_T(x))\leq \eps'+\alpha,$$
hence $|g^{-1}(T)-T|\leq \eta$. But this implies 
$$
d(\phi_{
 t}(x),\phi_{ T}
(x))\leq\frac{\gamma}{3},
$$ 
for every $t\in [g^{-1}(T),
 T]$. Furthermore, if $t\in [T,g(T)]$ we have
\begin{align*}
d(\phi_t(y), \phi_T(y)) & \leq d(\phi_{g(g^{-1}(t))}(y), \phi_{g^{-1}(t)}(x)) + d(\phi_{g^{-1}(t)}(x), \phi_T(x))+ d(\phi_T(x), \phi_{T}(y))\\
&\leq 
 \eps' + \frac{\gamma}{3} +\alpha\\
 &<\gamma
\end{align*}
Hence $ \phi_{[0,g(T)-T]}(\phi_T(y))\subset B_{\gamma}(\phi_T(y))$. On the other hand, since $T\leq T_0$, we obtain  $\phi_T(y)\in \overline{U}$ and the choice of $\gamma$ implies $|g(T)-T|\leq\eps T$.  Therefore \[
        \left|\frac{g(T)}{T}-1\right|\leq \eps,
        \] this completes the proof. 
\end{proof}

Recall that for any $\delta$-$T$-pseudo-orbit $P=(x_i,t_i)$, we denote   
\begin{equation}\label{eq:mathcal{P}}
    \SP(P)= \bigcup_{i\in \Z}\phi_{[0,t_i]}(x_i).
\end{equation}

\begin{theorem}\label{thm_regshad}

Let $\phi$ be a continuous flow with the oriented shadowing property and let
$\eta>0$.  For every $0<\varepsilon<\eta/2$ there exist
$0<\delta<\eta/2$ and $T>0$ such that every $\delta$--$T$-pseudo-orbit
$P$ satisfying
\[
 d(\mathcal P(P),\Sing(\phi))>\eta
\]
is strongly $\varepsilon$-traced.
\end{theorem}

\begin{proof}
	Start by defining
	\[
	K_0:=\{x\in X:d(x,\Sing(\phi))\geq\eta\},
	\qquad
	K_1:=\{x\in X:d(x,\Sing(\phi))\geq\eta/2\}.
	\]
	If \(K_0=\emptyset\), the theorem holds trivially, so assume that
\(K_0\neq\emptyset\) and put
	\[
	a:=\min\{\varepsilon/4,1/4\}.
	\]
	
	We first record a uniform consequence of
	Lemma~\ref{lemma_proprep}. After choosing \(T>0\) sufficiently small,
	there is \(\rho>0\) such that, for every \(L\in[T,2T]\), every
	\(x,y\in K_1\), and every \(q\in\mathrm{Rep}\),
	\begin{equation}\label{eq:uniform-time-control}
		\Bigg[ d\bigl(\phi_s(x),\phi_{q(s)}(y)\bigr)\leq\rho\quad
		\text{for all } 0\leq s\leq L \Bigg]
		\quad\Longrightarrow\quad
		\left|\frac{q(L)}{L}-1\right|\leq a.
	\end{equation}
	We may simultaneously require
	\begin{equation}\label{eq:short-motion}
		d(z,\phi_s(z))<\varepsilon/4
		\quad\text{for every }z\in X\text{ and }|s|\leq3T.
	\end{equation}
	
	For completeness, we justify the uniformity in
	\eqref{eq:uniform-time-control}. The set
	\[
	K_2:=\phi_{[-1,1]}(K_1)
	\]
	is compact and contains no singularity. Choose \(T<1/2\) sufficiently small, so that
	\(2T\) is smaller than the constant \(T_0\) supplied by
	Lemma~\ref{lemma_proprep} for \(K_1\), and that \(aT\) is smaller than
	the corresponding constant for \(K_2\).
	
	If \eqref{eq:uniform-time-control} failed for every \(\rho>0\), there
	would be a sequence $\rho_n$ decreasing to $0$, $L_n\in[T,2T]$,
	$x_n,y_n\in K_1$, and $q_n\in\mathrm{Rep}$ violating the implication.
	Assume it is the case, and then passing to subsequences assume additionally that
	\[
	x_n\to x\in K_1,\qquad y_n\to x,\qquad L_n\to L\in[T,2T].
	\]
	
	Suppose first that \(q_n(L_n)\leq L_n\) along a subsequence. After
	taking another subsequence, we may assume that
	\(q_n(L_n)\to r\in[0,2T]\). This gives
	\[
	\phi_r(x)=\phi_L(x).
	\]
	Lemma~\ref{lemma_proprep}\,(1), with $r,L\leq 2T<T_0$ and arbitrarily small $\kappa$, forces \(r=L\),
	contradicting
	\[
	|q_n(L_n)-L_n|>aL_n\geq aT.
	\]
	
	It remains to consider the case \(q_n(L_n)>L_n\). Set
	\[
	r_n:=q_n^{-1}(L_n).
	\]
	
	Since \(0\leq r_n\leq L_n\leq2T\), every subsequence of
	\((r_n)\) has a convergent subsequence. Suppose that
	\(r_{n_k}\to r\in[0,2T]\). Since \(q_{n_k}(r_{n_k})=L_{n_k}\), the
	shadowing estimate at time \(r_{n_k}\) gives
	\[
	d\bigl(\phi_{r_{n_k}}(x_{n_k}),
	\phi_{L_{n_k}}(y_{n_k})\bigr)\leq\rho_{n_k}.
	\]
	Passing to the limit yields
	\[
	\phi_r(x)=\phi_L(x).
	\]
	Lemma~\ref{lemma_proprep}\,(1), applied with arbitrarily small
	\(\kappa>0\), implies \(r=L\). Thus every convergent subsequence of
	\((r_n)\) converges to \(L\), and consequently
\begin{equation}\label{eq:small-arc}
	r_n-L_n\longrightarrow0.
\end{equation}

	For \(u\in[L_n,q_n(L_n)]\), we have
	\(q_n^{-1}(u)\in[r_n,L_n]\),  hence by \eqref{eq:small-arc}
	and
	uniform continuity, for every $\gamma>0$ there is $N$ such that for all $n\geq N$ we have
	\begin{equation}\label{eq:lambda-largeN}
	\phi_{[L_n,q_n(L_n)]}(y_n)
	\subset B_{\gamma}\bigl(\phi_{L_n}(y_n)\bigr).
	\end{equation}
	The points \(\phi_{L_n}(y_n)\) lie in the fixed compact regular set
	\(K_2\). 
	Let
	$
	z_n:=\phi_{L_n}(y_n),$ and
	$D_n:=q_n(L_n)-L_n>0$.
	Then
	\[
	\phi_{[0,D_n]}(z_n)
	=\phi_{[L_n,q_n(L_n)]}(y_n).
	\]
	Apply Lemma~\ref{lemma_proprep}\,(2) to \(K_2\) and to the time
	\(aT\), and let \(\gamma>0\) be the constant supplied by that lemma.
	For all sufficiently large \(n\),
	\[
	\phi_{[0,D_n]}(z_n)\subset B_\gamma(z_n),
	\]
	because \(z_n\in K_2\) and $\gamma$ in \eqref{eq:lambda-largeN} can be arbitrarily small, provided $n$ is sufficiently large.
	Consequently, there is $n$ such that
	\[
	D_n=q_n(L_n)-L_n\leq aT\leq aL_n,
	\]
	contradicting \(q_n(L_n)-L_n>aL_n\).
	Indeed
	\eqref{eq:uniform-time-control} holds.
	
	Decrease \(\rho\), if necessary, so that
	\begin{equation}\label{eq:rho-choice}
		0<\rho<\min\{\varepsilon/4,\eta/2\}.
	\end{equation}
	By the oriented shadowing property, together with the equivalence of
	the positive lower-time conventions recalled after
	Definition~\ref{def:standsh}, choose
	\[
	0<\delta<\min\{\rho,\eta/2\}
	\]
	such that every \(\delta\)--\(T\)-pseudo-orbit is \(\rho\)-traced with
	a reparametrization in \(\mathrm{Rep}\).
	
	Let \(P=(x_i,t_i)_{i\in\mathbb Z}\) be a
	\(\delta\)--\(T\)-pseudo-orbit satisfying
	\[
	d(\mathcal P(P),\Sing(\phi))>\eta.
	\]
	For every \(i\in\mathbb Z\), set
	\[
	m_i:=\lfloor t_i/T\rfloor\geq1,
	\qquad
	L_i:=\frac{t_i}{m_i}.
	\]
	Then
	\[
	T\leq L_i<2T.
	\]
	Replace the \(i\)-th orbit segment in the initial pseudo-orbit by the \(m_i\) pairs
	\[
	\bigl(\phi_{\ell L_i}(x_i),L_i\bigr),
	\qquad
	\ell=0,\ldots,m_i-1.
	\]
	Concatenating these pairs for all \(i\in\mathbb Z\) produces a
	\(\delta\)--\(T\)-pseudo-orbit
	\[
	Q=(w_j,L_j)_{j\in\mathbb Z}.
	\]
	Its interpolated pseudo-trajectory is exactly that of \(P\), because the
	inserted jumps have error zero, and the remaining jumps are precisely
	the original ones. Reindex \(Q\) so that its cumulative times
	\((R_j)_{j\in\mathbb Z}\) satisfy \(R_0=0\).
	
	By the oriented shadowing property, there are \(y\in X\) and \(h\in\mathrm{Rep}\) such that
	\begin{equation}\label{eq:oriented-Q}
		d\bigl(\phi_{h(R_j+s)}(y),\phi_s(w_j)\bigr)\leq\rho
		\quad\text{for every }j\in\mathbb Z
		\text{ and }0\leq s\leq L_j.
	\end{equation}
	Taking \(s=0\), we see that
	\[
	u_j:=\phi_{h(R_j)}(y)
	\]
	lies in \(K_1\), because \(w_j\in\mathcal P(P)\subset K_0\) and
	\(\rho<\eta/2\).
	
For each \(j\in\mathbb Z\), define \(q_j\colon\mathbb R\to\mathbb R\) by
\[
q_j(s):=h(R_j+s)-h(R_j),\qquad s\in\mathbb R.
\]
Then \(q_j\in\mathrm{Rep}\), and \eqref{eq:oriented-Q} gives
\[
d\bigl(\phi_s(w_j),\phi_{q_j(s)}(u_j)\bigr)\leq\rho
\qquad\text{for }0\leq s\leq L_j.
\]
	Consequently, \eqref{eq:uniform-time-control} gives
	\begin{equation}\label{eq:node-slopes}
		1-a
		\leq
		\frac{h(R_{j+1})-h(R_j)}{R_{j+1}-R_j}
		\leq
		1+a
		\quad\text{for every }j\in\mathbb Z.
	\end{equation}
	
	Let \(g\colon\mathbb R\to\mathbb R\) be the piecewise affine map
	satisfying
	\[
	g(R_j)=h(R_j)
	\quad\text{for every }j\in\mathbb Z.
	\]
	Recal that \(R_{j+1}-R_j\geq T\), \(g(0)=0\), and so
	\eqref{eq:node-slopes} shows that \(g\) is an increasing
	homeomorphism. Moreover, every difference quotient of \(g\) is a
	weighted average of the slopes in \eqref{eq:node-slopes}. Therefore
	\[
	g\in\mathrm{Rep}(a)\subset\mathrm{Rep}(\varepsilon).
	\]
	
	Finally, let \(t=R_j+s\), where \(0\leq s\leq L_j\). Then
	\[
	0\leq g(t)-h(R_j)\leq(1+a)L_j<3T.
	\]
	Using \eqref{eq:short-motion}, \eqref{eq:oriented-Q} at \(s=0\), and
	again \eqref{eq:short-motion}, we obtain
	\begin{eqnarray*}
		d\bigl(\phi_{g(t)}(y),\phi_s(w_j)\bigr)
		&\leq&
		d\bigl(\phi_{g(t)}(y),\phi_{h(R_j)}(y)\bigr)
		+d\bigl(\phi_{h(R_j)}(y),w_j\bigr) \\
		&&\qquad\qquad
		+d\bigl(w_j,\phi_s(w_j)\bigr) \\
		&<&\varepsilon/4+\rho+\varepsilon/4
		<\varepsilon.
	\end{eqnarray*}
	Since \(Q\) and \(P\) have the same interpolated pseudo-trajectory,
	\(y\) strongly \(\varepsilon\)-traces \(P\) with the
	reparametrization \(g\).
\end{proof}

An immediate consequence of the previous result is the following

\begin{corollary}\label{cor:locshad}
    Let $\phi$ be a flow with the oriented shadowing property. If $\Lambda\subset X$ is a compact and invariant set without singularities, then $\Lambda$ has the local standard shadowing property.
\end{corollary}

\begin{proof}
If $\Sing(\phi)=\emptyset$ take $\eta=1$ and otherwise put
	\[
	\eta=\tfrac14\,d\bigl(\Lambda,\Sing(\phi)\bigr)>0.
	\]
Denote $U=B_{\eta}(\Lambda)$.
It suffices to prove the statement assuming	$\varepsilon>0$ is sufficiently small, so without loss of generality fix any  $0<\varepsilon<\eta$ and
	$T>0$. Let $\delta_0>0$ and $T'>0$ be given by Theorem~\ref{thm_regshad}
	applied to $\eta$ and $\varepsilon/2$.
	
	Assume first $T\geq T'$. Every $\delta_0$-$T$-pseudo-orbit $P$ is a
	$\delta_0$-$T'$-pseudo-orbit, and if $\SP\subset U$ then
	$d(\SP(P),\Sing(\phi))\geq 4\eta-\eta>\eta$. Hence $U$ and $\delta_0$ ensure the
	local standard shadowing property for $\varepsilon$ and $T$ in the above caase.
	
	Assume now $T<T'$ and fix an integer $N$ with $NT\geq T'$. Given a
	$\delta$-$T$-pseudo-orbit $P=(x_i,t_i)_{i\in\Z}$, merge its steps in blocks of
	length $N$, that is, put
	\[
	P'=(x'_j,t'_j)_{j\in\Z},\qquad
	x'_j:=x_{Nj},\qquad t'_j:=\sum_{l=0}^{N-1}t_{Nj+l}\;\geq\;NT\geq T' .
	\]
	Since the total times are unchanged, the associated sequences satisfy
	$S'_j=S_{Nj}$ for every $j$, and $t'_j\le 2NT'$ because 
    by subdividing the orbit segments when necessary, we may assume without loss of
generality that
\[
t_i\in[T,2T)
\quad\text{for every }i\in\mathbb Z.
\]
	
	By uniform continuity of $(t,x)\mapsto\phi_t(x)$ on
	$[0,2NT']\times X$ there is $0<\delta\leq\delta_0$ such that,
	for every such $P$, every $j\in\Z$ and every $t\in[S'_j,S'_{j+1}]$ and $i$ such that  $Nj\leq i<N(j+1)$ and  $t\in[S_i,S_{i+1}]$, we have
	\[
	d\bigl(\phi_{t-S'_j}(x'_j),\phi_{t-S_i}(x_i)\bigr)\leq\frac{\varepsilon}{2}.
	\]
	This is obtained by taking $\delta$ sufficiently small and controlled propagation of at most
$N-1$  errors caused by $\delta$-jumps in pseudo-orbit before modification. In particular $P'$ is a $\delta_0$-$T'$-pseudo-orbit and
	$\SP(P')\subset B_{\varepsilon/2}(\SP(P))$.
	
	Suppose now $\SP(P)\subset U$. Then $\SP(P')\subset B_{\eta}(U)\subset B_{2\eta}(\Lambda)$,
	so $$d(\SP(P'),Sing(\phi))\geq 4\eta-2\eta>\eta$$ and Theorem~\ref{thm_regshad}
	provides $z\in X$ and $g\in\mathrm{Rep}(\varepsilon/2)$ with
	\[
	d\bigl(\phi_{t-S'_j}(x'_j),\phi_{g(t)}(z)\bigr)\leq\frac{\varepsilon}{2},
	\]
	for all $ j\in\Z$ and $ t\in[S'_j,S'_{j+1}]$.
	Combining the  above estimates and using
	$\mathrm{Rep}(\varepsilon/2)\subset\mathrm{Rep}(\varepsilon)$ we get
	$d\bigl(\phi_{t-S_i}(x_i),\phi_{g(t)}(z)\bigr)\leq\varepsilon$ for every $i\in\Z$
	and every $t\in[S_i,S_{i+1}]$.
	Thus $U$ and $\delta$ indeed ensure the local standard shadowing property for
	$\varepsilon$ and $T$, completing the proof.
\end{proof}

\section{Recurrence Versus Topological Entropy}\label{sec:classification}

In this section, we investigate the role of recurrence in the context of flows with shadowing. We are particularly interested in proving Theorems \ref{thm:positive entropy}, \ref{thm:standard}, and \ref{thm:surfacecharac}.  We will see that shadowing is a strong tool to obtain positive entropy, and the  possible types of recurrence allowing zero entropy flows with the standard shadowing property fall short in a simple classification.  To begin with, we introduce the terminology that will be predominant in all of our results hereafter. 

\begin{definition}
We say that $K$ is singularly minimal if every proper compact invariant subset of $K$ is contained in the set of singularities of $\phi$.
\end{definition}
We first observe that if a singularly minimal set is non-singular, then it is a classical minimal set. On the other hand, there are several examples of singular minimal sets that show that this concept is a proper extension of the concept of minimal set.

\begin{example}[Homoclinic Loops]
    A compact and invariant set $K$ is a homoclinic loop if there is a regular point $x\in K$ and $\sigma\in \operatorname{Sing}(\phi)$ such that $$K=\overline{\mathcal{O}(x)}=\mathcal{O}(x)\cup\{\sigma\}.$$ 
Homoclinic loops are clearly singularly minimal sets since their only compact and invariant sets are $K$ and $\{\sigma\}$. On the other hand, the presence of $\sigma$ implies that $K$ is not minimal.

\end{example}
 Besides, the previous trivial example, there are examples of singularly minimal sets with non-trivial dynamics. The next example presents a construction of singularities similar to one we performed in Section \ref{sec:examplesufaces}.

\begin{example}[Singular Suspensions]
Let $\phi:\hat M\to \hat M$ be a suspension flow of a given diffeomorphism $f: M\to M$. Let $K\subset  \hat M$ be a compact set and $\rho:
 \hat M\to [0,\infty) $ be a $C^{\infty}$-map such that $\rho(x)=0$ if and only if $x\in K$.  Let $V$ be the velocity vector field generated by $\phi$.  The singular suspension of $\phi$ with break function $\rho$ is the flow $\psi$ generated by the vector field $\rho V$. If one starts with a minimal diffeomorphism, then $\hat M$ is clearly a singularly minimal set for $\psi$. Although our example started with a diffeomorphism, one can also construct singular minimal suspensions starting with homeomorphisms of compact metric spaces (see for instance \cite{Ohno80}). In addition, one can construct singular minimal suspensions with zero, positive, or infinite entropy (see \cite{SYZ09,SZ11}).       
\end{example}

\begin{figure}[h!]
    \centering

    \begin{subfigure}[b]{0.48\textwidth}
        \centering
        \begin{tikzpicture}[scale=1.2]
   
    \draw[thick, smooth, tension=0.8] plot coordinates {
        (0,0) 
        (0.2, 1.5) 
        (1.5, 3) 
        (2.5, 2.5) 
        (2.3, 1) 
        (1.5, 0.8)
        (0,0)
    };
    
    \draw[thick, ->] (1.6, 0.84) -- (1.4, 0.75);
    
    \fill[OrangeRed] (0,0) circle (2.5pt);
    
    \node at (0.2, -0.4) {$\sigma$};
\end{tikzpicture}
        \caption{A homoclinic loop.}
        \label{fig:tikz}
    \end{subfigure}
    \hfill
    \begin{subfigure}[b]{0.48\textwidth}
        \centering
        \includegraphics[width=\linewidth]{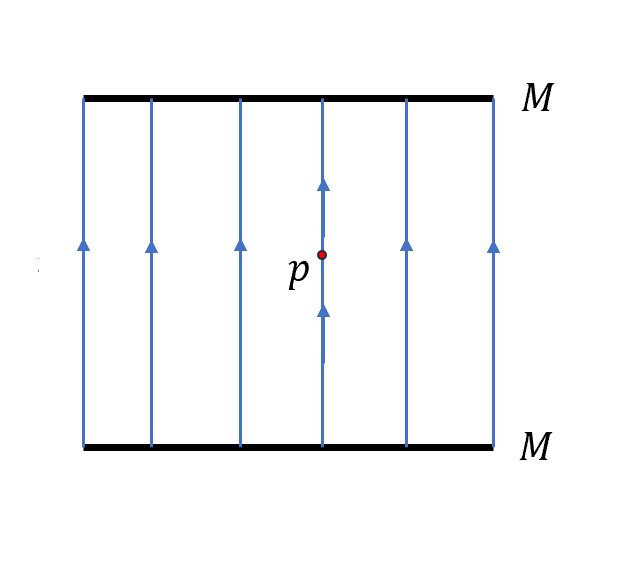}
        \caption{A singular suspension ($K=\{p\}$).}
        \label{fig:png}
    \end{subfigure}

\end{figure}

The next proposition gives  singular counter-part for a well known characterization of minimal sets. 
 
\begin{proposition}
Let $\phi$ be a continuous flow and let $K$ be a compact invariant set. 
 $K$ is singularly minimal if and only if every regular orbit contained in $K$ is dense in $K$.
\end{proposition}

\begin{proof}
 Assume $K$ is singularly minimal.
Let $x$ be a regular point in $K$, i.e., $x \in K$ is not a singularity.  Clearly, $\overline{\mathcal{O}(x)}$ is a compact invariant subset of $K$. Since $x$ is regular,  $\overline{\mathcal{O}(x)}$ is not contained in $\operatorname{Sing}(\phi)$. By the singular minimality of $K$, $\overline{\mathcal{O}(x)} = K$.

 Assume every regular orbit in $K$ is dense in $K$. Let $L \subset K$ be a proper compact invariant subset. Suppose, for contradiction, that $L$ contains a regular point $y$. Then the orbit $\mathcal{O}(y)$ is contained in $L$ because $L$ is invariant.  We get $L = K$, contradicting the fact that $L$ is proper. Therefore, $L$ cannot contain any regular point. $L \subseteq \operatorname{Sing}(\phi).$
\end{proof}

Next, we introduce the concepts of heteroclinic  and almost heteroclinic sets that will be essential to proving the main results in this section.  Recall that orbit of a point $x$ is said  \emph{a heteroclinic connection} if  $\overline{\mathcal{O}(x)}$ is not a homoclinic loop and $$\alpha(x)\cup \omega(x)\subset \operatorname{Sing}(\phi).$$

\begin{definition}
A compact and invariant chain-transitive set $\Lambda$ is said to be:
\begin{itemize}
    \item a \textit{heteroclinic set} if the orbit of any regular point $x$ is a heteroclinic connection. 
    \item an \textit{almost heteroclinic set} if any of its non-dense regular orbits is a heteroclinic connection.

\end{itemize}
 Furthermore, if $\Lambda$ is an almost heteroclinic set and does not contain any non-trivial proper almost heteroclinic set, we say that $\Lambda$ is an \textit{irreducible almost heteroclinic set}.
\end{definition}

\begin{figure}[ht]
    \centering

    \begin{subfigure}[b]{0.48\textwidth}
        \centering
        \begin{tikzpicture}[scale=1.5]
    \node[circle, fill=red!80!black, inner sep=2pt] (A) at (0,0) {};
    \node[circle, fill=red!80!black, inner sep=2pt] (B) at (2,0) {};

    \node at (-0.3, -0.5) {$\sigma_1$};
    \node at (2.3, -0.4) {$\sigma_2$};

    \draw[thick] (A) to[bend left=45] node[above] {$\leftarrow$} (B);
    \draw[thick] (A) to[bend right=45] node[below] {$\rightarrow$} (B);
\end{tikzpicture}
       \caption{An irreducible almost heteroclinic set}
   
    \end{subfigure}
    \hfill
    \begin{subfigure}[b]{0.48\textwidth}
        \centering
        \begin{tikzpicture}[scale=1.0, every node/.style={inner sep=0pt, outer sep=0pt}]

    \coordinate (v1) at (0,0);
    \coordinate (v2) at (0.5,2);
    \coordinate (v3) at (3,3);
    \coordinate (v4) at (3.5,0.5);
    \coordinate (v5) at (2,1);

    \begin{scope}[
        decoration={
            markings,
            mark=at position 0.5 with {\arrow{stealth}}
        }
    ]
        \draw[thick, postaction={decorate}] (v1) -- (v2);
        \draw[thick, bend left=15, postaction={decorate}] (v2) -- (v3);
        \draw[thick, postaction={decorate}] (v3) -- (v5);
        \draw[thick, postaction={decorate}] (v5) -- (v1);
        \draw[thick, postaction={decorate}] (v5) -- (v4);
        \draw[thick, postaction={decorate}] (v4) -- (v3);
    \end{scope}

    \foreach \v in {v1, v2, v3, v4, v5} {
        \fill[OrangeRed] (\v) circle (2.5pt);
    }

    \node[below=5pt of v1] {$\sigma_1$};
    \node[above left=2pt of v2] {$\sigma_2$};
    \node[above right=2pt of v3] {$\sigma_3$};
    \node[below right=2pt of v4] {$\sigma_4$};
    \node[below=5pt of v5] {$\sigma_5$};

\end{tikzpicture}
   \caption{A reducible almost heteroclinic set}
           
    \end{subfigure}

\end{figure}

\begin{rmk}
    Observe that any non-trivial singularly minimal set is trivially almost heteroclitic, because every regular orbit is dense. Also, every almost heteroclinic set is heteroclinic, provided it contains no dense orbits. Finally, in any case the set is aperiodic, unless it is a single periodic orbit.    
\end{rmk}

\begin{lemma}\label{lemma:posentr}
    Let $\Lambda$ be a compact and invariant set with the local standard shadowing property.
    Suppose  there exists $\eta>0$ such
    that, for every $0<\delta<\eta/10$ and every $T>0$, the following
    conditions hold:
    \begin{enumerate}
        \item there are points $x,y\in X$ and positive numbers
        $t_0,T_x,T_y$ such that
        $$
        T<\min\{t_0,T_y\},\qquad t_0<T_x<10T_y,
        $$
        and
        $$
        d(\phi_{T_x}(x),x)\leq \frac{\delta}{2},\quad
        d(\phi_{T_y}(y),y)\leq \frac{\delta}{2},\quad
        d(x,y)\leq\frac{\delta}{2}.
        $$
        \item $\phi_{[0,T_x]}(x)\subset B_{\delta}(\Lambda)$ and $\phi_{[0,T_y]}(y)\subset B_{\delta}(\Lambda)$.
        \item $\phi_{t_0}(x)\notin B_{2\eta}(\phi_{[0,T_y]}(y))$.
    \end{enumerate}
    Then, $\phi$ has positive topological entropy.
\end{lemma}

\begin{proof}
Let $\Lambda$ be a compact invariant set with the local standard
shadowing property, and assume that the hypotheses of the lemma hold.

Choose
\[
0<10\varepsilon<\eta,
\qquad
\beta:=\frac{2\varepsilon}{1-\varepsilon}<\frac1{100},
\]
and fix $T>0$. Let $U$ and $\delta_0>0$ be supplied by the local
standard shadowing property for $\varepsilon$ and $T$. Choose $r>0$
such that $B_r(\Lambda)\subset U$, and then fix
\[
0<\delta<\min\left\{\delta_0,r,\frac{\varepsilon}{10}\right\}.
\]

Take $x,y\in S$ and
\[
T<t_0<T_x<10T_y
\]
satisfying the assumptions of the lemma. Notice that
every transition between the orbit segment of $x$ and the orbit
segment of $y$ produces a $\delta$-$T$ jump. Indeed,
\[
d(\phi_{T_x}(x),y)
\leq d(\phi_{T_x}(x),x)+d(x,y)\leq\delta,
\]
and similarly
\[
d(\phi_{T_y}(y),x)
\leq d(\phi_{T_y}(y),y)+d(y,x)\leq\delta.
\]

For every $s=(s_i)_{i\in\mathbb Z}\in\{0,1\}^{\mathbb Z}$, define
\[
P_s=(x_i^s,t_i^s)_{i\in\mathbb Z},
\qquad
(x_i^s,t_i^s)=
\begin{cases}
(x,T_x),&s_i=0,\\
(y,T_y),&s_i=1.
\end{cases}
\]
Then $P_s$ is a $\delta$-$T$-pseudo-orbit and, by item~(2),
\[
\mathcal P_s\subset B_\delta(\Lambda)\subset U.
\]
Consequently, every $P_s$ is strongly $\varepsilon$-traced by some
$y_s\in X$ with a reparametrization
$h_s\in\mathrm{Rep}(\varepsilon)$.

Let $M=100$, $R=50$,
and introduce the words
\[
\mathcal B_0:=0\,1^M,\qquad
\mathcal B_1:=0\,1^{M+R}.
\]
together with associated times given by
\[
D_0:=T_x+MT_y,\qquad
D_1:=T_x+(M+R)T_y.
\]
Since $T_x<10T_y$, we have
\begin{equation}\label{eq:macro-estimates}
100T_y<D_0<110T_y,\qquad
150T_y<D_1<160T_y.
\end{equation}

For $n\geq1$ and $w=(w_0,\ldots,w_{n-1})\in\{0,1\}^n$, let
$s(w)\in\{0,1\}^{\mathbb Z}$ be obtained by concatenating
\[
\mathcal B_{w_0}\mathcal B_{w_1}\cdots\mathcal B_{w_{n-1}}
\]
and completing the sequence in both directions with copies of
$\mathcal B_0$. Define
\[
z_w=y_{s(w)},\qquad
h_w=h_{s(w)}\in\mathrm{Rep}(\varepsilon).
\]

Put $R_0^w=0$, $a_0^w=t_0$ and for $j=1,\ldots,n$ denote
$$
R_j^w:=\sum_{i=0}^{j-1}D_{w_i}\quad(j\geq1),\qquad
a_j^w:=R_j^w+t_0.
$$
The time $a_j^w$ corresponds to $\phi_{t_0}(x)$ in the segment associated with $x$  at
the beginning of the $j$-th coding word. Consider the sequence of times and points
\[
H_n=(1+\varepsilon)(nD_1+t_0),\qquad
E_n=\{z_w:w\in\{0,1\}^n\}.
\]

We claim that $E_n$ is $(H_n,\eta)$-separated for the flow. Assume the contrary, and  let
 $w,v\in\{0,1\}^n$ be distinct words such that for every $0\leq t\leq H_n$ we have
\begin{equation}\label{eq:not-separated}
d(\phi_t(z_w),\phi_t(z_v))\leq\eta.
\end{equation}
Let $q=h_v^{-1}\circ h_w$ and observe that since $h_w,h_v\in\mathrm{Rep}(\varepsilon)$, we have that
$q\in\mathrm{Rep}(\beta)$. Fix $0\leq j\leq n$ and note that $h_w(a_j^w)\leq H_n$, therefore we may apply
\eqref{eq:not-separated} at time $h_w(a_j^w)$. If $q(a_j^w)$
belonged to a $y$-block of $P_{s(v)}$, with some phase $u\in[0,T_y]$, then
\[
\begin{aligned}
d(\phi_{t_0}(x),\phi_u(y))
&\leq d(\phi_{t_0}(x),\phi_{h_w(a_j^w)}(z_w))\\
&\quad+d(\phi_{h_w(a_j^w)}(z_w),
          \phi_{h_v(q(a_j^w))}(z_v))\\
&\quad+d(\phi_{h_v(q(a_j^w))}(z_v),\phi_u(y))\\
&\leq2\varepsilon+\eta<2\eta,
\end{aligned}
\]
contradicting item~(3). If \(q(a_j^w)\) were a transition time adjacent to a \(y\)-block,
we could apply the same estimate above, taking
\(u=0\) or \(u=T_y\), according as the transition is its initial or
terminal endpoint. Every transition time of \(P_{s(v)}\) is adjacent to a
\(y\)-block, since two \(x\)-blocks are never consecutive. Hence
\(q(a_j^w)\) lies in the interior of an \(x\)-block of \(P_{s(v)}\). 
Moreover,
\[
q(a_0^w)=q(t_0)<(1+\beta)T_x<10.1T_y,
\]
whereas the second $x$-block starts after time $100T_y$. Thus
$q(a_0^w)$ belongs to the initial $x$-block. Since
\[
a_{j+1}^w-a_j^w=D_{w_j},
\]
we have
\begin{equation}\label{eq:q-macro-gap}
(1-\beta)D_{w_j}
\leq q(a_{j+1}^w)-q(a_j^w)
\leq(1+\beta)D_{w_j}.
\end{equation}
The two images cannot belong to the same $x$-block, because their
difference would be at most $T_x<10T_y$, whereas
$(1-\beta)D_0>99T_y$. They cannot skip an $x$-block, because then
their difference would be greater than
\[
2D_0-T_x>190T_y,
\]
whereas $(1+\beta)D_1<161.6T_y$. Hence they belong to consecutive
$x$-blocks.

The intervening block $\mathcal B_{w}$ is determined uniquely. If $w_j=0$ but
the corresponding block of $P_{s(v)}$ has type $1$, then
\[
q(a_{j+1}^w)-q(a_j^w)\leq(1+\beta)D_0<111.1T_y,
\]
whereas its value must be at least $D_1-T_x=150T_y$. Conversely, if
$w_j=1$ but the corresponding block of $P_{s(v)}$ has type $0$, then
\[
q(a_{j+1}^w)-q(a_j^w)\geq(1-\beta)D_1>148.5T_y,
\]
whereas its value must be at most $D_0+T_x<120T_y$. Both alternatives
are impossible. Induction therefore gives
$
w_j=v_j$ for all $j=0,\ldots,n-1$,
contrary to $w\neq v$. Hence $E_n$ is $(H_n,\eta)$-separated and
$|E_n|=2^n$.

By uniform continuity of the flow on $[0,1]\times X$, there exists
$\eta_0>0$ such that
\[
d(u,v)\leq\eta_0
\quad\Longrightarrow\quad
d(\phi_s(u),\phi_s(v))\leq\eta
\qquad  \text{for all }0\leq s\leq1.
\]
Thus $E_n$ is an $(\lceil H_n\rceil+1,\eta_0)$-separated set for
$\phi_1$, and
\[
\begin{aligned}
h_{\mathrm{top}}(\phi)=h_{\mathrm{top}}(\phi_1)
&\geq\limsup_{n\to\infty}
\frac{\log|E_n|}{\lceil H_n\rceil+1}\\
&=\frac{\log2}{(1+\varepsilon)D_1}>0.
\end{aligned}
\]
This completes the proof.
\end{proof}

\begin{proof}[Proof of Theorem \ref{thm:positive entropy}]
Let $\Lambda$ be a non-trivial chain-transitive set with the local
standard shadowing property. Assume that $\Lambda$ is not an
irreducible almost heteroclinic set. The proof is by reduction to 
Lemma~\ref{lemma:posentr}.

Suppose first that $\Lambda$ is not almost heteroclinic. By definition,
there exists a regular point $p\in\Lambda$ whose orbit is not dense in
$\Lambda$ and is not a heteroclinic connection. Fix such a point $p$.
Then exactly one of the following alternatives holds:
\begin{itemize}
	\item[(i)]
	$
	(\alpha(p)\cup\omega(p))
	\setminus\operatorname{Sing}(\phi)\neq\emptyset;
	$
	\item[(ii)] $\overline{\mathcal O(p)}$ is a homoclinic loop.
\end{itemize}
If, on the other hand, $\Lambda$ is almost heteroclinic, then it is not 
irreducible, and hence:
\begin{itemize}
	\item[(iii)] $\Lambda$ contains a proper non-trivial almost
	heteroclinic subset.
\end{itemize}
 In case~(i), let $\widetilde\Lambda$ be one of the limit sets
$\alpha(p)$ or $\omega(p)$ containing a regular point. In case~(ii),
put $\widetilde\Lambda:=\overline{\mathcal O(p)}$. In case~(iii), let
$\widetilde\Lambda\subsetneq\Lambda$ be a non-trivial proper almost
heteroclinic set. In every case $\widetilde\Lambda$ is compact,
invariant, chain transitive and proper subset of $\Lambda$, and contains a
regular point.

Choose a regular point $z\in\widetilde\Lambda$ and a point
$q\in\Lambda\setminus\widetilde\Lambda$, and define $\chi=d(q,\widetilde\Lambda)>0.$
Fix any $\eta>0$ such that $10\eta<\chi$ and
let $0<\delta<\eta/10$ and $T>0$ be arbitrary.
 Choose
\[
0<\rho<\min\left\{\frac12,\frac{\delta}{4},
                    \frac{\chi-2\eta}{2}\right\}
\quad\text{and}\quad
\widehat T>\frac{T}{1-\rho}.
\]
Apply the definition of local standard shadowing property with accuracy $\rho$ and
minimum step length $\widehat T$. Let $U$ be the resulting
neighborhood of $\Lambda$ and let $\gamma_0>0$ be the corresponding
jump size. Fix $0<\gamma\leq\gamma_0$.

Since $z$ is regular, choose $z'\in\widetilde\Lambda$ with $z'\neq z$.
By chain transitivity of $\widetilde\Lambda$, concatenate a
$\gamma$-$\widehat T$ chain from $z$ to $z'$ with one from $z'$ to
$z$. This gives a non-trivial closed $\gamma$-$\widehat T$ chain
based at $z$. Denote its total pseudo-time by $L_y$ and repeat it
periodically in both directions. Its orbit segments are contained in
$\widetilde\Lambda\subset\Lambda\subset U$; hence it is strongly
$\rho$-traced by a point $y$ with some
$h_y\in\mathrm{Rep}(\rho)$.
 At every cumulative time $kL_y$, $k\in\Z$, the periodic
chain is again at $z$. Consequently,
\begin{equation}\label{eq:theoremB4-y-return}
d(y,z)\leq\rho,
\qquad
d\bigl(\phi_{h_y(kL_y)}(y),z\bigr)\leq\rho,
\end{equation}
and
\begin{equation}\label{eq:theoremB4-y-near}
d(\phi_s(y),\widetilde\Lambda)\leq\rho
\qquad\text{for every }s\in[0,h_y(kL_y)].
\end{equation}

Next, use the chain transitivity of $\Lambda$ to concatenate a
$\gamma$-$\widehat T$ chain from $z$ to $q$ with one from $q$ back to
$z$. Repeat the resulting closed chain periodically in both
directions. Let $L_x$ be its period and let
$A\in(0,L_x)$ be the smallest cumulative time at which pseudo-orbit reaches $q$. Same as before, also this
periodic pseudo-orbit is contained in $\Lambda$, and so it is strongly
$\rho$-traced by a point $x$ with some
$h_x\in\mathrm{Rep}(\rho)$.
 Define
\[
t_0=h_x(A),\qquad T_x=h_x(L_x),
\]
and observe that by the definition:
\begin{equation}\label{eq:theoremB4-x-estimates}
d(x,z)\leq\rho,
\qquad
d(\phi_{T_x}(x),z)\leq\rho,
\qquad
d(\phi_{t_0}(x),q)\leq\rho,
\end{equation}
and
\begin{equation}\label{eq:theoremB4-x-near}
d(\phi_s(x),\Lambda)\leq\rho
\qquad\text{for every }s\in[0,T_x].
\end{equation}
Since $A\geq\widehat T$ and $h_x\in\mathrm{Rep}(\rho)$, we have
\[
t_0=h_x(A)\geq(1-\rho)A
\geq(1-\rho)\widehat T>T,
\]
and $t_0<T_x$. Clearly $h_y(kL_y)$ is unbounded, so there is $k\geq1$ such that if we put
$T_y=h_y(kL_y)$ then $T_y>T$ and $T_x<10T_y$.
Equations~\eqref{eq:theoremB4-y-return} and
\eqref{eq:theoremB4-x-estimates} give
\[
d(\phi_{T_x}(x),x)\leq2\rho<\frac{\delta}{2},
\qquad
d(\phi_{T_y}(y),y)\leq2\rho<\frac{\delta}{2},
\]
and
\[
d(x,y)\leq d(x,z)+d(z,y)\leq2\rho<\frac{\delta}{2}.
\]
Furthermore, by \eqref{eq:theoremB4-y-near} and
\eqref{eq:theoremB4-x-near},
\[
\phi_{[0,T_x]}(x)\subset B_\delta(\Lambda),
\qquad
\phi_{[0,T_y]}(y)\subset B_\delta(\Lambda).
\]

 Finally, for every $s\in[0,T_y]$, we have
\[
\begin{aligned}
d(\phi_{t_0}(x),\phi_s(y))
&\geq d(q,\widetilde\Lambda)
      -d(\phi_{t_0}(x),q)
      -d(\phi_s(y),\widetilde\Lambda)\\
&\geq\chi-2\rho>2\eta
\end{aligned}
\]
which in particular implies that
\[
\phi_{t_0}(x)\notin
B_{2\eta}\bigl(\phi_{[0,T_y]}(y)\bigr).
\]
All the hypotheses of Lemma~\ref{lemma:posentr} are satisfied, and
therefore $h_{\mathrm{top}}(\phi)>0$. The proof is completed.
\end{proof}

We are now in a position of proving Theorem \ref{thm:standard}.

\begin{proof}[Proof of Theorem~\ref{thm:standard}]
	Let $\Lambda$ be a non-trivial chain-recurrent class of $\phi$.
	Then $\Lambda$ is a compact, invariant, and chain-transitive set.
	The standard shadowing property of $\phi$ implies that
	$\Lambda$ has the local standard shadowing property.
	If $\Lambda$ were not an irreducible almost heteroclinic set,
	Theorem~\ref{thm:positive entropy} would imply that
	$h_{\mathrm{top}}(\phi)>0$,
	contrary to the hypothesis. Consequently, every non-trivial
	chain-recurrent class of $\phi$ is an irreducible almost heteroclinic
	set.
	
	Suppose now that $\Lambda$ is non-singular. We claim that it is minimal.
	Indeed, if $x\in\Lambda$ had a non-dense orbit in $\Lambda$, then, since
	$\Lambda$ is almost heteroclinic and $x$ is regular, the orbit of $x$
	would be a heteroclinic connection, that is
	$$
    \alpha(x)\cup\omega(x)\subset\Sing(\phi).
	$$
	On the other hand, compactness and invariance of $\Lambda$ give
	$\alpha(x)\cup\omega(x)\subset\Lambda$.
	Since both limit sets are nonempty and
	$\Lambda\cap\Sing(\phi)=\emptyset$, this is impossible. Thus every orbit
	in $\Lambda$ is dense in $\Lambda$, and therefore $\Lambda$ is minimal, proving the claim.
	
	Finally, if $\phi$ is regular, every chain-recurrent class is
	non-singular. The non-trivial classes are minimal by the preceding
	argument, while every trivial class, being a compact single orbit, is
	minimal as well. Since the chain-recurrent classes partition
	the chain-recurrent set, it must be a disjoint union of minimal
	sets.
\end{proof}

Finally, we end this work by providing  a proof for Theorem \ref{thm:surfacecharac}.

 \begin{proof}[Proof of Theorem~\ref{thm:surfacecharac}]
  Every continuous flow on a compact surface has zero topological
	entropy \cite{LSY}. Let $H$ be a non-singular, non-trivial chain-recurrent class.
By Corollary~\ref{cor:locshad}, $H$ has the local standard shadowing
property, and hence Theorem~\ref{thm:positive entropy} implies that
$H$ is an irreducible almost heteroclinic set. 
If $x\in H$ had a non-dense orbit in $H$, then the orbit of
$x$ would be a heteroclinic connection by definition. Hence
$
\alpha(x)\cup\omega(x)\subset\Sing(\phi)\cap H=\emptyset
$
which is a contradiction.
Thus every orbit in $H$ is dense, and
therefore $H$ is minimal.

It remains to prove that
\[
H(\sigma)=\{\sigma\}
\qquad\text{for every }\sigma\in\Sing(\phi).
\]
Suppose, by contradiction, that there exists $\sigma\in\Sing(\phi)$
such that $H(\sigma)\neq\{\sigma\}$. Since $\Sing(\phi)$ is finite,
 the class $H(\sigma)$ contains a regular point
$p\neq\sigma$. Choose a Jordan domain $U$ such that
\[
\sigma\in U,\qquad
 p\notin\overline U,
\qquad\text{and}\qquad
\overline U\cap\Sing(\phi)=\{\sigma\}.
\]

We first record an important consequence of oriented shadowing.
Suppose that $C\subset U$ is an invariant Jordan curve, its interior
contains a point $q\in H(\sigma)$, and $p$ lies in its exterior. 

\textbf{Barrier claim.} \textit{We claim that such a curve $C$ cannot exist.}
Indeed, choose
\[
0<\varepsilon_C<\frac13
\min\{d(q,C),d(p,C)\}
\]
and let $\delta_C>0$ be given by the oriented shadowing property.
Since $q,p\in H(\sigma)$, concatenate a $\delta_C$-chain from $q$ to
$p$ with one from $p$ to $q$ and repeat the resulting closed chain in
both directions. Any orbit which $\varepsilon_C$-traces this
pseudo-orbit must pass from the interior of $C$ to its exterior and
hence must intersect $C$. Since $C$ is invariant, such an orbit is
contained in $C$, contrary to the choice of $\varepsilon_C$. This
proves the claim.

We shall also use the claim for homoclinic loops contained in $U$.
Suppose that
\[
C=\mathcal O(z)\cup\{\sigma\}\subset U
\]
is such a loop, and denote its interior by $D$. Then $D\subset
H(\sigma)$. To see this, fix $x\in D$. By the Poincar\'e--Bendixson theorem
\cite[Section~10.5]{HSD13}, a nonempty limit set of $x$ is either a
periodic orbit or a singular graph (including singleton $\{\sigma\}$). The first possibility is excluded because
every periodic orbit in $D$ surrounds a singularity, whereas
$D\cap\Sing(\phi)=\emptyset$. Consequently, the limit
sets of $\alpha(x)$, $\omega(x)$ are contained in singular graphs based at $\sigma$, whose
points belong to $H(\sigma)$. Following the orbit of $x$ from a point
close to its $\alpha$-limit set and towards its $\omega$-limit set shows
that $x$ is chain-related to $\sigma$. Thus $x\in H(\sigma)$. Taking
any $q\in D$ in the barrier claim shows that no homoclinic loop
contained in $U$ can exist.

Choose $\varepsilon>0$ sufficiently small, so that
\[
\overline{B_{\varepsilon}(\sigma)}\subset U
\qquad\text{and}\qquad
 B_{\varepsilon}(p)\cap\overline U=\emptyset.
\]
 Let $\delta>0$ be given by the oriented shadowing property for
$\varepsilon$.
Since $p\in H(\sigma)$, there are a $\delta$-chain from $\sigma$ to
$p$ and a $\delta$-chain from $p$ to $\sigma$. Concatenate them and
extend the resulting pseudo-orbit in both directions by the constant
pseudo-orbit at $\sigma$. Thus there is a $\delta$-pseudo-orbit
$(x_i,t_i)_{i\in\mathbb Z}$ such that $(x_i,t_i)=(\sigma,t_i)$ in both
tails and $x_{i_1}=p$ for some $i_1$.

There exist $h\in\mathrm{Rep}$ and $y_1\in M$ such that
\[
d\bigl(\phi_{t-S_i}(x_i),\phi_{h(t)}(y_1)\bigr)
\leq\varepsilon
\]
for every $i\in\mathbb Z$ and every $t\in[S_i,S_{i+1}]$. Hence
\[
\alpha(y_1)\cup\omega(y_1)
\subset\overline{B_{\varepsilon}(\sigma)}\subset U,
\]
and the orbit of $y_1$ meets $B_{\varepsilon}(p)$.

 Applying the Poincar\'e--Bendixson theorem again, each of the limit sets of $y_1$ is either
$\{\sigma\}$, a periodic orbit, or a nondegenerate singular graph. If a periodic
orbit $\gamma\subset U$ occurs, then it is a Jordan curve and its
interior contains a singularity by
\cite[Corollary~4 in Section~10.6]{HSD13}. Hence its interior contains
$\sigma$, and the barrier claim, with $q=\sigma$, excludes $\gamma$.
If a non-trivial singular graph occurs, then, since $\sigma$ is the
only singularity in $U$, it contains a homoclinic loop based at
$\sigma$, which was also excluded above. Therefore
\[
\alpha(y_1)=\omega(y_1)=\{\sigma\}.
\]
In particular, $y_1\in H(\sigma)$, and its orbit leaves $U$ because it
meets $B_{\varepsilon}(p)$.

We next produce two further homoclinic orbits. Choose a regular point
of $\mathcal O(y_1)\setminus\overline U$ and a flow box $W$ around it
whose closure meets $\mathcal O(y_1)$ in exactly one orbit arc. Fix a
smaller central flow box $W_0$ with $\overline{W_0}\subset W$. Choose
$\varepsilon_1>0$ sufficiently small that an
$\varepsilon_1$-tracing orbit enters $W_0$ only while tracing this
distinguished arc, and let $\delta_1>0$ be the corresponding oriented
shadowing constant. Since
$\alpha(y_1)=\omega(y_1)=\{\sigma\}$, choose $t^-<0<t^+$, with
$t^+-t^->1$, so that
\[
d(\phi_t(y_1),\sigma)<\frac{\delta_1}{2}
\qquad\text{for every }t\notin[t^-,t^+],
\]
and so that $\phi_{[t^-,t^+]}(y_1)$ crosses $W_0$ exactly once.

Form two pseudo-orbits which are constant at $\sigma$ in both tails
and follow the arc $\phi_{[t^-,t^+]}(y_1)$ twice and three times,
respectively. All their jumps are smaller than $\delta_1$ by definition, so there are points $y_2$
and $y_3$ which $\varepsilon_1$-trace them, respectively. Their alpha- and
omega-limit sets are contained in $U$, so the preceding
Poincar\'e--Bendixson and barrier argument gives
\[
\alpha(y_i)=\omega(y_i)=\{\sigma\},
\qquad i=2,3.
\]
The orbits of $y_1,y_2,y_3$ cross $W_0$ once, twice, and three times,
respectively. They are therefore pairwise distinct.

Choose a smaller Jordan domain $V$ containing $\sigma$, adapted to
the embedded six-star formed by the incoming and outgoing tails, such
that, for each $i=1,2,3$, the intersection $\mathcal O(y_i)\cap V$
consists of exactly two connected tails. Let $a_i<b_i$ be the exit and entrance
times of these tails through $\partial V$, and put
\[
A_i:=\phi_{(-\infty,a_i]}(y_i)\cup\{\sigma\},
\qquad
B_i:=\phi_{[b_i,\infty)}(y_i)\cup\{\sigma\}.
\]
The six arcs are pairwise disjoint except for their common endpoint
$\sigma$. For every $i,j$, the set
\[
Q_{i,j}:=A_i\cup B_j\cup\partial V
\]
is a $\Theta$-curve. Hence $V\setminus Q_{i,j}$ has two connected
components, each homeomorphic to an open disc.

\begin{figure}[t]
	\centering
	\begin{tikzpicture}[
		scale=1.05,
		every node/.style={font=\small},
		flow/.style={
			line width=.9pt,
			postaction={decorate},
			decoration={markings,
				mark=at position .58 with {\arrow{Stealth[length=2.1mm]}}
			}
		},
		Abranch/.style={flow,blue!65!black},
		Bbranch/.style={flow,orange!80!black},
		barrier/.style={flow,line width=1.5pt,black},
		jump/.style={
			densely dashed,violet!80!black,
			-{Stealth[length=2mm]},line width=.8pt
		}
		]
		
		\def\R{3}
		\coordinate (S)  at (0,0);
		\coordinate (A1) at (150:\R);
		\coordinate (A2) at (90:\R);
		\coordinate (B2) at (30:\R);
		\coordinate (B1) at (-30:\R);
		\coordinate (A3) at (-90:\R);
		\coordinate (B3) at (-150:\R);
		
		\fill[blue!7]
		(S)--(A1)
		arc[start angle=150,end angle=-30,radius=\R]--cycle;
		\fill[orange!9]
		(S)--(B1)
		arc[start angle=-30,end angle=-210,radius=\R]--cycle;
		
		\draw[line width=1.35pt] (S) circle[radius=\R];
		\draw[barrier] (S)--(A1)
		node[pos=.73,above left=-40pt] {$A_3$};
		\draw[barrier] (B1)--(S)
		node[pos=.28,below right=-40pt] {$B_1$};
		
		\draw[Abranch] (S)--(A2)
		node[pos=.78,above left=-30pt] {$A_2$};
		\draw[Bbranch] (B2)--(S)
		node[pos=.22,above right=-45pt] {$B_2$};
		\draw[Abranch] (S)--(A3)
		node[pos=.78,right] {$A_1$};
		\draw[Bbranch] (B3)--(S)
		node[pos=.22,above left=-20pt] {$B_3$};
		
		\fill (90:1.22) circle (1.5pt)
		node[left=2pt] {$u$};
		\fill (-150:1.72) circle (1.5pt)
		node[below right=1pt] {$v$};
		
		\draw[jump] (-150:.42)--(90:.42);
		
		\fill (S) circle (2pt)
		node[above right=1pt] {$\sigma$};
		
		\node at (1.35,1.05) {$C_+$};
		\node at (-1.35,-0.5) {$C_-$};
		\node[fill=white,inner sep=1.5pt] at (19:3.32)
		{$\partial V$};
		
	\end{tikzpicture}
	\caption{The local configuration inside $V$, after relabelling.
		The $\Theta$-curve 
		$Q_{3,1}=A_3\cup B_1\cup\partial V$  separate $V$ into the
		components $C_+$ and $C_-$. The points $u\in A_2$ and $v\in B_3$
		belong to different components. Arrows indicate the direction of
		the flow, while the dashed arrow represents the pseudo-orbit jump
		near $\sigma$.}
	\label{fig:six-star}
\end{figure}
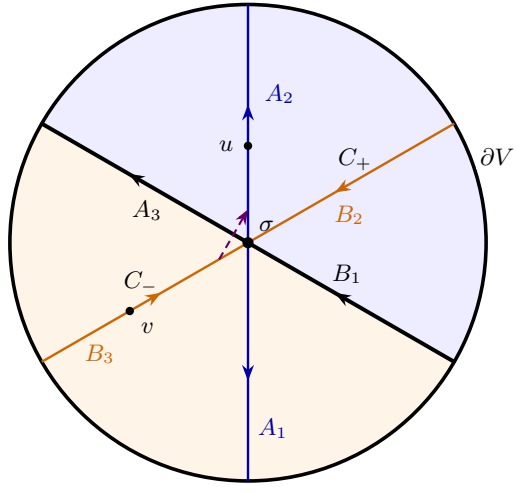

 Intersect the six arcs with the boundary of a sufficiently
small disk around $\sigma$ and consider their cyclic order. One can
choose an outgoing arc $A_i$ and an incoming arc $B_j$ so that each of
the two complementary boundary intervals contains at least one of the
remaining outgoing arcs and at least one of the remaining incoming
arcs. Consequently, there are $k\neq i$ and $s\neq j$ such that
$A_k\setminus\{\sigma\}$ and $B_s\setminus\{\sigma\}$ lie in distinct
components of $V\setminus Q_{i,j}$.

Choose points $u\in A_k\setminus\{\sigma\}$ and
$v\in B_s\setminus\{\sigma\}$ in these two components. Fix a tracing
accuracy $\varepsilon_2>0$ so small that the
$\varepsilon_2$-neighborhoods of $u$ and $v$ remain in their
respective components and the compact pieces of $A_k$ and $B_s$
between these points and $\sigma$ have their
$\varepsilon_2$-neighborhoods contained in $V$. Let $\delta_2>0$ be
the corresponding oriented shadowing constant.

Follow the orbit of $y_s$ along its incoming branch until a point
sufficiently close to $\sigma$, make a jump of size less than
$\delta_2$ to a sufficiently far negative point of the outgoing
branch of $y_k$, and then follow the orbit of $y_k$. Subdividing the
two orbit pieces into steps of length at least one gives a
$\delta_2$-pseudo-orbit. Any orbit which $\varepsilon_2$-traces it
must pass inside $V$ from the component containing $v$ to the
component containing $u$, and therefore must meet $Q_{i,j}$. It cannot
meet $\partial V$ by the choice of $\varepsilon_2$, cannot pass
through the fixed point $\sigma$, and cannot meet $A_i$ or $B_j$.
Indeed, uniqueness of flow orbits would then make it an orbit of $y_i$
or $y_j$. After crossing the outgoing tail $A_i$, the orbit of $y_i$
leaves $V$ before returning through $B_i$ and similarly, before crossing
the incoming tail $B_j$, the orbit of $y_j$ lies outside $V$. Either
possibility contradicts the fact that the whole tracing segment
between the visits near $v$ and $u$ remains in $V$. This contradiction proves that
$H(\sigma)=\{\sigma\}$.

Thus every singular chain-recurrent class is a singleton, while every
non-singular chain-recurrent class is minimal. Therefore every
chain-recurrent class of $\phi$ is minimal.
  \end{proof}

\section{Questions and further directions}\label{sec:questions}
We conclude this work by outlining a number of questions that we believe constitute natural directions for future research.

 Theorem~\ref{thm:denseflows} gives a dense set of surface vector fields with the oriented shadowing property, but lacking the standard shadowing property. Murakami established the equivalence between the two shadowing notions under the assumption that the nonwandering set contains only finitely many critical elements \cite{Mu1}. Moreover, Theorem~\ref{thm:surfacecharac} shows that, for surface flows with oriented shadowing, the mere finiteness of $\Sing(\phi)$ already forces every chain-recurrent class to be minimal. This suggests that such a finiteness condition might be sufficient for a full equivalence, prompting the following question.

\begin{question}
Let $\phi$ be a continuous flow on a compact boundaryless surface with $\Sing(\phi)$ finite. Is it true that $\phi$ has the oriented shadowing property if and only if it has the standard shadowing property.
\end{question}

The proof of Theorem~\ref{thm:denseflows} relies crucially on the classical $C^r$-density of Morse--Smale vector fields on orientable surfaces for every $r \ge 1$ \cite[Theorem~2.6]{Palis}, together with the structural stability of such flows, which in particular implies standard shadowing. This density result, however, is known to fail in higher dimensions. Consequently, establishing an analogue of Theorem~\ref{thm:denseflows} in dimension at least three would require genuinely new mechanisms, as formalised in the following problem.

\begin{question}
Is it possible to prove an analogue of Theorem~\ref{thm:denseflows} for manifolds of dimension at least three?
\end{question}

It is worth noting that Theorem~\ref{thm:positive entropy} only invokes standard shadowing in a neighbourhood of the invariant set $\Lambda$. By Corollary~\ref{cor:locshad}, this hypothesis is automatically satisfied for compact invariant sets that avoid singularities. Thus, the only obstruction arises for chain-transitive sets that contain singularities, precisely the setting where Theorem~\ref{thm:denseflows} shows that the two shadowing notions may differ.

\begin{question}\label{q:local-oriented}
Does Theorem~\ref{thm:positive entropy} remain valid under the local \emph{oriented} shadowing property?
\end{question}

Homoclinic loops and singular suspensions are examples of irreducible almost heteroclinic sets, and they may carry zero, positive, or even infinite entropy \cites{SYZ09,SZ11}. This raises the natural question of whether a more refined characterisation of such sets can be obtained within the class of flows having the standard shadowing property.

\begin{question}\label{q:realization}
Which irreducible almost heteroclinic sets can arise as chain-recurrent classes of a flow with the standard shadowing property? In particular, is it true that every such class of a zero-entropy flow is either minimal or singularly minimal?
\end{question}

\section*{Acknowledgements}
This work was partially supported by the project No.~CZ.02.01.01/00/23\_021/0008759 supported by EU funds through the Operational Programme Johannes Amos Comenius.
Part of this project was conducted while Sakshi Jain was affiliated with Monash University whose support is gratefully acknowledged.
\begin{table}[h]
\begin{tabularx}{\linewidth}{p{1.5cm}  X}
\includegraphics [width=1.8cm]{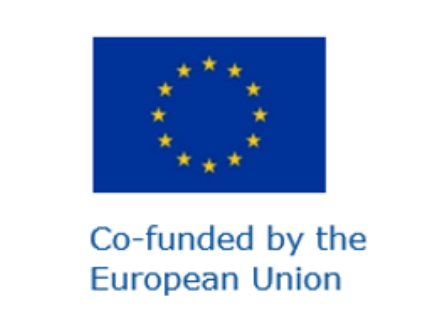} &
\vspace{-1.5cm}

 This research is part of a project that has received funding from
the European Union's European Research Council Marie Sklodowska-Curie Project No. 101151716 -- TMSHADS -- HORIZON--MSCA--2023--PF--01. For the purpose of open access, and in fulfilment of the obligations arising from the grant agreement, the author has applied a Creative Commons Attribution 4.0 International (CC BY 4.0) license to any Author Accepted Manuscript version arising from
this submission.
 \\
\end{tabularx}
\end{table}

\nocite{*}
\bibliographystyle{amsplain}
 \bibliography{biblio}
\end{document}